\documentclass{article}
\usepackage{lmodern}
\usepackage{bm}
\usepackage{multicol}
\usepackage{multirow}
\usepackage{scrextend}
\usepackage{amsmath}
\usepackage[titletoc,title]{appendix}
\usepackage{amsbsy}
\usepackage{amssymb}
\usepackage{amsthm}
\usepackage{amsfonts}
\usepackage{xcolor}
\usepackage{anyfontsize}
\usepackage{epsfig}
\usepackage{soul}
\usepackage{epstopdf}
\usepackage{wrapfig}
\usepackage[normalem]{ulem}
\usepackage{cancel}
\usepackage{fullpage}
\usepackage{float}
\usepackage{breqn}
\usepackage{graphicx}
\usepackage{caption}
\usepackage{subcaption}
\usepackage{calrsfs}
\usepackage{verbatim}
\usepackage{lineno}
\usepackage[hidelinks]{hyperref}
\usepackage{cleveref}
\usepackage{url}
\usepackage{enumitem}

\pdfoutput=1

\DeclareMathAlphabet{\mathpzc}{OT1}{pzc}{m}{it}

\newtheorem{theorem}{Theorem}[section]
\newtheorem{algorithm}{Algorithm}[section]
\newtheorem{corollary}{Corollary}[section]
\newtheorem{lemma}{Lemma}[section]
\newtheorem{proposition}{Proposition}[section]

\newtheorem{assumption}{Assumption}
\newtheorem{definition}{Definition}[section]
\newtheorem{remark}{Remark}[section]

\numberwithin{equation}{section}

\newcommand{\bu}{\mathbf{u}}
\newcommand{\bv}{\mathbf{v}}
\newcommand{\bU}{\mathbf{U}}

\newcommand{\bX}{\mathbf{X}}
\newcommand{\bw}{\mathbf{w}}

\newcommand{\be}{\mathbf{e}}
\newcommand{\bff}{\mathbf{f}}

\newcommand{\bphi}{{\boldsymbol \phi}}

\newcommand{\bfX}{\mathbf{X}}

\newcommand{\bfu}{\mathbf{u}}

\newcommand{\bfeta}{{\boldsymbol \eta}}
\newcommand{\btheta}{{\boldsymbol \theta}}

\date{}
\usepackage[section]{placeins}
\makeatletter
\AtBeginDocument{%
	\expandafter\renewcommand\expandafter\subsection\expandafter{%
		\expandafter\@fb@secFB\subsection
	}%
}

\usepackage{authblk}

\title{Accelerating the Improved Arrow--Hurwicz Iteration via the Anderson Algorithm for Steady-State Navier--Stokes Equations}

\author[, $\chi$, $\mu$]{Sinan Ergen\thanks{Corresponding author: sinan.ergen@balikesir.edu.tr}}
\author[$\chi$, $\alpha$]{Mustafa Ağgül}
\author[$\chi$]{Mustafa Türkyılmazoğlu}

\affil[$\chi$]{Department of Mathematics, Hacettepe University, 06800, Ankara, Türkiye}
\affil[$\mu$]{Department of Mathematics, Balikesir University, 10145, Balikesir, Türkiye}
\affil[$\alpha$]{Department of Mathematics, Southern Methodist University, 75205, Dallas, TX, USA}

\begin{document}

\maketitle

\begin{abstract}
We apply Anderson acceleration to the improved Arrow--Hurwicz (IAH) method
for the finite element solution of the steady-state incompressible
Navier--Stokes equations. The IAH scheme avoids saddle-point solves by
decoupling the velocity and pressure updates, but can require prohibitively
many iterations, particularly at high Reynolds numbers. To place the
acceleration on a rigorous footing, we reformulate the IAH iteration as a
nonlinear fixed-point operator $G$ for the grad-div augmented discrete
formulation induced by the scheme and establish its well-definedness,
Lipschitz continuity, and Fr\'{e}chet differentiability, thereby verifying
the required smoothness conditions locally near the fixed point.
Numerical experiments on problems with known analytical solutions,
lid-driven cavity flow up to $Re = 15{,}000$, and channel flow over a full 
step demonstrate that the resulting Anderson-accelerated improved
Arrow--Hurwicz algorithm substantially reduces iteration counts and CPU
time while retaining the reported manufactured-solution convergence rates
and centerline-velocity agreement.
\end{abstract}

\section{Introduction}

The aim of this study is to apply the Anderson algorithm to the improved Arrow--Hurwicz scheme proposed in~\cite{IAH_for_NSE} for the steady-state Navier--Stokes equations. We consider the steady incompressible Navier--Stokes equations given by
\begin{equation}\label{nse}
    \begin{aligned}
        \bu \cdot \nabla \bu - \nu \Delta \bu + \nabla p &= \bff, \text{ in } \Omega, \\
        \nabla \cdot \bu &= 0, \text{ in } \Omega,\\
        \bu &= 0, \text{ on } \partial\Omega.
    \end{aligned}
\end{equation}
Here, $\bu$ denotes the velocity vector, $p$ the pressure, $\nu > 0$ the kinematic viscosity coefficient, and $\bff$ the external force. The domain is assumed to satisfy $\Omega \subset \mathbb{R}^d$ with $d = 2,3$ and to have a Lipschitz boundary $\partial\Omega$. A well-established result states that the system~\eqref{nse} always possesses at least one solution, while the uniqueness of the solution is guaranteed under a small data condition~\cite{Temam79}

\begin{equation}\label{small_data_condition}
    \Lambda = \nu^{-2} \mathcal{M} \|\bff\|_{-1} < 1
\end{equation}

where
\begin{equation} \label{M}
    \mathcal{M} = \sup_{\bfu, \bv, \bw \in [H_0^1(\Omega)]^d} \frac{|(\bfu \cdot \nabla \bv, \bw)|}{\| \nabla \bfu \| \| \nabla \bv \| \| \nabla \bw \|}.
\end{equation}
Steady-state incompressible Navier--Stokes equations serve as a fundamental framework for modeling numerous physical problems in fluid mechanics; for instance, certain hemodynamic flows occur in steady regimes~\cite{viguerie2019deconvolution}. Consequently, the solution of these equations occupies a significant place in the mathematical mo\-deling literature~\cite{MarcusTuckerman1987a,MarcusTuckerman1987b}. Numerical solution of these equations using the finite element method presents significant computational difficulties due to the nonlinear nature created by the convection term and the saddle-point relationship between velocity and pressure variables~\cite{benzi2005numerical}. Although techniques such as Newton or Picard iterations are frequently employed in the literature to address these difficulties and resolve the nonlinear structure~\cite{he2009convergence}, these methods necessitate the solution of large, coupled linear systems at each iteration step. This requirement significantly increases computational costs and memory demands, particularly when dealing with fine meshes or complex three-dimensional problems~\cite{cao2015relaxed}.

To decrease these computational costs, various decoupling methods have been developed to separate the velocity and pressure computations, thereby generating smaller sub-problems~\cite{GUERMOND20066011}. A particularly notable approach among these is the Arrow--Hurwicz (AH) method. While originally developed by Arrow and Hurwicz for constrained optimization~\cite{arrow1958gradient}, it was later tailored to the Navier--Stokes equations by Temam~\cite{Temam79}. The traditional AH method employs a sequential iterative process to update velocity and pressure; however, this approach often suffers from slow convergence, particularly at high Reynolds numbers or when the iteration parameters ($\alpha$ and $\rho$) are not chosen carefully~\cite{CHEN2017100}. Consequently, reaching a solution can require a prohibitive number of iterations. To counter these efficiency issues and improve stability, researchers have introduced various enhancements, such as grad-div stabilization~\cite{olshanskii2004grad} and the application of preconditioners~\cite{PAN2006762newpre}.

Recently, a new scheme termed the ``improved Arrow--Hurwicz'' method has been proposed for the steady-state Navier--Stokes equations~\cite{IAH_for_NSE} and subsequently extended to natural convection equations~\cite{AH_NCE}. Compared to the classical AH scheme, the IAH method produces more stable and efficient results. Nevertheless, the iteration counts remain open to improvement, particularly at high Reynolds numbers.

On the other hand, the Anderson acceleration (AA) technique has garnered significant attention in recent years for its ability to speed up fixed-point iterations~\cite{anderson1965iterative}. Originally proposed in 1965, this method updates the next step by utilizing ``history'' information from previous iterations within a least-squares optimization framework. Anderson acceleration has been shown to significantly accelerate convergence in Picard iterations for Navier--Stokes equations~\cite{AA_for_picard,REBHOLZ2021114178}, fixed-point iterations~\cite{AA_for_fixed}, grad-div stabilized methods~\cite{GEREDELI2023114920}, and other nonlinear systems~\cite{LOTT201292}. In some instances, it can even induce convergence in challenging problems where classical Newton or Picard methods fail or diverge~\cite{convergence_analysis_aa}.

In this study, Anderson acceleration is applied to the IAH method proposed in~\cite{IAH_for_NSE} for the solution of steady incompressible Navier--Stokes equations. Our primary goal is to minimize the number of iterations and computational time via Anderson acceleration, while preserving the computational efficiency derived from the decoupling nature of the IAH method without compromising solution accuracy. To establish the theoretical foundation for the applicability of Anderson acceleration to this system, the IAH algorithm from~\cite{IAH_for_NSE} is explicitly reformulated as a nonlinear fixed-point operator on $\mathcal{H}_h=\bX_h\times Q_h$ of the form $x=G(x)$ for the grad-div augmented fixed-point formulation induced by the scheme. The well-posedness, Lipschitz continuity, and Fréchet differentiability of the defined operator $G$ are then rigorously established.

The rest of the article is structured as follows: Section 2 presents the basic representations and weak formulation of the Navier--Stokes equations and the improved Arrow--Hurwicz scheme. Section 3 introduces the Anderson algorithm; subsequently, the scheme is expressed as the fixed-point operator $(G)$ and its differen\-tiability properties are analyzed. Section 4 presents numerical experiments demonstrating the performance of the proposed method.

\section{Preliminaries}

Let $\Omega \subset \mathbb{R}^d$ $(d = 2,3)$ be a domain with a Lipschitz boundary $\partial\Omega$. The zero mean subspace of $L^2(\Omega)$ is denoted by $L^2_0(\Omega)$. We define the natural NSE velocity and pressure spaces as $\bX := [H_0^1(\Omega)]^d$ and $Q := L_0^2(\Omega)$, respectively. The dual space of $\bX$ will be denoted by $\bX'$, and its dual norm by $\|\cdot\|_{-1}$.

The skew-symmetric trilinear form is defined as follows

\begin{equation}
    b^*(\bfu, \bv, \bw) = \frac{1}{2}((\bfu \cdot \nabla)\bv, \bw) - \frac{1}{2}((\bfu \cdot \nabla)\bw, \bv), \quad \forall \bfu, \bv, \bw \in \bfX.
\end{equation}

For any given velocity fields $\bfu, \bv, \bw \in \bfX$, the trilinear form satisfies the well-known continuity bound~\cite{girault1986finite, Temam79}
\begin{equation}
    |b^*(\bfu, \bv, \bw)| \le \mathcal{M} \|\nabla \bfu\| \|\nabla \bv\| \|\nabla \bw\|, \label{eq:b_bound}
\end{equation}
where $\mathcal{M}$ is defined in~\eqref{M}. Based on this, the standard continuous weak formulation for the system~\eqref{nse} is defined as seeking a velocity-pressure pair $(\bfu, p) \in \bfX \times Q$ that satisfies the following equations for any test functions $(\bv, q) \in \bfX \times Q$:
\begin{align}
    \nu(\nabla \bfu, \nabla \bv) + b^*(\bfu, \bfu, \bv) - (p, \nabla \cdot \bv) & = (\bff, \bv), \label{eq:weak_form_1} \\
    (\nabla \cdot \bfu, q) & = 0. \label{eq:weak_form_2}
\end{align}

To establish the discrete setting, we introduce a regular and conforming triangulation $\mathcal{T}_h$ over the domain $\Omega$, where $h$ represents the maximum element diameter ($h = \max_{K \in \mathcal{T}_h} h_K$). The corresponding Galerkin finite element approximation requires finding $(\bfu_h, p_h) \in \bX_h \times Q_h$ such that for all discrete test functions $(\bv_h, q_h) \in \bX_h \times Q_h$,
\begin{align}
    \nu(\nabla \bfu_h, \nabla \bv_h) + b^*(\bfu_h, \bfu_h, \bv_h) - (p_h, \nabla \cdot \bv_h) & = (\bff, \bv_h), \label{eq:discrete_1} \\
    (\nabla \cdot \bfu_h, q_h) & = 0. \label{eq:discrete_2}
\end{align}
In this coupled formulation, the finite-dimensional subspaces $\bX_h \subset \bfX$ and $Q_h \subset Q$ must be chosen to satisfy the discrete Ladyzhenskaya-Babuška-Brezzi (LBB) compatibility condition [see, e.g.,~\cite{girault1986finite},~\cite{john2016finite}]. Specifically, there must exist a mesh-independent constant $\beta > 0$ ensuring
\begin{equation}
    \inf_{q_h \in Q_h} \sup_{\bv_h \in \bX_h} \frac{(q_h, \nabla \cdot \bv_h)}{\|\nabla \bv_h\| \|q_h\|} \ge \beta > 0. \label{eq:inf_sup}
\end{equation}

The fulfillment of the discrete inf-sup condition~\eqref{eq:inf_sup}, together with the small data assumption~\eqref{small_data_condition}, mathematically guarantees that the discrete nonlinear system~\eqref{eq:discrete_1}--\eqref{eq:discrete_2} admits a unique solution. Further\-more, the discrete velocity field $\bfu_h \in \bX_h$ satisfies the following standard a priori stability bound (see~\cite{john2016finite, Temam79, layton2008introduction}  for detailed derivations)
\begin{equation}
    \|\nabla \bfu_h\| \le \nu^{-1} \|\bff\|_{-1}. \label{eq:stability}
\end{equation}

Since the improved Arrow--Hurwicz iteration updates both velocity and pressure, we denote the product space and its norm by
\[
\mathcal{H}_h := \bX_h \times Q_h,
\qquad
\|(\bv_h,q_h)\|_{\mathcal{H}_h}
:=
\left(\|\nabla \bv_h\|^2+\alpha\|q_h\|^2\right)^{1/2}.
\]
For notational convenience, we write each velocity-pressure pair as
\[
x_k := (\bfu_h^k,p_h^k)\in \mathcal{H}_h.
\]

\section{Improved Arrow--Hurwicz Scheme and Anderson Acceleration}

In this section, an Anderson acceleration algorithm (AA-IAH) is introduced for the improved Arrow--Hurwicz scheme of the steady-state Navier--Stokes equations~\cite{IAH_for_NSE}. The fundamental theoretical analyses of the IAH scheme under appropriate parameter choices and the small data condition, including uniform boundedness of iterates, convergence properties, and error estimates, have been rigorously established in~\cite{IAH_for_NSE}. To avoid redundancy, we present the method in the form used in this work and refer the interested reader to~\cite{IAH_for_NSE} for detailed theoretical results.

Unlike~\cite{IAH_for_NSE}, we initialize the solution $(\bfu_h^0, p_h^0) \in \bX_h \times Q_h$ directly as $(\bfu_h^0, p_h^0) = (\mathbf{0}, 0)$, rather than utilizing a solution obtained from a preliminary saddle-point problem. This approach is motivated by the objective of reducing the overall computational overhead, particularly by avoiding an expensive coupled solve at the initial step. Such a zero-initialization strategy is consistent with modern Anderson acceleration frameworks for Navier--Stokes equations, which demonstrate that the acceleration mechanism can effectively compensate for the absence of sophisticated initial guesses or continuation methods~\cite{AA_for_picard,contractive_noncontractive}.
We now state the IAH scheme as defined in~\cite{IAH_for_NSE}.

\begin{algorithm}[Improved Arrow--Hurwicz Method] \label{algo_IAH}

 Let $\rho$ and $\alpha$ be positive parameters, and initialize $(\bfu_h^0, p_h^0) = (\mathbf{0}, 0)$. For $n \ge 0$, find $(\bfu_h^{n+1}, p_h^{n+1}) \in \bX_h \times Q_h$ such that, for all $(\bv_h, q_h) \in \bX_h \times Q_h$,
\begin{align}
\frac{1}{\rho}(\nabla(\bfu_h^{n+1} - \bfu_h^n), \nabla \bv_h)
+ \nu(\nabla \bfu_h^{n+1}, \nabla \bv_h)
+ b^*(\bfu_h^n, \bfu_h^{n+1}, \bv_h) \nonumber \\
- (p_h^n, \nabla \cdot \bv_h)
+ \frac{\rho}{\alpha}(\nabla \cdot \bfu_h^{n+1}, \nabla \cdot \bv_h)
& = (\bff, \bv_h), \label{eq:weak_mom} \\[6pt]
\alpha(p_h^{n+1} - p_h^n, q_h) + \rho(\nabla \cdot \bfu_h^{n+1}, q_h) & = 0. \label{eq:weak_cont}
\end{align}
\end{algorithm}

We now introduce the Anderson acceleration (AA) procedure and its convergence properties.
Let $G: \mathcal{H}_h \to \mathcal{H}_h$ denote the fixed-point operator associated with Algorithm~\ref{algo_IAH}.
Then, the fixed-point problem is
\[
G(x)=x.
\]
We define the nonlinear residual, also referred to as the update step, by
\[
r_j := G(x_{j-1})-x_{j-1}.
\]
The AA algorithm with depth $m$, which reduces to the standard fixed-point iteration when $m=0$, applied to the fixed-point problem $G(x)=x$ is presented next.

\begin{algorithm}[Anderson Acceleration]
\label{algo:anderson}

\noindent The Anderson acceleration algorithm with depth $m$ is as follows:
\begin{description}
\item[{Step $0$\quad\quad:}]
Choose $x_0 \in \mathcal{H}_h$.
\item[{Step $1$\quad\quad:}]
Find $\widetilde{x}_1 \in \mathcal{H}_h$ such that $\widetilde{x}_1 = G(x_0)$. Set $x_1 = \widetilde{x}_1$.
\item[{Step $k+1$\,\,:}]
For $k = 1, 2, 3, \dots$, set $m_k = \min\{k, m\}$.
\begin{description}
\item[{\textbf{(a)}}] Find $\widetilde{x}_{k+1} = G(x_k)$.
\item[{\textbf{(b)}}] Solve the minimization problem for
$\{\alpha_j^{k+1}\}_{j=k-m_k}^k$:
\begin{equation*}
\min_{\alpha^{k+1}}
\left\|
\sum_{j = k-m_k}^k
\alpha_j^{k+1}(\widetilde{x}_{j+1}-x_j)
\right\|_{\mathcal{H}_h}
\quad
\text{subject to}
\quad
\sum_{j = k-m_k}^k \alpha_j^{k+1} = 1.
\end{equation*}

\item[{\textbf{(c)}}] Set
\[
x_{k+1}
=
\sum_{j = k-m_k}^k
\alpha_j^{k+1}\widetilde{x}_{j+1}.
\]

\end{description}
\end{description}
\end{algorithm}

\begin{remark}
The Anderson acceleration algorithm is stated in the $\mathcal H_h$ norm because
this is the natural product norm for the theoretical analysis of the IAH
fixed-point map. In the numerical implementation, however, the least-squares
problem is solved with the standard discrete Euclidean norm of the coefficient
vectors. This is the classical choice in many computational implementations of
Anderson acceleration: it avoids assembling and applying additional weighted
Gram matrices at every nonlinear iteration and keeps the acceleration step small
relative to the finite element solves. Since the underlying spaces are finite
dimensional on each fixed mesh, the discrete Euclidean norm and the finite
element product norms are equivalent, although the equivalence constants
generally depend on the basis and mesh and are not used in the theoretical
estimates below. Thus
the $\mathcal H_h$ norm should be understood as the analytical norm used to
derive local smoothness and residual bounds, while the reported computations
use the classical algebraic norm for the Anderson least-squares problem and for
the recorded relative iterate changes.
\end{remark}

To understand how Anderson acceleration improves convergence, for $\alpha^{k+1}$ from Algorithm~\ref{algo:anderson} we define the optimization gain factor $\theta_k$ by

\begin{equation}
    \min_{\alpha^{k+1}} \left\| \sum_{j = k-m_k}^k \alpha_j^{k+1} (\widetilde{x}_{j+1} - x_j) \right\|_{\mathcal{H}_h} = \theta_k \left\| \widetilde{x}_{k+1} - x_k \right\|_{\mathcal{H}_h}.
\end{equation}

\noindent The following assumptions, which are verified later in this section, provide the fundamental conditions used in the theoretical analysis of Anderson acceleration applied to the fixed-point operator $G$, similar to the framework in~\cite{contractive_noncontractive}.

\begin{assumption}\label{assump:1}

Assume that $G \in C^1(\mathcal{H}_h)$ has a fixed point $x_h^*$ in $\mathcal{H}_h$, and that there is a neighborhood $\mathcal U$ of $x_h^*$ on which positive constants $C_0$ and $C_1$ satisfy the following conditions:
\begin{enumerate}
    \item For all $x\in \mathcal U$ and all $y \in \mathcal{H}_h$, $\|G'(x)y\|_{\mathcal{H}_h} \leq C_0\|y\|_{\mathcal{H}_h}$,
    \item For all $x,\widetilde{x}\in \mathcal U$ and all $y \in \mathcal{H}_h$, $\|G'(x)y - G'(\widetilde{x})y\|_{\mathcal{H}_h} \leq C_1\|x - \widetilde{x}\|_{\mathcal{H}_h} \|y\|_{\mathcal{H}_h}$.
\end{enumerate}

\end{assumption}

\begin{assumption}\label{assump:2}

Assume there exists a constant $\sigma > 0$ for which the differences between consecutive residuals and iterates satisfy
\begin{equation}
    \|r_{k+1} - r_k\|_{\mathcal{H}_h} \geq \sigma \|x_k - x_{k-1}\|_{\mathcal{H}_h}, \qquad k \geq 1.
    \label{eq:assump_sigma}
\end{equation}

\end{assumption}

In the remainder of this section, it will be shown that the fixed-point operator associated with the IAH method satisfies the conditions required by Assumption~\ref{assump:1}. Assumption~\ref{assump:2}, on the other hand, will be verified locally in a neighborhood of the solution. Once both assumptions are established, Theorem~\ref{theorem:pollock} from~\cite{contractive_noncontractive} allows the residual norm $\|r_{k+1}\|_{\mathcal{H}_h}$ to be bounded from above in terms of the current residual $\|r_k\|_{\mathcal{H}_h}$ and the preceding residuals.

\begin{theorem}[Pollock and Rebholz 2021] \label{theorem:pollock}
Suppose Assumptions \ref{assump:1} and \ref{assump:2} hold, and suppose the direction sines between each column $i$ of the matrix $\mathcal{F}_j$ defined by
\[
\mathcal{F}_j := ((r_j - r_{j-1}), \, (r_{j-1} - r_{j-2}) \, \dots \, (r_{j-m_j+1} - r_{j-m_j}))
\]

and the subspace spanned by the preceding columns satisfy $|\sin(\mathcal{F}_{j,i}, \mathrm{span } \{\mathcal{F}_{j,1}, \dots, \mathcal{F}_{j,i-1}\})| \ge c_s > 0$, for $j = k-m_k, \dots, k-1$. Then the residual $r_{k+1} = G(x_k) - x_k$ from Algorithm~\ref{algo:anderson} (depth $m$) satisfies the following bound:
\begin{multline}
    \|r_{k+1}\|_{\mathcal{H}_h} \le \|r_k\|_{\mathcal{H}_h} \Bigg( \theta_k C_0 + \frac{C C_1 \sqrt{1 - \theta_k^2}}{2} \bigg( \|r_k\|_{\mathcal{H}_h} h(\theta_k) \\
    + 2 \sum_{n = k-m_k+1}^{k-1} (k-n) \|r_n\|_{\mathcal{H}_h} h(\theta_n) + m_k \|r_{k-m_k}\|_{\mathcal{H}_h} h(\theta_{k-m_k}) \bigg) \Bigg),
    \label{eq:theorem_bound}
\end{multline}
where each $h(\theta_j) \le C \sqrt{1 - \theta_j^2} +  \theta_j$, and $C$ depends on $c_s$ and the implied upper bound on the direction cosines.
\end{theorem}
\begin{remark}
Theorem~\ref{theorem:pollock} provides an upper bound that describes the effect of Anderson acceleration on the residual at the next step. In this bound, the optimization gain factor $\theta_k$ directly scales the first-order term, while the higher-order terms appear together with the factor $\sqrt{1-\theta_k^2}$ and the preceding residual norms. Thus, the theorem illustrates the balance between the gain obtained at the Anderson step and the contribution of the higher-order terms.
Furthermore, the constant $C$ appearing in the theorem depends on constants related to the geometric independence among the columns of the optimization matrix~\cite{contractive_noncontractive}. Therefore, preserving the linear independence of the columns of the optimization matrix is essential for the theoretical bound to remain meaningful.
\end{remark}

In what follows, the solution operator $G$ associated with the IAH scheme will be defined, and the conditions under which the assumptions required by the theorem are satisfied will be examined.
The IAH iteration can be expressed via a fixed-point operator defined on the finite element spaces. To this end, let us define the operator $G: \mathcal{H}_h \longrightarrow \mathcal{H}_h$. For any input $(\bfu_h, p_h) \in \bX_h \times Q_h$, let $G(\bfu_h, p_h) = \bigl(G_1(\bfu_h, p_h), G_2(\bfu_h, p_h)\bigr)$, where $G_1(\bfu_h, p_h) \in \bX_h$ denotes the updated velocity component and $G_2(\bfu_h, p_h) \in Q_h$ denotes the updated pressure component.

\begin{definition}
\label{def:G_operator}
For any given $(\bfu_h, p_h) \in \bX_h \times Q_h$, the operator $G(\bfu_h, p_h) = \bigl(G_1(\bfu_h, p_h), G_2(\bfu_h, p_h)\bigr)$ is defined to satisfy the following equations for all test functions $(\bv_h, q_h) \in \bX_h \times Q_h$:
\begin{align}
&\frac{1}{\rho}\bigl(\nabla(G_1(\bfu_h, p_h) - \bfu_h), \nabla \bv_h\bigr)
+ \nu\bigl(\nabla G_1(\bfu_h, p_h), \nabla \bv_h\bigr)
+ b^*\bigl(\bfu_h, G_1(\bfu_h, p_h), \bv_h\bigr)
\nonumber \\
&\quad + \frac{\rho}{\alpha}\bigl(\nabla \cdot G_1(\bfu_h, p_h), \nabla \cdot \bv_h\bigr)
- \bigl(p_h, \nabla \cdot \bv_h\bigr)
 = (\bff, \bv_h),
\label{eq:G_momentum}
\\[6pt]
&\alpha\bigl(G_2(\bfu_h, p_h) - p_h,\, q_h\bigr)
+ \rho\bigl(\nabla \cdot G_1(\bfu_h, p_h),\, q_h\bigr)
 = 0.
\label{eq:G_mass}
\end{align}
Here $\rho > 0$ and $\alpha > 0$ are algorithmic parameters.
\end{definition}

\begin{remark}\label{rem:fixed_point_augmented}
A fixed point of $G$ satisfies the grad-div augmented discrete Navier--Stokes formulation
\begin{align*}
\nu(\nabla \bfu_h,\nabla \bv_h)
+ b^*(\bfu_h,\bfu_h,\bv_h)
+\frac{\rho}{\alpha}(\nabla\cdot \bfu_h,\nabla\cdot \bv_h)
-(p_h,\nabla\cdot \bv_h) &= (\bff,\bv_h), \\
(\nabla\cdot \bfu_h,q_h) &=0.
\end{align*}
This is the fixed-point problem induced by the IAH scheme. If the grad-div term vanishes on discretely divergence-free velocities, for example under a divergence-compatible discretization, this formulation coincides with the standard discrete NSE system~\eqref{eq:discrete_1}--\eqref{eq:discrete_2}; otherwise it is interpreted as the grad-div stabilized discrete formulation. In either case, testing the fixed-point equations with $\bv_h=\bfu_h$ and $q_h=p_h$, and using the skew-symmetry of $b^*$, gives
\[
\nu\|\nabla\bfu_h\|^2+\frac{\rho}{\alpha}\|\nabla\cdot\bfu_h\|^2=(\bff,\bfu_h),
\]
and hence the velocity stability bound $\|\nabla\bfu_h\|\le \nu^{-1}\|\bff\|_{-1}$.
\end{remark}

We now show the well-definedness and boundedness of the operator $G$ with respect to the norm on $\mathcal{H}_h$
\[
\|(\bv_h, q_h)\|_{\mathcal{H}_h} := \sqrt{\|\nabla \bv_h\|^2 + \alpha\|q_h\|^2}.
\]

\begin{lemma}\label{lem:G_welldefined}
The operator $G$ is well-defined. Furthermore, for every $(\bfu_h, p_h) \in \bX_h \times Q_h$, the pair $G(\bfu_h, p_h) = \bigl(G_1(\bfu_h, p_h), G_2(\bfu_h, p_h)\bigr)$ satisfies
\[
\|G(\bfu_h, p_h)\|_{\mathcal{H}_h} \leq \|(\bfu_h, p_h)\|_{\mathcal{H}_h} + \sqrt{\frac{\rho}{\nu}}\,\|\bff\|_{-1}.
\]
\end{lemma}

\begin{proof}
Fix $(\bfu_h, p_h) \in \bX_h \times Q_h$. Set $\bU_h := G_1(\bfu_h, p_h)$ and $P_h := G_2(\bfu_h, p_h)$. Then the system~\eqref{eq:G_momentum}--\eqref{eq:G_mass} takes the form: for all $(\bv_h, q_h) \in \bX_h \times Q_h$,
\begin{align}
&\frac{1}{\rho}(\nabla \bU_h, \nabla \bv_h)
+ \nu(\nabla \bU_h, \nabla \bv_h)
+ b^*(\bfu_h, \bU_h, \bv_h)
+ \frac{\rho}{\alpha}(\nabla \cdot \bU_h, \nabla \cdot \bv_h)
\nonumber \\
&\qquad = (\bff, \bv_h) + \frac{1}{\rho}(\nabla \bfu_h, \nabla \bv_h) + (p_h, \nabla \cdot \bv_h),
\label{eq:G_linear_momentum} \\[4pt]
&\alpha(P_h, q_h) + \rho(\nabla \cdot \bU_h, q_h) = \alpha(p_h, q_h).
\label{eq:G_linear_mass}
\end{align}
To show that the system is well-defined, it suffices to verify that the corresponding homogeneous system admits only the trivial solution. Suppose $(\bU_h, P_h) \in \bX_h \times Q_h$ satisfies, for all $(\bv_h, q_h) \in \bX_h \times Q_h$,
\begin{align}
&\frac{1}{\rho}(\nabla \bU_h, \nabla \bv_h)
+ \nu(\nabla \bU_h, \nabla \bv_h)
+ b^*(\bfu_h, \bU_h, \bv_h)
+ \frac{\rho}{\alpha}(\nabla \cdot \bU_h, \nabla \cdot \bv_h) = 0,
\label{eq:G_hom_momentum} \\[4pt]
&\alpha(P_h, q_h) + \rho(\nabla \cdot \bU_h, q_h) = 0.
\label{eq:G_hom_mass}
\end{align}
Setting $\bv_h = \bU_h$ in~\eqref{eq:G_hom_momentum} and using the skew-symmetry $b^*(\bfu_h, \bU_h, \bU_h) = 0$ yields
\[
\frac{1}{\rho}\|\nabla \bU_h\|^2 + \nu\|\nabla \bU_h\|^2 + \frac{\rho}{\alpha}\|\nabla \cdot \bU_h\|^2 = 0,
\]
that is,
\[
\left(\frac{1}{\rho} + \nu\right)\|\nabla \bU_h\|^2 + \frac{\rho}{\alpha}\|\nabla \cdot \bU_h\|^2 = 0.
\]
Since $\rho, \alpha, \nu > 0$, we conclude $\|\nabla \bU_h\| = 0$, and since $\bU_h \in \bX_h \subset [H_0^1(\Omega)]^d$, it follows that $\bU_h = 0$. Setting $q_h = P_h$ in~\eqref{eq:G_hom_mass} and using $\bU_h = 0$ then gives $\alpha\|P_h\|^2 = 0$, hence $P_h = 0$. Thus the homogeneous system admits only the trivial solution, and by the fundamental theorem of finite-dimensional linear algebra, the system~\eqref{eq:G_linear_momentum}--\eqref{eq:G_linear_mass} has a unique solution for every $(\bfu_h, p_h) \in \bX_h \times Q_h$. Therefore, $G : \mathcal{H}_h \to \mathcal{H}_h$ is well-defined.

We now derive the stated bound.

Setting $\bv_h = \bU_h$ and $q_h = P_h$ in~\eqref{eq:G_momentum}--\eqref{eq:G_mass} and using $b^*(\bfu_h, \bU_h, \bU_h) = 0$, we obtain
\[
\frac{1}{\rho}(\nabla(\bU_h - \bfu_h), \nabla \bU_h)
+ \nu(\nabla \bU_h, \nabla \bU_h)
+ \frac{\rho}{\alpha}(\nabla \cdot \bU_h, \nabla \cdot \bU_h)
- (p_h, \nabla \cdot \bU_h) = (\bff, \bU_h)
\]
and
\[
\alpha(P_h - p_h, P_h) + \rho(\nabla \cdot \bU_h, P_h) = 0.
\]
Multiplying the momentum equation by $\rho$ and applying the polarization identity
\[
(a - b, a) = \frac{1}{2}\bigl(\|a\|^2 - \|b\|^2 + \|a - b\|^2\bigr)
\]
to both $(\nabla(\bU_h - \bfu_h), \nabla \bU_h)$ and $\alpha(P_h - p_h, P_h)$, then adding the resulting equations gives
\[
\begin{aligned}
&\frac{1}{2}\|\nabla \bU_h\|^2
+ \frac{\alpha}{2}\|P_h\|^2
+ \frac{1}{2}\|\nabla(\bU_h - \bfu_h)\|^2
+ \rho\nu\|\nabla \bU_h\|^2
+ \frac{\rho^2}{\alpha}\|\nabla \cdot \bU_h\|^2 \\
&\qquad + \frac{\alpha}{2}\|P_h - p_h\|^2
+ \rho(P_h - p_h, \nabla \cdot \bU_h) \\
& = \frac{1}{2}\|\nabla \bfu_h\|^2 + \frac{\alpha}{2}\|p_h\|^2 + \rho(\bff, \bU_h).
\end{aligned}
\]
Setting $a := \sqrt{\alpha}(P_h - p_h)$ and $b := \frac{\rho}{\sqrt{\alpha}}\nabla \cdot \bU_h$, the identity $\frac{1}{2}\|a\|^2 + (a, b) + \|b\|^2 = \frac{1}{2}\|a + b\|^2 + \frac{1}{2}\|b\|^2$ gives
\[
\frac{\alpha}{2}\|P_h - p_h\|^2
+ \rho(P_h - p_h, \nabla \cdot \bU_h)
 + \frac{\rho^2}{\alpha}\|\nabla \cdot \bU_h\|^2
 = \frac{\alpha}{2}\left\|P_h - p_h + \frac{\rho}{\alpha}\nabla \cdot \bU_h\right\|^2
+ \frac{\rho^2}{2\alpha}\|\nabla \cdot \bU_h\|^2 \geq 0.
\]
Dropping the non-negative terms $\frac{1}{2}\|\nabla(\bU_h - \bfu_h)\|^2$ and the expression above from the left-hand side, we obtain
\[
\frac{1}{2}\|\nabla \bU_h\|^2 + \frac{\alpha}{2}\|P_h\|^2 + \rho\nu\|\nabla \bU_h\|^2
\leq \frac{1}{2}\|\nabla \bfu_h\|^2 + \frac{\alpha}{2}\|p_h\|^2 + \rho(\bff, \bU_h).
\]
Applying the dual norm inequality and Young's inequality to the last term:
\[
|\rho(\bff, \bU_h)| \leq \rho\|\bff\|_{-1}\|\nabla \bU_h\| \leq \frac{\rho\nu}{2}\|\nabla \bU_h\|^2 + \frac{\rho}{2\nu}\|\bff\|_{-1}^2.
\]
Substituting and multiplying by $2$:
\[
(1 + \rho\nu)\|\nabla \bU_h\|^2 + \alpha\|P_h\|^2
\leq \|\nabla \bfu_h\|^2 + \alpha\|p_h\|^2 + \frac{\rho}{\nu}\|\bff\|_{-1}^2.
\]
Since $1 + \rho\nu > 1$, we conclude
\begin{equation}
\|G(\bfu_h, p_h)\|_{\mathcal{H}_h}^2 \leq \|(\bfu_h, p_h)\|_{\mathcal{H}_h}^2 + \frac{\rho}{\nu}\|\bff\|_{-1}^2.
\label{G_bound}
\end{equation}
Taking square roots and applying $\sqrt{a + b} \leq \sqrt{a} + \sqrt{b}$ yields the desired inequality.
\end{proof}

\begin{remark}
The bound established in Lemma~\ref{lem:G_welldefined} shows that the solution operator $G$ maps bounded inputs to bounded outputs. Since the constants appearing in the Lipschitz continuity, Fréchet differentiability, and Lipschitz continuity of the derivative operator depend on the input variables, it is natural to interpret these results in a local regime around the solution. To this end, when working on a sufficiently small closed and bounded neighborhood of a fixed point $x_h^*$, the finite-dimensionality of $\bX_h \times Q_h$ ensures that such a neighborhood is compact, and consequently the relevant constants can be bounded uniformly over it. Therefore, although the smoothness results that follow are stated in the general $\bX_h \times Q_h$ setting, they are to be interpreted as uniform constants on sufficiently small bounded neighborhoods around the fixed point.
\end{remark}

We now show the Lipschitz property of the solution operator $G$.

\begin{lemma} \label{lem:lipschitz}
For any pair $(\bu_h, p_h), (\bw_h, z_h) \in \bX_h\times Q_h$ with $\frac{3d}{2}\rho^2 < \alpha$, the following inequality holds:
\begin{equation}
    \|G(\bu_h, p_h) - G(\bw_h, z_h)\|_{\mathcal{H}_h} \le C_L \|(\bu_h, p_h) - (\bw_h, z_h)\|_{\mathcal{H}_h}.
\end{equation}
Here the constant $C_L$ is defined as
\begin{equation*}
    C_L = \sqrt{2}\max\left\{\sqrt{\frac{2}{K}},\frac{1}{\sqrt{K}} \left(1 + \rho \mathcal{M} \left(\|(\bw_h,z_h)\|_{\mathcal{H}_h}+\sqrt{\frac{\rho}{\nu}}\,\|\bff\|_{-1}\right)\right)\right\}.
\end{equation*} 
\end{lemma}

\begin{proof}
From Definition~\ref{def:G_operator}, we write the difference $G(\bu_h,p_h)-G(\bw_h,z_h)$. For every $\bv_h\in \bX_h$,
\begin{equation}
\begin{aligned}
&\quad (\nabla (G_1(\bu_h,p_h)-G_1(\bw_h,z_h)),\nabla \bv_h)
-(\nabla (\bu_h-\bw_h),\nabla \bv_h) \\
&\quad+\rho\nu(\nabla (G_1(\bu_h,p_h)-G_1(\bw_h,z_h)),\nabla \bv_h)
+\rho\, b^*(\bu_h,\, G_1(\bu_h,p_h)-G_1(\bw_h,z_h),\, \bv_h) \\
&\quad+\rho\, b^*(\bu_h-\bw_h,\, G_1(\bw_h,z_h),\, \bv_h)
+\frac{\rho^2}{\alpha}\big(\nabla\cdot (G_1(\bu_h,p_h)-G_1(\bw_h,z_h)),\, \nabla\cdot \bv_h\big) \\
&\quad = \rho\,(p_h-z_h,\nabla\cdot \bv_h).
\end{aligned} \label{denklem1}
\end{equation}
Moreover, for every $q_h\in Q_h$,
\begin{equation*}
\alpha\,(G_2(\bu_h,p_h)-G_2(\bw_h,z_h),q_h)-\alpha\,(p_h-z_h,q_h)+\rho\,(\nabla\cdot (G_1(\bu_h,p_h)-G_1(\bw_h,z_h)),q_h) = 0.
\end{equation*}
We now choose
\[
\bv_h = G_1(\bu_h,p_h)-G_1(\bw_h,z_h),\quad q_h = G_2(\bu_h,p_h)-G_2(\bw_h,z_h).
\]
This choice cancels the first nonlinear term in~\eqref{denklem1}. Dropping the positive terms
\[
\rho\nu\|\nabla(G_1(\bu_h,p_h)-G_1(\bw_h,z_h))\|^2, \quad
\frac{\rho^2}{\alpha}\|\nabla\cdot(G_1(\bu_h,p_h)-G_1(\bw_h,z_h))\|^2
\]
from the left-hand side and combining the remaining terms yields the following inequality:
\begin{align*}
&\quad \|\nabla (G_1(\bu_h,p_h)-G_1(\bw_h,z_h))\|^2
+\alpha\,\|G_2(\bu_h,p_h)-G_2(\bw_h,z_h)\|^2 \\
&\quad \leq
\rho\,(p_h-z_h,\nabla\cdot (G_1(\bu_h,p_h)-G_1(\bw_h,z_h))) \\
&\quad
+(\nabla (\bu_h-\bw_h),
\nabla (G_1(\bu_h,p_h)-G_1(\bw_h,z_h))) \\
&\quad
-\rho\, b^*(\bu_h-\bw_h,\,
G_1(\bw_h,z_h),\,
G_1(\bu_h,p_h)-G_1(\bw_h,z_h)) \\
&\quad
+\alpha\,(p_h-z_h,\,
G_2(\bu_h,p_h)-G_2(\bw_h,z_h)) \\
&\quad
-\rho\,(\nabla\cdot (G_1(\bu_h,p_h)-G_1(\bw_h,z_h)),\,
G_2(\bu_h,p_h)-G_2(\bw_h,z_h)).
\end{align*}
We now bound each term on the right-hand side. Applying the Cauchy--Schwarz inequality, the estimate $\|\nabla \cdot \bv_h\| \leq \sqrt{d}\,\|\nabla \bv_h\|$, and Young's inequality gives
\begin{align*}
&\quad\rho\,(p_h-z_h,\nabla\cdot (G_1(\bu_h,p_h)-G_1(\bw_h,z_h))) \\
&\quad \leq
\alpha\,\|p_h-z_h\|^2
+\frac{d\rho^2}{4\alpha}
\|\nabla (G_1(\bu_h,p_h)-G_1(\bw_h,z_h))\|^2,
\\[6pt]
&\quad(\nabla (\bu_h-\bw_h),
\nabla (G_1(\bu_h,p_h)-G_1(\bw_h,z_h))) \\
&\quad \leq
\|\nabla (\bu_h-\bw_h)\|^2
+\frac{1}{4}
\|\nabla (G_1(\bu_h,p_h)-G_1(\bw_h,z_h))\|^2.
\end{align*}
For the nonlinear term, we use the boundedness of the trilinear form~\eqref{eq:b_bound}, Young's inequality, and Lemma~\ref{lem:G_welldefined}:
\begin{align*}
&\qquad \left|
-\rho\, b^*\bigl(\bu_h-\bw_h,\,
G_1(\bw_h,z_h),\,
G_1(\bu_h,p_h)-G_1(\bw_h,z_h)\bigr)
\right|
\\
&\qquad \leq
\rho^2\,\mathcal{M}^2
\left(
\|\nabla \bw_h\|^2+\alpha\|z_h\|^2+\frac{\rho}{\nu}\|\bff\|_{-1}^2
\right)
\|\nabla(\bu_h-\bw_h)\|^2
\\
&\qquad
+\frac{1}{4}
\|\nabla(G_1(\bu_h,p_h)-G_1(\bw_h,z_h))\|^2.
\end{align*}
Finally, applying the Cauchy--Schwarz and Young inequalities gives
\begin{align*}
&\alpha\bigl(p_h-z_h,\,
G_2(\bu_h,p_h)-G_2(\bw_h,z_h)\bigr)
\leq
\alpha\|p_h-z_h\|^2
+\frac{\alpha}{4}
\|G_2(\bu_h,p_h)-G_2(\bw_h,z_h)\|^2,
\\[6pt]
&\left|
-\rho\bigl(
\nabla\cdot(G_1(\bu_h,p_h)-G_1(\bw_h,z_h)),\,
G_2(\bu_h,p_h)-G_2(\bw_h,z_h)
\bigr)
\right|
\\
&\leq
\frac{\alpha}{2}
\|G_2(\bu_h,p_h)-G_2(\bw_h,z_h)\|^2
+\frac{d\rho^2}{2\alpha}
\|\nabla(G_1(\bu_h,p_h)-G_1(\bw_h,z_h))\|^2.
\end{align*}
Combining all these bounds yields
\begin{equation*}
\begin{aligned}
&\left(\frac{1}{2}-\frac{3d\rho^2\alpha^{-1}}{4}\right)\,\|\nabla (G_1(\bu_h,p_h)-G_1(\bw_h,z_h))\|^2
+\frac{\alpha}{4}\,\|G_2(\bu_h,p_h)-G_2(\bw_h,z_h)\|^2 \\
&\le 2\alpha\,\|p_h-z_h\|^2
+\left(1+\rho^2\mathcal{M}^2\left(\|\nabla \bw_h\|^2+\alpha\|z_h\|^2+\frac{\rho}{\nu}\|\bff\|_{-1}^2\right)\right)\,\|\nabla (\bu_h-\bw_h)\|^2.
\end{aligned}
\end{equation*}
Define $K = \min\left\{\frac{1}{4},\ \left(\frac{1}{2}-\frac{3d\rho^2\alpha^{-1}}{4}\right)\right\}$. Dividing both sides by $K$, taking square roots, and applying $\sqrt{a+b} \leq \sqrt{a} + \sqrt{b}$, then taking the maximum of the coefficients of the desired norms on the right-hand side, yields the result.
\end{proof}

We now define the operator $G'$ and show that it is indeed the Fréchet derivative operator of $G$.

\begin{definition}\label{def:Frechet}
For a given $(\bu_h,p_h)\in \bX_h\times Q_h$, define the operator
$G'(\bu_h,p_h;\cdot,\cdot):\bX_h\times Q_h\to \bX_h\times Q_h$ by
$G'(\bu_h,p_h;\bw_h,s_h) := \bigl(G_1'(\bu_h,p_h;\bw_h,s_h),\,G_2'(\bu_h,p_h;\bw_h,s_h)\bigr)$
as follows. For every $(\bw_h,s_h)\in \bX_h\times Q_h$ and every
$(\bv_h,q_h)\in \bX_h\times Q_h$, the following equations hold:
\begin{align}
&\bigl(\nabla (G_1'(\bu_h,p_h;\bw_h,s_h)-\bw_h),\nabla \bv_h\bigr)
+\rho\nu\bigl(\nabla G_1'(\bu_h,p_h;\bw_h,s_h),\nabla \bv_h\bigr)
\nonumber\\
&\quad
+\rho\, b^*\bigl(\bu_h,G_1'(\bu_h,p_h;\bw_h,s_h),\bv_h\bigr)
+\rho\, b^*\bigl(\bw_h,G_1(\bu_h,p_h),\bv_h\bigr)
\nonumber\\
&\quad
+\frac{\rho^2}{\alpha}
\bigl(\nabla\cdot G_1'(\bu_h,p_h;\bw_h,s_h),\nabla\cdot \bv_h\bigr)
-\rho\bigl(s_h,\nabla\cdot \bv_h\bigr) = 0,
\label{eq:Gprime_momentum}
\\[6pt]
&\alpha\bigl(G_2'(\bu_h,p_h;\bw_h,s_h)-s_h,q_h\bigr)
+\rho\bigl(\nabla\cdot G_1'(\bu_h,p_h;\bw_h,s_h),q_h\bigr) = 0.
\label{eq:Gprime_mass}
\end{align}
\end{definition}

\begin{lemma}\label{lem:Gprime_welldefined}
The operator $G'$ defined above is well-defined. That is, for every $(\bu_h,p_h),(\bw_h,s_h)\in \bX_h\times Q_h$ with $\frac{3d}{2}\rho^2 < \alpha$, the following inequality holds:
\[
\big\|G'(\bu_h,p_h;\bw_h,s_h)\big\|_{\mathcal{H}_h} \le C_L\,\big\|(\bw_h,s_h)\big\|_{\mathcal{H}_h}.
\]
\end{lemma}

\begin{proof}
Set $\bv_h = G_1'(\bu_h,p_h;\bw_h,s_h), q_h = G_2'(\bu_h,p_h;\bw_h,s_h)$
in~\eqref{eq:Gprime_momentum}--\eqref{eq:Gprime_mass} and add the resulting equations:
\begin{align*}
&\|\nabla G_1'(\bu_h,p_h;\bw_h,s_h)\|^2
+\alpha\|G_2'(\bu_h,p_h;\bw_h,s_h)\|^2
+\rho\nu\|\nabla G_1'(\bu_h,p_h;\bw_h,s_h)\|^2
+\frac{\rho^2}{\alpha}
\|\nabla\cdot G_1'(\bu_h,p_h;\bw_h,s_h)\|^2
\\
& = 
(\nabla \bw_h,\nabla G_1'(\bu_h,p_h;\bw_h,s_h))
+\alpha(s_h,G_2'(\bu_h,p_h;\bw_h,s_h))
+\rho(s_h,\nabla\cdot G_1'(\bu_h,p_h;\bw_h,s_h))
\\
&\quad
-\rho\, b^*(\bw_h,G_1(\bu_h,p_h),G_1'(\bu_h,p_h;\bw_h,s_h))
-\rho(\nabla\cdot G_1'(\bu_h,p_h;\bw_h,s_h),G_2'(\bu_h,p_h;\bw_h,s_h)).
\end{align*}
Dropping the positive terms
$\rho\nu\|\nabla G_1'(\bu_h,p_h;\bw_h,s_h)\|^2
\quad \text{and} \quad
\frac{\rho^2}{\alpha}
\|\nabla\cdot G_1'(\bu_h,p_h;\bw_h,s_h)\|^2$
from the left-hand side and bounding the right-hand side using the Cauchy--Schwarz inequality, Young's inequality, the boundedness of the trilinear form~\eqref{eq:b_bound}, and Lemma~\ref{lem:G_welldefined}, following the same strategy as in the proof of Lemma~\ref{lem:lipschitz}, we arrive at
\begin{equation*}
\begin{aligned}
&\left(\frac{1}{2}-\frac{3d\rho^2\alpha^{-1}}{4}\right)\,\|\nabla G_1'(\bu_h,p_h;\bw_h,s_h)\|^2
+\frac{\alpha}{4}\,\|G_2'(\bu_h,p_h;\bw_h,s_h)\|^2 \\
&\qquad\le \left(1+\rho^2\mathcal{M}^2\left(\|\nabla \bfu_h\|^2+\alpha\|p_h\|^2+\frac{\rho}{\nu}\|\bff\|_{-1}^2\right)\right)\,\|\nabla \bw_h\|^2
+2\alpha\|s_h\|^2.
\end{aligned}
\end{equation*}
Defining $K$ as in Lemma~\ref{lem:lipschitz}, dividing both sides by $K$, and applying the same argument yields the result.
\end{proof}
We now show that the operator $G'$ defined in Definition~\ref{def:Frechet} is the Fréchet
derivative operator of $G$ defined in Definition~\ref{def:G_operator}.
\begin{lemma}\label{lem:Gprime_is_Frechet}
The operator $G'$ defined above is the Fréchet derivative of $G$. That is, if
$\frac{d\rho^2}{1+2\rho\nu}<\alpha$, then for every $(\bu_h,p_h)\in \bX_h\times Q_h$,
\begin{equation*}
\big\|G(\bu_h+\bw_h, p_h+s_h)-G(\bu_h,p_h)-G'(\bu_h,p_h;\bw_h,s_h)\big\|_{\mathcal{H}_h}
\le \frac{\rho\mathcal{M}C_L}{\sqrt{2K'}}\,\big\|(\bw_h,s_h)\big\|_{\mathcal{H}_h}^2.
\end{equation*}
Here
\[
K' = \min\left\{\frac{1}{2}+\rho\nu-\frac{d\rho^2\alpha^{-1}}{2},\ \frac{1}{2}\right\}.
\]
\end{lemma}

\begin{proof}
Let
\begin{align*}
\bfeta_1 &:= G_1(\bu_h+\bw_h,p_h+s_h)-G_1(\bu_h,p_h)-G_1'(\bu_h,p_h;\bw_h,s_h),\\
\eta_2 &:= G_2(\bu_h+\bw_h,p_h+s_h)-G_2(\bu_h,p_h)-G_2'(\bu_h,p_h;\bw_h,s_h).
\end{align*}
From Definition~\ref{def:G_operator} and Definition~\ref{def:Frechet}, we compute the
difference $G(\bu_h+\bw_h,p_h+s_h)-G(\bu_h,p_h)-G'(\bu_h,p_h;\bw_h,s_h)$:
\begin{align}
(1+\rho\nu)(\nabla\bfeta_1,\nabla\bv_h)+\frac{\rho^2}{\alpha}(\nabla \cdot \bfeta_1,\nabla \cdot \bv_h)+\rho\,b^*(\bu_h,\bfeta_1,\bv_h)
\;\nonumber \\+\rho\,b^*\big(\bw_h, G_1(\bu_h+\bw_h,p_h+s_h)-G_1(\bu_h,p_h),\bv_h\big) = 0,\label{eq:Frechet_diff_mom}\\
\alpha(\eta_2,q_h)+\rho(\nabla \cdot \bfeta_1,q_h) = 0.
\end{align}
Setting $\bv_h = \bfeta_1$ and $q_h = \eta_2$ cancels the first nonlinear term on the left-hand
side. Adding the resulting equations and using absolute values gives
\begin{equation}\label{eq:Frechet_energy_1}
(1+\rho\nu)\|\nabla\bfeta_1\|^2+\alpha\|\eta_2\|^2
\le \rho\left|b^*\big(\bw_h, G_1(\bu_h+\bw_h,p_h+s_h)-G_1(\bu_h,p_h),\bfeta_1\big)\right|+\rho\left|(\nabla \cdot \bfeta_1,\eta_2)\right|,
\end{equation}
where we have dropped the positive term $\frac{\rho^2}{\alpha}\|\nabla\cdot \bfeta_1\|^2$
from the left-hand side of~\eqref{eq:Frechet_diff_mom}.
We bound the nonlinear term on the right-hand side of~\eqref{eq:Frechet_energy_1} using
the boundedness of the trilinear form~\eqref{eq:b_bound}, Young's inequality, and
Lemma~\ref{lem:lipschitz}. The remaining term is bounded using the Cauchy--Schwarz
inequality and $\|\nabla \cdot \bfeta_1\|\leq \sqrt{d}\,\|\nabla \bfeta_1\|$, followed by Young's
inequality, to obtain
\begin{align*}
(1+\rho\nu)\|\nabla\bfeta_1\|^2+\alpha\|\eta_2\|^2
&\leq
\frac{\rho^2\mathcal{M}^2C_L^2}{2}
\|(\bw_h,s_h)\|_{\mathcal{H}_h}^2\|\nabla\bw_h\|^2
\\
&\quad
+\frac{1}{2}\|\nabla\bfeta_1\|^2
+\frac{d\rho^2\alpha^{-1}}{2}\|\nabla\bfeta_1\|^2
+\frac{\alpha}{2}\|\eta_2\|^2.
\end{align*}
Using $\big\|(\bw_h,s_h)\big\|_{\mathcal{H}_h}^2 = \|\nabla\bw_h\|^2+\alpha\|s_h\|^2 \ge \|\nabla\bw_h\|^2$
and rearranging yields
\begin{align*}
&\left(\frac{1}{2}+\rho\nu-\frac{d\rho^2\alpha^{-1}}{2}\right)
\|\nabla\bfeta_1\|^2
+\frac{\alpha}{2}\|\eta_2\|^2
\\
&\quad \leq
\frac{\rho^2\mathcal{M}^2C_L^2}{2}
\|(\bw_h,s_h)\|_{\mathcal{H}_h}^2\|\nabla\bw_h\|^2
\\
&\quad \leq
\frac{\rho^2\mathcal{M}^2C_L^2}{2}
\|(\bw_h,s_h)\|_{\mathcal{H}_h}^4.
\end{align*}
Defining $K' = \min\left\{\frac{1}{2}+\rho\nu-\frac{d\rho^2\alpha^{-1}}{2},\frac{1}{2}\right\}$,
dividing both sides by $K'$, and taking square roots on both sides yields the result.
\end{proof}
We now show the Lipschitz property of the operator $G'$.
\begin{lemma}\label{lem:Gprime_lipschitz}
The operator $G'$ is Lipschitz continuous. That is, for all $\bu_h,\bw_h,\btheta_h\in \bX_h$
and $p_h,s_h,\xi_h\in Q_h$, if $\frac{d\rho^2}{1+2\rho\nu}<\alpha$, then the following
inequality holds:
\begin{equation*}
\big\|G'(\bu_h+\bw_h,p_h+s_h;\btheta_h,\xi_h)
-G'(\bu_h,p_h;\btheta_h,\xi_h)\big\|_{\mathcal{H}_h}
\leq
\sqrt{\frac{2}{K'}}\,\rho\mathcal{M}C_L\,
\big\|(\bw_h,s_h)\big\|_{\mathcal{H}_h}\,
\big\|(\btheta_h,\xi_h)\big\|_{\mathcal{H}_h}.
\end{equation*}
Here $K'$ is the constant defined in Lemma~\ref{lem:Gprime_is_Frechet}.
\end{lemma}

\begin{proof}
Define
\begin{align*}
\be_1 &:= 
G_1'(\bu_h+\bw_h,p_h+s_h;\btheta_h,\xi_h)
-G_1'(\bu_h,p_h;\btheta_h,\xi_h),\\
e_2 &:= 
G_2'(\bu_h+\bw_h,p_h+s_h;\btheta_h,\xi_h)
-G_2'(\bu_h,p_h;\btheta_h,\xi_h).
\end{align*}
Using Definition~\ref{def:Frechet}, we compute the difference
$G'(\bu_h+\bw_h,p_h+s_h;\btheta_h,\xi_h)-G'(\bu_h,p_h;\btheta_h,\xi_h)$:
\begin{align}
&(1+\rho\nu)(\nabla \be_1,\nabla\bv_h)
+\rho\,b^*(\btheta_h,
G_1(\bu_h+\bw_h,p_h+s_h)-G_1(\bu_h,p_h),\bv_h)
\nonumber\\
&\quad
+\rho\,b^*(\bu_h,\be_1,\bv_h)
+\rho\,b^*(\bw_h,
G_1'(\bu_h+\bw_h,p_h+s_h;\btheta_h,\xi_h),\bv_h)
\nonumber\\
&\quad
+\frac{\rho^2}{\alpha}(\nabla \cdot \be_1,\nabla \cdot\bv_h) = 0, \label{denklem2}
\\[4pt]
&\alpha(e_2,q_h)+\rho(\nabla \cdot \be_1,q_h) = 0.
\end{align}
Setting $\bv_h = \be_1$ and $q_h = e_2$ cancels the third term on the left-hand side
of~\eqref{denklem2}. Rearranging and adding the equations, then dropping the positive
term $\frac{\rho^2}{\alpha}\|\nabla \cdot \be_1\|^2$ from the left-hand side, gives
\begin{align*}
(1+\rho\nu)\|\nabla \be_1\|^2+\alpha\|e_2\|^2
\leq&
\rho\left|(\nabla \cdot \be_1,e_2)\right|
\\
&+\rho\left|b^*(\btheta_h,
G_1(\bu_h+\bw_h,p_h+s_h)-G_1(\bu_h,p_h),\be_1)\right|
\\
&+\rho\left|b^*(\bw_h,
G_1'(\bu_h+\bw_h,p_h+s_h;\btheta_h,\xi_h),\be_1)\right|.
\end{align*}

We now bound each term on the right-hand side. Applying the Cauchy--Schwarz inequality,
$\|\nabla\cdot \be_1\|\le \sqrt{d}\,\|\nabla \be_1\|$, and Young's inequality gives
\begin{equation*}
|-\rho(\nabla \cdot \be_1,e_2)|
\leq
\frac{d\rho^2\alpha^{-1}}{2}\|\nabla \be_1\|^2
+\frac{\alpha}{2}\|e_2\|^2.
\end{equation*}
Using the boundedness of the trilinear form~\eqref{eq:b_bound}, Young's inequality, and
Lemma~\ref{lem:lipschitz}:
\begin{align*}
&|-\rho\,b^*(\btheta_h,
G_1(\bu_h+\bw_h,p_h+s_h)-G_1(\bu_h,p_h),\be_1)|
\\
&\quad \leq
\rho^2\mathcal{M}^2 C_L^2\,
\|\nabla\btheta_h\|^2\,\|(\bw_h,s_h)\|_{\mathcal{H}_h}^2
+\frac{1}{4}\|\nabla \be_1\|^2.
\end{align*}
Similarly, using the boundedness of the trilinear form~\eqref{eq:b_bound}, Young's
inequality, and Lemma~\ref{lem:Gprime_welldefined}:
\begin{align*}
&|-\rho\,b^*(\bw_h,
G_1'(\bu_h+\bw_h,p_h+s_h;\btheta_h,\xi_h),\be_1)|
\\
&\quad \leq
\rho^2\mathcal{M}^2 C_L^2\,
\|\nabla\bw_h\|^2\,\|(\btheta_h,\xi_h)\|_{\mathcal{H}_h}^2
+\frac{1}{4}\|\nabla \be_1\|^2.
\end{align*}
Combining all these bounds:
\begin{align*}
&\left(\frac{1}{2}+\rho\nu-\frac{d\rho^2\alpha^{-1}}{2}\right)
\|\nabla \be_1\|^2
+\frac{\alpha}{2}\|e_2\|^2
\\
&\quad \leq
\rho^2\mathcal{M}^2 C_L^2\,
\|\nabla\btheta_h\|^2\,\|(\bw_h,s_h)\|_{\mathcal{H}_h}^2
+
\rho^2\mathcal{M}^2 C_L^2\,
\|\nabla\bw_h\|^2\,\|(\btheta_h,\xi_h)\|_{\mathcal{H}_h}^2.
\end{align*}
Using
\begin{equation}
\begin{aligned}
\|\nabla\btheta_h\|^2
&\leq
\|\nabla\btheta_h\|^2+\alpha\|\xi_h\|^2
 = 
\|(\btheta_h,\xi_h)\|_{\mathcal{H}_h}^2,\\
\|\nabla\bw_h\|^2
&\leq
\|\nabla\bw_h\|^2+\alpha\|s_h\|^2
 = 
\|(\bw_h,s_h)\|_{\mathcal{H}_h}^2,
\end{aligned}
\end{equation}
we obtain
\begin{align}
&\left(\frac{1}{2}+\rho\nu-\frac{d\rho^2\alpha^{-1}}{2}\right)
\|\nabla \be_1\|^2
+\frac{\alpha}{2}\|e_2\|^2
\nonumber\\
&\quad \leq
2\rho^2\mathcal{M}^2 C_L^2\,
\|(\btheta_h,\xi_h)\|_{\mathcal{H}_h}^2\,
\|(\bw_h,s_h)\|_{\mathcal{H}_h}^2.
\end{align}
Taking $K'$ as in Lemma~\ref{lem:Gprime_is_Frechet}, dividing both sides by $K'$, and
taking square roots on both sides yields the result.
\end{proof}

\begin{proposition}[Local verification of Assumption~\ref{assump:1}]\label{prop:assump1_local}
Let $\mathcal U\subset \mathcal H_h$ be a bounded neighborhood of a fixed point $x_h^*$. Under the parameter restrictions
\[
\frac{3d}{2}\rho^2<\alpha,
\qquad
\frac{d\rho^2}{1+2\rho\nu}<\alpha,
\]
there exist constants $C_0(\mathcal U)>0$ and $C_1(\mathcal U)>0$ such that for all $x,\widetilde{x}\in \mathcal U$ and all $y\in \mathcal H_h$,
\[
\|G'(x)y\|_{\mathcal H_h}\le C_0(\mathcal U)\|y\|_{\mathcal H_h},
\]
and
\[
\|G'(x)y-G'(\widetilde{x})y\|_{\mathcal H_h}
\le C_1(\mathcal U)\|x-\widetilde{x}\|_{\mathcal H_h}\|y\|_{\mathcal H_h}.
\]
Thus Assumption~\ref{assump:1} holds locally on $\mathcal U$.
\end{proposition}

\begin{proof}
The bounds in Lemmas~\ref{lem:Gprime_welldefined} and~\ref{lem:Gprime_lipschitz} contain constants depending on the base point through bounded $\mathcal H_h$-norms of the input variables. Since $\mathcal U$ is bounded and $\mathcal H_h$ is finite dimensional, these quantities are uniformly bounded on $\mathcal U$. Taking $C_0(\mathcal U)$ and $C_1(\mathcal U)$ to be the corresponding suprema gives the stated local bounds.
\end{proof}

\begin{remark}
The restrictions in Proposition~\ref{prop:assump1_local} are sufficient
conditions used to close the analytical estimates. They arise from conservative
applications of Cauchy--Schwarz and Young inequalities and are not intended as
sharp practical parameter-selection rules. Consequently, the numerical tests
below do not enforce these bounds. Instead, they use parameter values adopted
from the IAH literature and selected for computational performance. The
resulting experiments should therefore be interpreted as evidence that the
AA-IAH iteration is effective in regimes beyond those covered by the sufficient
local theory.
\end{remark}

\subsection{Local Verification of Assumption~\ref{assump:2}}

In the preceding sections, we examined the well-definedness, Lipschitz continuity,
Fréchet differentiability, and Lipschitz continuity of $G'$ for the fixed-point operator
\[
G:\mathcal{H}_h\to \mathcal{H}_h
\]
associated with the IAH scheme. Together with Proposition~\ref{prop:assump1_local}, these results establish that the smoothness conditions
required by the Pollock--Rebholz theory are satisfied locally for the IAH operator.

The second fundamental condition in the Pollock--Rebholz theory requires that the
differences of consecutive residuals be bounded below by the differences of consecutive
iterates. In our notation, with $x_k = (\bfu_h^k,p_h^k)\in \mathcal{H}_h$, the residual is
defined as
\[
r_{k+1} = G(x_k)-x_k.
\]
Assumption~\ref{assump:2} then requires that for some $\sigma>0$,
\[
\|r_{k+1}-r_k\|_{\mathcal{H}_h} \ge \sigma\|x_k-x_{k-1}\|_{\mathcal{H}_h}.
\]
As noted in~\cite{contractive_noncontractive}, this condition is automatically satisfied
when $G$ is contractive. The contraction property of the IAH scheme under the small data
condition and specific parameter choices was established in~\cite{IAH_for_NSE} (Theorem~3).
However, contractivity is not the only path to this condition. The fixed-point problem
$G(x) = x$ can be recast as a zero-finding problem $\mathcal{R}(x) = 0$ via the residual
operator $\mathcal{R}(x) := G(x)-x$. As noted in Remark~2.1 of~\cite{contractive_noncontractive},
Assumption~\ref{assump:2} is satisfied whenever the derivative of the residual operator is
bounded below on the relevant space, and the same remark states that this condition can be
localized to a neighborhood of the solution. Since we work in a local regime around the
solution, it suffices to work with the residual operator to verify
Assumption~\ref{assump:2} locally.

Following the discussion above, the key step is first to show that $\mathcal{R}'(x_h^*)$
is bounded below with respect to the $\mathcal{H}_h$-norm, and then to transfer this lower
bound to a sufficiently small neighborhood of $x_h^*$. To this end, we first show that
in the finite-dimensional discrete space,
\[
\ker \mathcal{R}'(x_h^*) = \{(0,0)\}.
\]
In a finite-dimensional space, once this property is established, it implies the lower bound in
Corollary~\ref{cor:residual_lower_bound} below. We then transfer this lower bound to a sufficiently
small neighborhood of $x_h^*$ to verify Assumption~\ref{assump:2} locally.

The following lemma establishes that the kernel of the residual derivative is trivial.

\begin{lemma}\label{lem:Residual_kernel}
Let $x_h^* = (\bfu_h^*,p_h^*)\in \mathcal{H}_h$ be a fixed point satisfying
$G(x_h^*) = x_h^*$. Suppose the small data condition~\eqref{small_data_condition}
and the discrete inf--sup condition~\eqref{eq:inf_sup} hold. Then
\[
\ker \mathcal{R}'(x_h^*) = \{(0,0)\}.
\]
\end{lemma}

\begin{proof}
Since $\mathcal{R}(x) = G(x)-x$ for $x\in \mathcal{H}_h$, we have
$\mathcal{R}'(x_h^*) = G'(x_h^*)-I$.
Suppose $\bphi_h\in \bX_h$ and $\psi_h\in Q_h$ satisfy
\[
\mathcal{R}'(x_h^*)(\bphi_h,\psi_h) = (0,0).
\]
Our goal is to show that $(\bphi_h,\psi_h) = (0,0)$.

Write $G'(x_h^*)(\bphi_h,\psi_h) = (\bfeta_h,\xi_h)$. Then
\[
\mathcal{R}'(x_h^*)(\bphi_h,\psi_h)
 = G'(x_h^*)(\bphi_h,\psi_h)-(\bphi_h,\psi_h)
 = (\bfeta_h-\bphi_h,\,\xi_h-\psi_h),
\]
so the assumption $\mathcal{R}'(x_h^*)(\bphi_h,\psi_h) = (0,0)$ yields
\[
\bfeta_h = \bphi_h, \qquad \xi_h = \psi_h.
\]

By Definition~\ref{def:Frechet} of $G'$, for all $(\bv_h,q_h)\in \bX_h\times Q_h$,
\[
\begin{aligned}
&(\nabla(\bfeta_h-\bphi_h),\nabla \bv_h)
+\rho\nu(\nabla \bfeta_h,\nabla \bv_h)
+\rho\, b^*(\bfu_h^*,\bfeta_h,\bv_h) \\
&\quad
+\rho\, b^*(\bphi_h,G_1(x_h^*),\bv_h)
+\frac{\rho^2}{\alpha}(\nabla\cdot\bfeta_h,\nabla\cdot \bv_h)
-\rho(\psi_h,\nabla\cdot \bv_h) = 0
\end{aligned}
\]
and
\[
\alpha(\xi_h-\psi_h,q_h)+\rho(\nabla\cdot\bfeta_h,q_h) = 0.
\]
Since $x_h^*$ is a fixed point, $G(x_h^*) = x_h^*$ and hence $G_1(x_h^*) = \bfu_h^*$.
Substituting $\bfeta_h = \bphi_h$ and $\xi_h = \psi_h$, the system reduces to
\begin{equation}\label{eq:rezidu_cekirdek_1}
\begin{aligned}
&\rho\nu(\nabla \bphi_h,\nabla \bv_h)
+\rho\, b^*(\bfu_h^*,\bphi_h,\bv_h)
+\rho\, b^*(\bphi_h,\bfu_h^*,\bv_h) \\
&\quad
+\frac{\rho^2}{\alpha}(\nabla\cdot\bphi_h,\nabla\cdot \bv_h)
-\rho(\psi_h,\nabla\cdot \bv_h) = 0
\end{aligned}
\end{equation}
and
\[
\rho(\nabla\cdot\bphi_h,q_h) = 0.
\]
Dividing~\eqref{eq:rezidu_cekirdek_1} by $\rho>0$ gives
\begin{equation}\label{denklemA}
\nu(\nabla \bphi_h,\nabla \bv_h)
+b^*(\bfu_h^*,\bphi_h,\bv_h)
+b^*(\bphi_h,\bfu_h^*,\bv_h)
+\frac{\rho}{\alpha}(\nabla\cdot\bphi_h,\nabla\cdot \bv_h)
-(\psi_h,\nabla\cdot \bv_h) = 0,
\end{equation}
and we also have
\begin{equation}\label{denklemB}
(\nabla\cdot\bphi_h,q_h) = 0 \qquad \forall q_h\in Q_h.
\end{equation}
Setting $\bv_h = \bphi_h$ in~\eqref{denklemA} and $q_h = \psi_h$ in~\eqref{denklemB} yields
$(\nabla\cdot\bphi_h,\psi_h) = 0$ and
\[
\nu\|\nabla\bphi_h\|^2
+b^*(\bfu_h^*,\bphi_h,\bphi_h)
+b^*(\bphi_h,\bfu_h^*,\bphi_h)
+\frac{\rho}{\alpha}\|\nabla\cdot\bphi_h\|^2
-(\psi_h,\nabla\cdot\bphi_h) = 0.
\]
By the skew-symmetry of $b^*$, we have $b^*(\bfu_h^*,\bphi_h,\bphi_h) = 0$, and since
$(\nabla\cdot\bphi_h,\psi_h) = 0$, we obtain
\[
\nu\|\nabla\bphi_h\|^2
+\frac{\rho}{\alpha}\|\nabla\cdot\bphi_h\|^2
+b^*(\bphi_h,\bfu_h^*,\bphi_h) = 0,
\]
which gives
\[
\nu\|\nabla\bphi_h\|^2
+\frac{\rho}{\alpha}\|\nabla\cdot\bphi_h\|^2
\le
|b^*(\bphi_h,\bfu_h^*,\bphi_h)|.
\]
Applying the trilinear form bound~\eqref{eq:b_bound},
\[
|b^*(\bphi_h,\bfu_h^*,\bphi_h)|
\le
\mathcal{M}\|\nabla\bphi_h\|\,\|\nabla \bfu_h^*\|\,\|\nabla\bphi_h\|
 = 
\mathcal{M}\|\nabla \bfu_h^*\|\,\|\nabla\bphi_h\|^2.
\]
Using the fixed-point stability bound from Remark~\ref{rem:fixed_point_augmented} and the small data
condition~\eqref{small_data_condition},
\[
\mathcal{M}\|\nabla \bfu_h^*\|
\le
\mathcal{M}\nu^{-1}\|\bff\|_{-1}
 = \Lambda\nu,
\]
so that
\[
|b^*(\bphi_h,\bfu_h^*,\bphi_h)| \le \Lambda\nu\|\nabla\bphi_h\|^2.
\]
Substituting back,
\[
\nu(1-\Lambda)\|\nabla\bphi_h\|^2
+\frac{\rho}{\alpha}\|\nabla\cdot\bphi_h\|^2
\le 0.
\]
Since $\Lambda<1$ by the small data condition and $\rho,\alpha>0$, this inequality holds
only if $\|\nabla\bphi_h\| = 0$. Since $\bX_h\subset H_0^1(\Omega)^d$, we conclude
$\bphi_h = 0$.

Substituting $\bphi_h = 0$ into~\eqref{denklemA} gives
$(\psi_h,\nabla\cdot \bv_h) = 0$ for all $\bv_h\in \bX_h$.
By the discrete inf--sup condition~\eqref{eq:inf_sup}, we conclude $\psi_h = 0$.
Therefore,
\[
\mathcal{R}'(x_h^*)(\bphi_h,\psi_h) = (0,0)
\;\Longrightarrow\;
(\bphi_h,\psi_h) = (0,0),
\]
i.e., $\ker \mathcal{R}'(x_h^*) = \{(0,0)\}$.
\end{proof}

\begin{corollary}\label{cor:residual_lower_bound}
By Lemma~\ref{lem:Residual_kernel}, the kernel of $\mathcal{R}'(x_h^*)$ contains only the
zero element. Define
\[
\sigma_*
:= 
\min_{\|z_h\|_{\mathcal{H}_h} = 1}
\|\mathcal{R}'(x_h^*)z_h\|_{\mathcal{H}_h}.
\]
Since the unit $\mathcal{H}_h$-sphere is compact in the finite-dimensional space
$\mathcal{H}_h$, this minimum is attained. If $\sigma_* = 0$, there would exist
$z_0\in \mathcal{H}_h$ with $\|z_0\|_{\mathcal{H}_h} = 1$ such that 
$\mathcal{R}'(x_h^*)z_0 = 0$, contradicting $\ker\mathcal{R}'(x_h^*) = \{(0,0)\}$.
Hence $\sigma_*>0$.

For any $z_h\ne 0$, set $\widehat{z}_h = z_h/\|z_h\|_{\mathcal{H}_h}$, so $\|\widehat{z}_h\|_{\mathcal{H}_h} = 1$ and
by definition of $\sigma_*$,
\[
\|\mathcal{R}'(x_h^*)\widehat{z}_h\|_{\mathcal{H}_h}\ge \sigma_*.
\]
By linearity of $\mathcal{R}'(x_h^*)$,
\[
\frac{\|\mathcal{R}'(x_h^*)z_h\|_{\mathcal{H}_h}}{\|z_h\|_{\mathcal{H}_h}}\ge \sigma_*,
\]
and therefore
\begin{equation}\label{eq:residual_lower_bound}
\|\mathcal{R}'(x_h^*)z_h\|_{\mathcal{H}_h} \ge \sigma_*\|z_h\|_{\mathcal{H}_h}
\qquad \forall z_h\in \mathcal{H}_h.
\end{equation}
\end{corollary}

The lower bound in Corollary~\ref{cor:residual_lower_bound} transfers to a sufficiently small
neighborhood of the fixed point. Since $\mathcal{R}' = G'-I$, the Lipschitz continuity of $G'$ in the
$\mathcal{H}_h$-norm carries over to $\mathcal{R}'$. We denote the corresponding local Lipschitz
constant by $L_{R'}$. Choose a convex neighborhood $\mathcal{U}$ of $x_h^*$ such that
$\|z-x_h^*\|_{\mathcal{H}_h}\le \varepsilon$ for all $z\in\mathcal U$. For $x,y\in\mathcal U$, set
$\delta=x-y$. Then
\[
\mathcal R(x)-\mathcal R(y)
=
\mathcal R'(x_h^*)\delta
+
\int_0^1
\left[\mathcal R'(y+t\delta)-\mathcal R'(x_h^*)\right]\delta\,dt.
\]
Therefore, by the reverse triangle inequality, \eqref{eq:residual_lower_bound}, and the local
Lipschitz continuity of $\mathcal R'$,
\[
\begin{aligned}
\|\mathcal R(x)-\mathcal R(y)\|_{\mathcal H_h}
&\ge
\|\mathcal R'(x_h^*)\delta\|_{\mathcal H_h}
-
\int_0^1
\|[\mathcal R'(y+t\delta)-\mathcal R'(x_h^*)]\delta\|_{\mathcal H_h}\,dt \\
&\ge
\sigma_*\|\delta\|_{\mathcal H_h}
-
L_{R'}\varepsilon\|\delta\|_{\mathcal H_h}.
\end{aligned}
\]
Choosing $\varepsilon$ so that $L_{R'}\varepsilon\le \sigma_*/2$ gives
\[
\|\mathcal R(x)-\mathcal R(y)\|_{\mathcal H_h}
\ge
\frac{\sigma_*}{2}\|x-y\|_{\mathcal H_h}
\qquad \forall x,y\in\mathcal U.
\]
Since $r_{k+1}=\mathcal R(x_k)$, Assumption~\ref{assump:2} is locally satisfied with
$\sigma=\sigma_*/2$ as long as the relevant iterates remain in this neighborhood.

\subsection{Residual Analysis of the AA-IAH Algorithm for Steady-State NSE}

The preceding lemmas establish that the solution operator $G$ associated with the IAH iteration~\eqref{eq:G_momentum}--\eqref{eq:G_mass} satisfies the conditions required by Theorem~\ref{theorem:pollock} locally near the fixed point. Specifically, Proposition~\ref{prop:assump1_local}
verifies the boundedness and Lipschitz-continuity conditions of
Assumption~\ref{assump:1}. Assumption~\ref{assump:2} is verified
locally in a neighborhood of the fixed point via Lemma~\ref{lem:Residual_kernel},
Corollary~\ref{cor:residual_lower_bound}, and the local lower-bound argument immediately above.
With both assumptions in place on this local neighborhood,
Theorem~\ref{theorem:pollock} applies to AA-IAH iterates that remain in the neighborhood and yields the
following one-step residual bound.

\begin{corollary}\label{cor:residual_bound}
Let Assumptions~\ref{assump:1} and~\ref{assump:2} hold locally for the operator $G$ associated
with the IAH iteration~\eqref{eq:G_momentum}--\eqref{eq:G_mass}, and assume that the relevant AA-IAH iterates remain in this local neighborhood. Then for any step
$k > m$, the residual $r_{k+1}$ of the AA-IAH iteration satisfies
\begin{equation}\label{eq:residual_bound}
\begin{aligned}
\|r_{k+1}\|_{\mathcal{H}_h}
\leq\; & \theta_k C_L \|r_k\|_{\mathcal{H}_h} \\
&+ C\sqrt{1-\theta_k^2}\,\|r_k\|_{\mathcal{H}_h}
  \sum_{j = 1}^{m} \|r_{k-j+1}\|_{\mathcal{H}_h},
\end{aligned}
\end{equation}
where $\theta_k \in [0,1]$ denotes the gain factor obtained from the Anderson
optimization step, $C_L$ denotes the local derivative-bound constant $C_0(\mathcal U)$
from Proposition~\ref{prop:assump1_local}, and $C$ is a generic constant depending on $C_1$, the Anderson depth $m$, $c_s$, the direction-cosine bounds, and the bounds on $h(\theta_j)$.
\end{corollary}

\begin{proof}
The result follows by applying Theorem~\ref{theorem:pollock} on the local neighborhood where Assumptions~\ref{assump:1}--\ref{assump:2} hold, replacing $C_0$ by the local derivative-bound constant denoted by $C_L$, and absorbing the bounded factors involving $C_1$, $m$, $h(\theta_j)$, and the optimization-matrix geometry into the generic constant $C$.
\end{proof}

\section{Numerical Experiments}
\label{sec:numerical_experiments}

In this section, a series of numerical experiments is presented to evaluate the accuracy, stability, and computational efficiency of the proposed Anderson-accelerated improved Arrow--Hurwicz algorithm in comparison with the improved Arrow--Hurwicz method. The numerical tests are organized into three main stages:
\begin{itemize}
    \item \textbf{Accuracy Analysis:} Designed to mathematically demonstrate that the algorithm is correctly implemented and achieves the expected theoretical error convergence rates for the spatial discretization.
    \item \textbf{Lid-Driven Cavity Flow:} Investigated to test the performance, stability, and acceleration capacity of the algorithm, particularly at high Reynolds numbers, by comparing the results with well-known benchmark data in the literature.
    \item \textbf{Channel Flow Over a Full Step:} Considered to demonstrate the capability of the proposed method in modeling complex physical flow structures, such as flow separation and recirculation zones.
\end{itemize}

In all numerical simulations, biquadratic ($Q_2$) and bilinear ($Q_1$) continuous polynomials are used for the velocity and pressure fields, respectively. The implementation of the algorithms and the solution of the linear systems are carried out in C++ using version 9.8.0-pre (commit id: 82c56fc0ed) of the open-source finite element library \texttt{deal.II}~\cite{2025:arndt.bangerth.ea:deal}. All computations and CPU time measurements presented in this study are executed on a computer architecture equipped with an Apple M4 processor.

\subsection{Convergence Test}

To perform the error analysis of our numerical scheme and determine the convergence rates, we use the manufactured solution developed by Takhirov et al.~\cite{IAH_for_NSE}. On the unit square $\Omega = (0,1)^2$, the analytic velocity and pressure solutions for equations~\eqref{nse} are defined as follows

\[
\bfu = \begin{pmatrix} 2\theta(x^2 - x)^2(y^2 - y)(2y - 1) \\ -2\theta(y^2 - y)^2(x^2 - x)(2x - 1) \end{pmatrix}, \quad p = \theta(2x - 1)(2y - 1).
\]

Here, $\theta$ is a parameter used to adjust the value of $\Lambda$, and the right-hand side vector $\bff$ is directly computed to correspond to this solution. For varying spatial mesh sizes $h = 1/2^N$, where $N$ represents the number of refinements, the model parameters are chosen in accordance with the reference study~\cite{IAH_for_NSE}.

The relative iterate change used in the numerical implementation is
\begin{equation} \label{stopping_criteria}
\eta_{n+1}:=
\max \left\{ \frac{\|\bfu_h^{n+1} - \bfu_h^n\|_{\ell^2}}{\|\bfu_h^{n+1}\|_{\ell^2}}, \frac{\|p_h^{n+1} - p_h^n\|_{\ell^2}}{\|p_h^{n+1}\|_{\ell^2}} \right\}.
\end{equation}
The convergence tables use the stopping tolerance $\eta_{n+1}\leq 10^{-10}$ so that the stopping criterion is much tighter than the smallest reported discretization errors, while the cavity and channel performance tests use $\eta_{n+1}\leq 10^{-6}$. All error integrals in the convergence study are evaluated with a five-point Gauss rule in each coordinate direction.

The errors in the various norms and their corresponding convergence rates are reported for the IAH method in Table~\ref{tab:convergence_rates_1} and for AA-IAH with depth $m=2$ in Table~\ref{tab:convergence_rates_2}. Since Anderson acceleration changes the nonlinear iteration path but not the underlying $Q_2$--$Q_1$ spatial discretization, the two algorithms yield nearly identical errors and convergence rates. Therefore, only one representative table is presented for each algorithm, while the remaining nearly identical results are omitted to avoid redundancy and save space. Both tables exhibit the expected spatial convergence behavior.

\begin{table}[ht]
\centering
\caption{Errors and convergence rates for the IAH algorithm with $\theta = 2$, $\alpha = 1$, and $\rho = 543.1$.}
\label{tab:convergence_rates_1}
\renewcommand{\arraystretch}{1.2}
\resizebox{\textwidth}{!}{
\begin{tabular}{|c|c||c|c|c|c|c|c||c|c|}
\cline{3-10}
\multicolumn{2}{c|}{} & \multicolumn{6}{c||}{Errors and CRs in Velocity ($\bu$)} & \multicolumn{2}{c|}{Pressure ($p$)} \\
\hline
$N$ & DoFs ($\bu,p$) & $\|\bu-\bu^h\|$ & CR & $\|\nabla \bu - \nabla \bu^h\|$ & CR & $\|\nabla \cdot \bu^h\|$ & CR & $\|p-p^h\|$ & CR \\
\hline
2 & 162, 25 & 3.5566e-3 & - & 4.8941e-2 & - & 8.7394e-4 & - & 9.0389e-3 & - \\
\hline
3 & 578, 81 & 6.1252e-4 & 2.54 & 1.6824e-2 & 1.54 & 5.9236e-4 & 0.56 & 1.6750e-3 & 2.43 \\
\hline
4 & 2178, 289 & 7.1799e-5 & 3.09 & 3.9932e-3 & 2.07 & 2.6441e-4 & 1.16 & 2.7662e-4 & 2.60 \\
\hline
5 & 8450, 1089 & 6.4077e-6 & 3.49 & 7.5713e-4 & 2.40 & 8.4235e-5 & 1.65 & 3.5775e-5 & 2.95 \\
\hline
6 & 33282, 4225 & 4.9985e-7 & 3.68 & 1.2913e-4 & 2.55 & 2.3258e-5 & 1.86 & 3.3421e-6 & 3.42 \\
\hline
7 & 132098, 16641 & 3.5773e-8 & 3.80 & 2.0310e-5 & 2.67 & 6.0435e-6 & 1.94 & 2.5652e-7 & 3.70 \\
\hline
8 & 526338, 66049 & 2.5922e-9 & 3.79 & 3.3142e-6 & 2.62 & 1.5302e-6 & 1.98 & 1.8211e-8 & 3.82 \\
\hline
\end{tabular}%
}
\end{table}

\begin{table}[ht]
\centering
\caption{Errors and convergence rates for the AA-IAH algorithm with $m=2$, $\theta = 1$, $\alpha = 1$, and $\rho = 2240.4$.}
\label{tab:convergence_rates_2}
\renewcommand{\arraystretch}{1.2}
\resizebox{\textwidth}{!}{
\begin{tabular}{|c|c||c|c|c|c|c|c||c|c|}
\cline{3-10}
\multicolumn{2}{c|}{} & \multicolumn{6}{c||}{Errors and CRs in Velocity ($\bu$)} & \multicolumn{2}{c|}{Pressure ($p$)} \\
\hline
$N$ & DoFs ($\bu,p$) & $\|\bu-\bu^h\|$ & CR & $\|\nabla \bu - \nabla \bu^h\|$ & CR & $\|\nabla \cdot \bu^h\|$ & CR & $\|p-p^h\|$ & CR \\
\hline
2 & 162, 25 & 2.0058e-3 & - & 2.7555e-2 & - & 1.2007e-4 & - & 4.8060e-3 & - \\
\hline
3 & 578, 81 & 4.4136e-4 & 2.18 & 1.2103e-2 & 1.19 & 1.0511e-4 & 0.19 & 1.0035e-3 & 2.26 \\
\hline
4 & 2178, 289 & 7.7214e-5 & 2.52 & 4.2335e-3 & 1.52 & 7.2599e-5 & 0.53 & 2.0136e-4 & 2.32 \\
\hline
5 & 8450, 1089 & 9.1375e-6 & 3.08 & 1.0117e-3 & 2.07 & 3.2779e-5 & 1.15 & 3.4537e-5 & 2.54 \\
\hline
6 & 33282, 4225 & 8.1737e-7 & 3.48 & 1.9127e-4 & 2.40 & 1.0495e-5 & 1.64 & 4.5428e-6 & 2.93 \\
\hline
7 & 132098, 16641 & 6.3515e-8 & 3.69 & 3.2208e-5 & 2.57 & 2.9034e-6 & 1.85 & 4.2750e-7 & 3.41 \\
\hline
8 & 526338, 66049 & 4.4479e-9 & 3.84 & 4.8437e-6 & 2.73 & 7.5508e-7 & 1.94 & 3.2908e-8 & 3.70 \\
\hline
\end{tabular}%
}
\end{table}

\subsection{Lid-Driven Cavity Flow}

In this subsection, the lid-driven cavity problem on $\Omega = (0,1)^2$ is considered as a second benchmark. A unit tangential velocity drives the top lid in the positive $x$-direction, while homogeneous Dirichlet (no-slip) conditions are enforced on the three remaining walls. To prevent the corner singularities that arise from the velocity discontinuity at the upper edges, the top-lid profile is smoothed following the regularization proposed in~\cite{deFrutos2016}
\begin{equation*}
    \bfu(x, 1) = (u_1(x), 0)^T
\end{equation*}
where $u_1(x)$ is defined as follows
\begin{equation}
    u_1(x) = \begin{cases}
    1 - \frac{1}{4}\left(1 - \cos\left(\frac{0.1 - x}{0.1}\pi\right)\right)^2 & 0 \le x \le 0.1, \\
    1 & 0.1 < x < 0.9, \\
    1 - \frac{1}{4}\left(1 - \cos\left(\frac{x - 0.9}{0.1}\pi\right)\right)^2 & 0.9 \le x \le 1.
    \end{cases}
\end{equation}

We evaluate the algorithm for six different Reynolds numbers ($Re = 1{,}000, 2{,}500, 5{,}000, 10{,}000, 12{,}500$, and $15{,}000$), setting $\rho = 100$ and $\alpha = 1$ for all simulations. The algorithmic parameters $\rho = 100$ and $\alpha = 1$ are adopted
directly from the reference study~\cite{IAH_for_NSE}, in which this
choice was shown to yield stable and convergent results across a wide
range of Reynolds numbers. All subsequent results are reported on a mesh with a resolution of $256 \times 256$, corresponding to a total of 592,387 degrees of freedom (DoFs), comprising 526,338 DoFs for the velocity field and 66,049 DoFs for the pressure.

Figures~\ref{fig:performance_comparison} and~\ref{fig:converge_history} compare the IAH and proposed AA-IAH methods in terms of total iteration counts, CPU times, and relative iterate-change histories for Reynolds numbers ranging from $1{,}000$ to $15{,}000$. To avoid visual complexity in the iterate-change plots, only the case $m=4$ is presented for all Reynolds numbers. As the results indicate, the AA-IAH method significantly reduces the computational burden in both low- and high-$Re$ regimes. The performance panels show greater efficiency than IAH for every tested value of $m$ and $Re$.

Moreover, a detailed examination of the convergence histories reveals that the IAH method exhibits a slow but smoother decrease toward the tolerance threshold. In contrast, the proposed AA-IAH algorithm displays noticeable oscillations. These fluctuations are a natural consequence of combining previous mapped iterates, as also observed in~\cite{contractive_noncontractive}. The influence of the memory depth $m$ is examined below through centerline-profile comparisons.

\begin{figure*}[ht]
    \centering
    \begin{subfigure}{0.32\textwidth}
        \includegraphics[width = \linewidth]{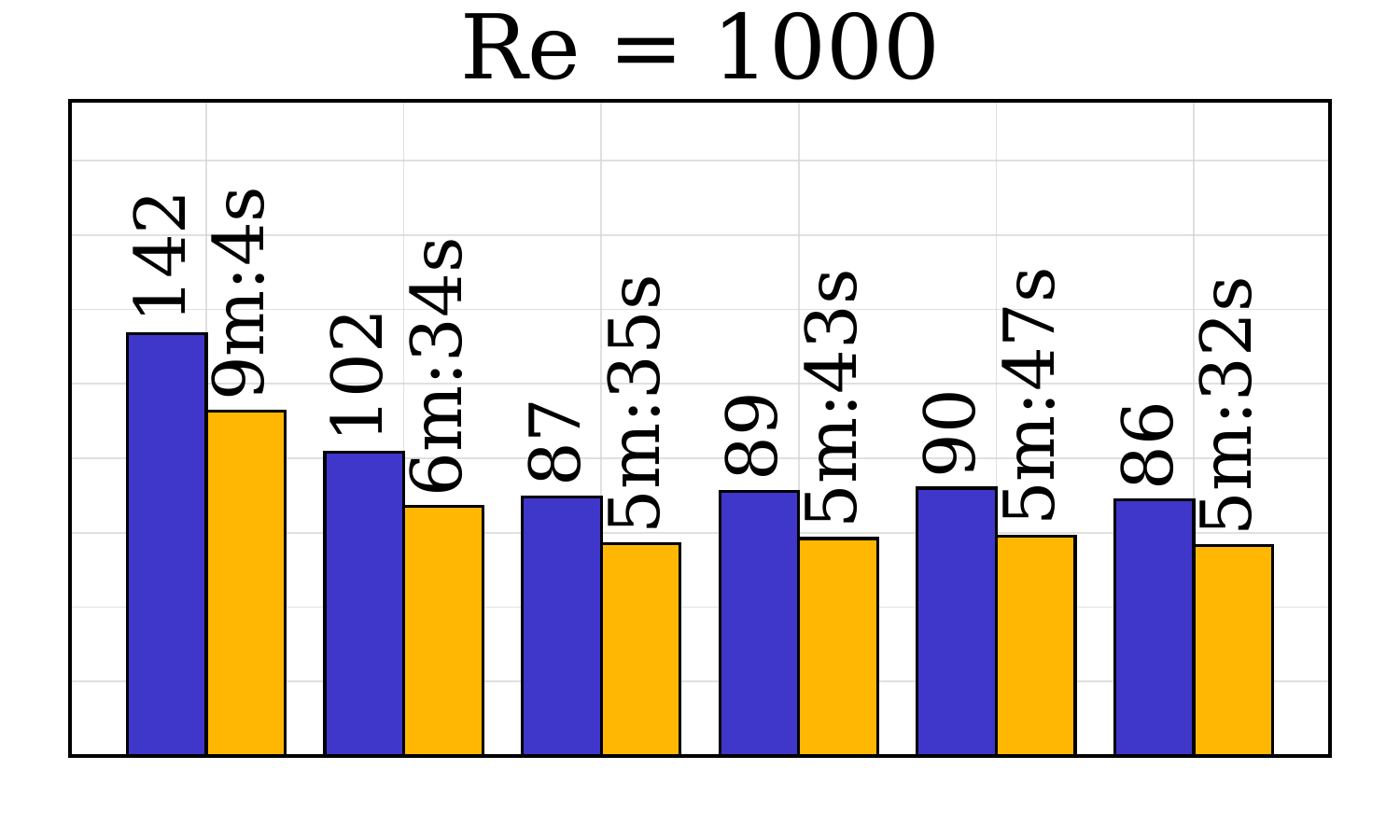}
    \end{subfigure}
    \begin{subfigure}{0.32\textwidth}
        \includegraphics[width = \linewidth]{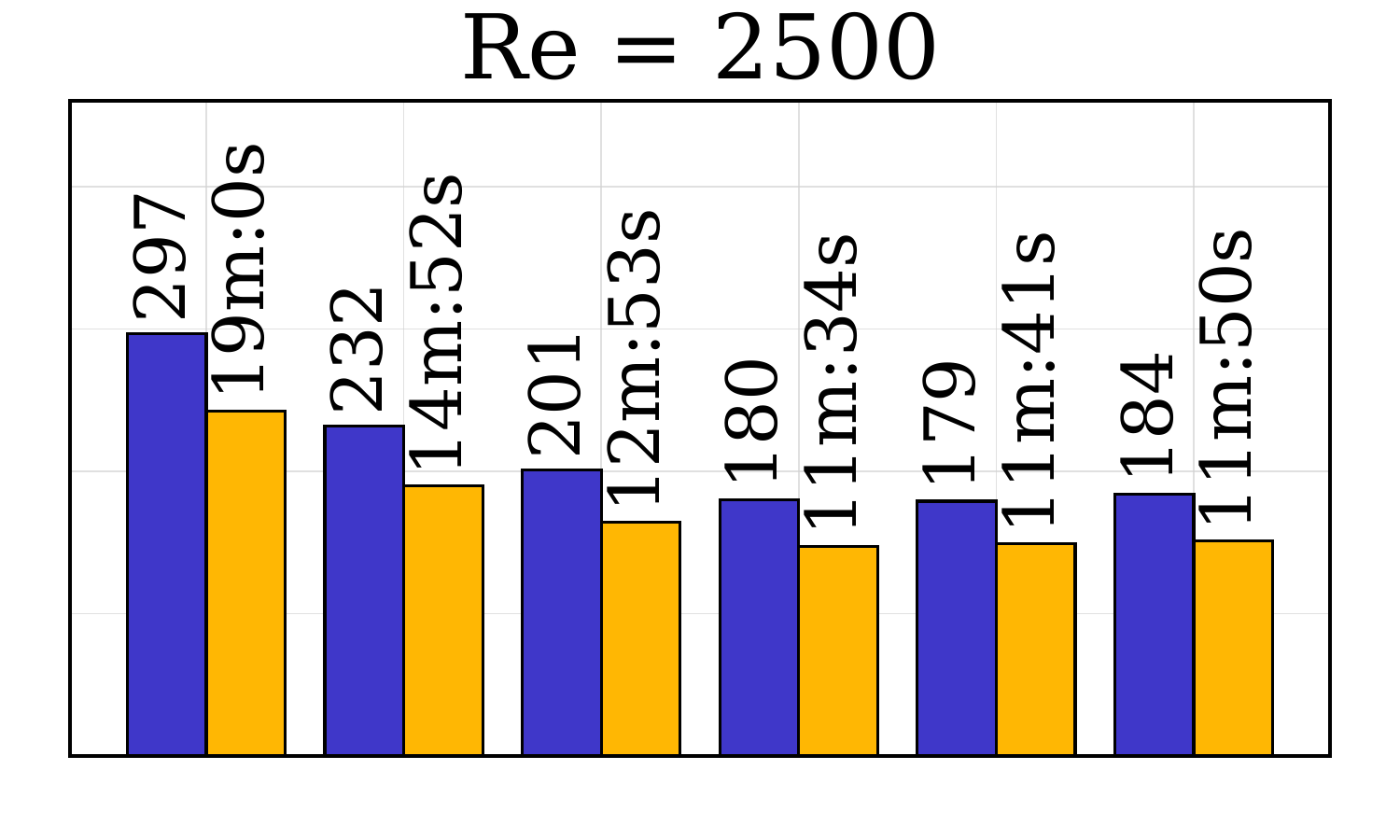}
    \end{subfigure}
    \begin{subfigure}{0.32\textwidth}
        \includegraphics[width = \linewidth]{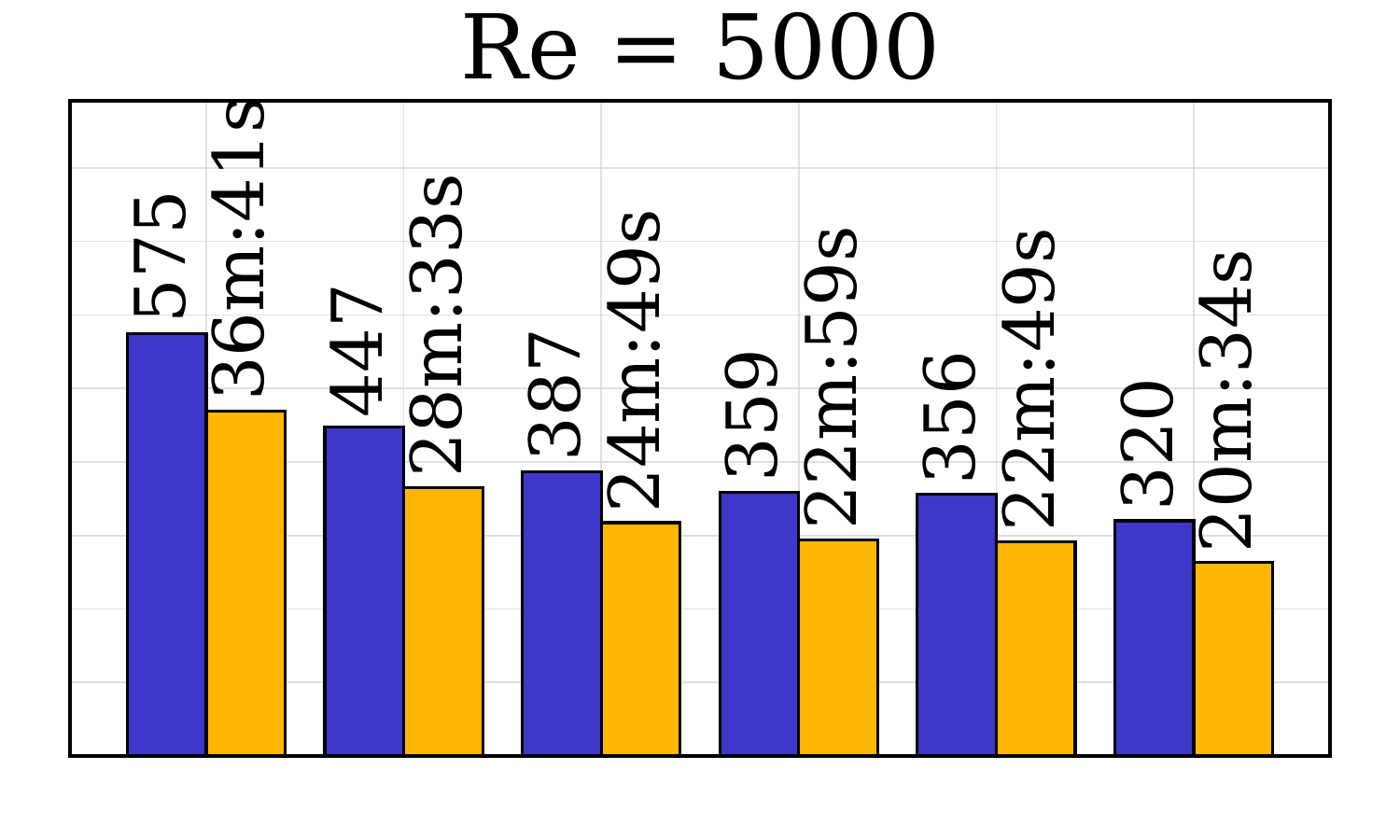}
    \end{subfigure}
    \begin{subfigure}{0.32\textwidth}
        \includegraphics[width = \linewidth]{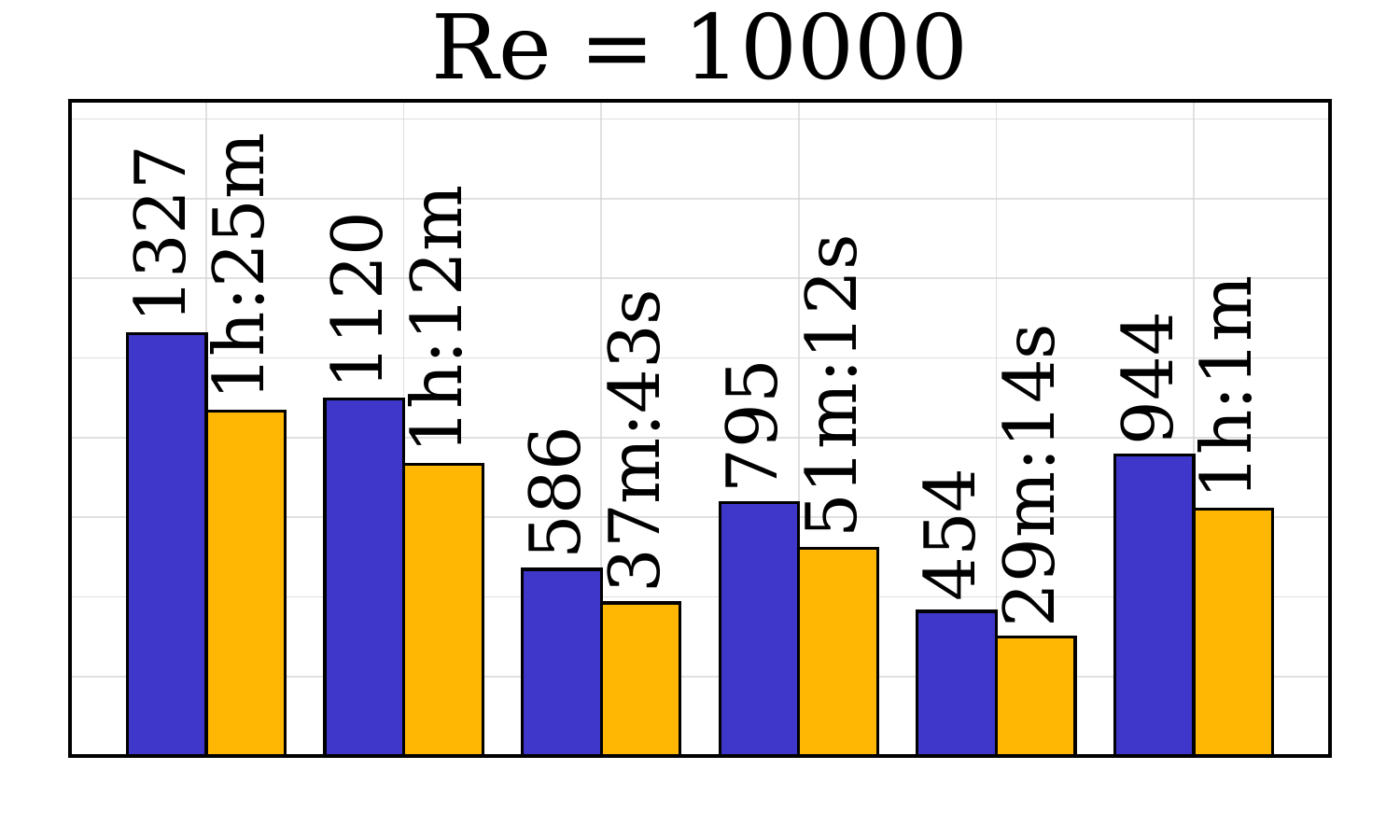}
    \end{subfigure}
    \begin{subfigure}{0.32\textwidth}
        \includegraphics[width = \linewidth]{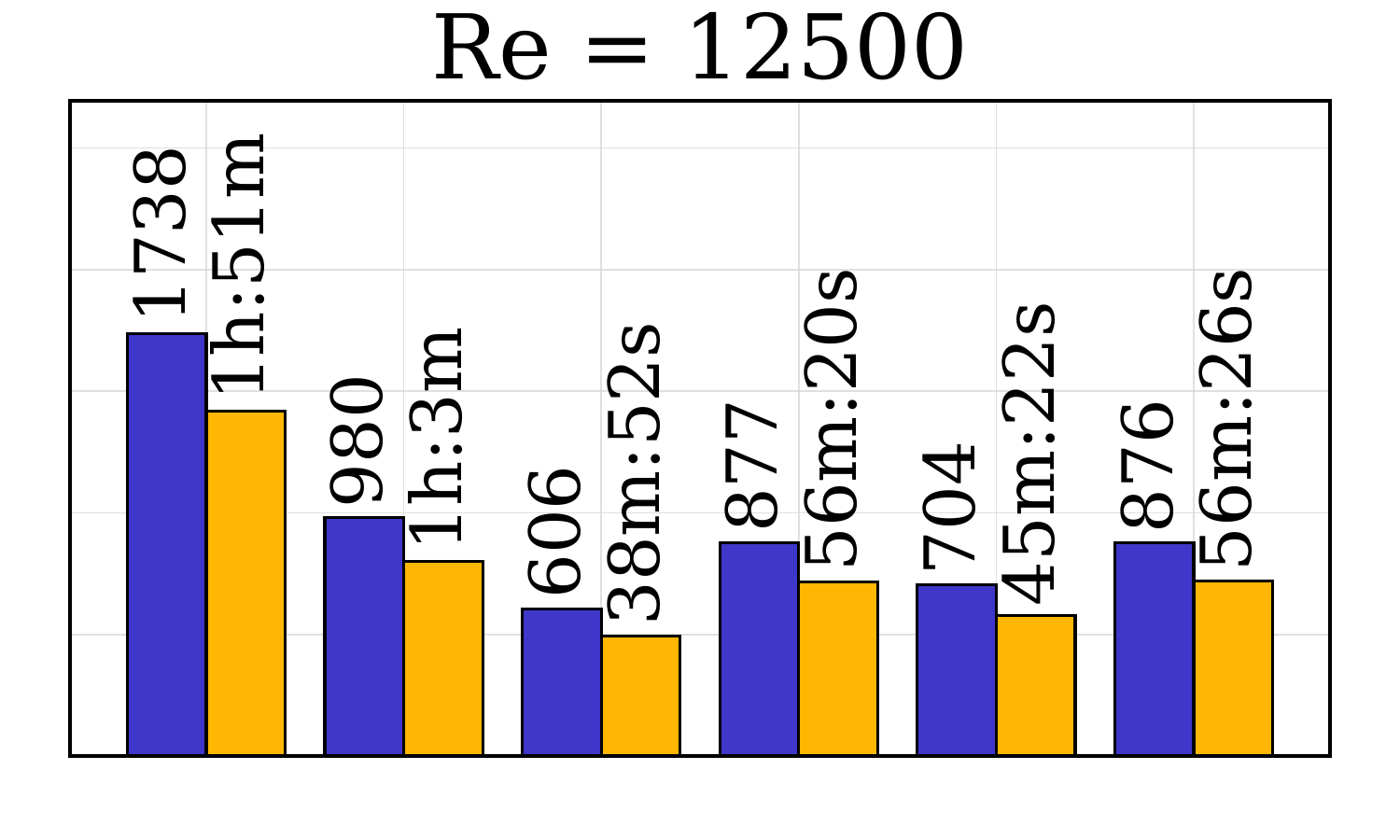}
    \end{subfigure}
    \begin{subfigure}{0.32\textwidth}
        \includegraphics[width = \linewidth]{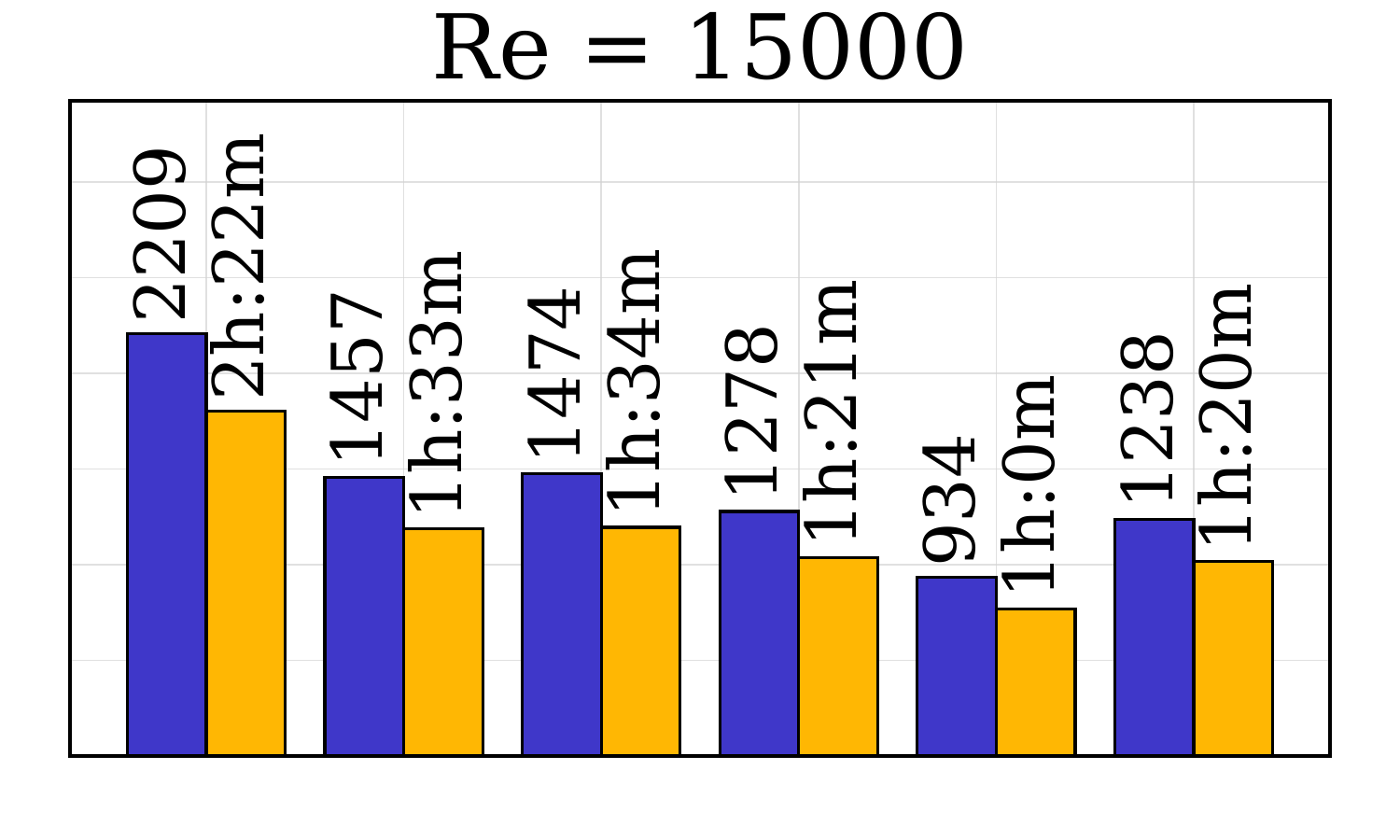}
    \end{subfigure}
    \caption{Performance comparison of IAH and AA-IAH($m$) algorithms in terms of total iteration counts and CPU times. Leftmost bars represent IAH, while the remaining bars correspond to AA-IAH with $m = 1, 2, 3, 4$ from left to right.}
    \label{fig:performance_comparison}
\end{figure*}

\begin{figure*}[ht]
    \centering

    \begin{subfigure}[b]{0.365\textwidth}
        \includegraphics[trim={0 0 0 0}, clip, width=\textwidth]{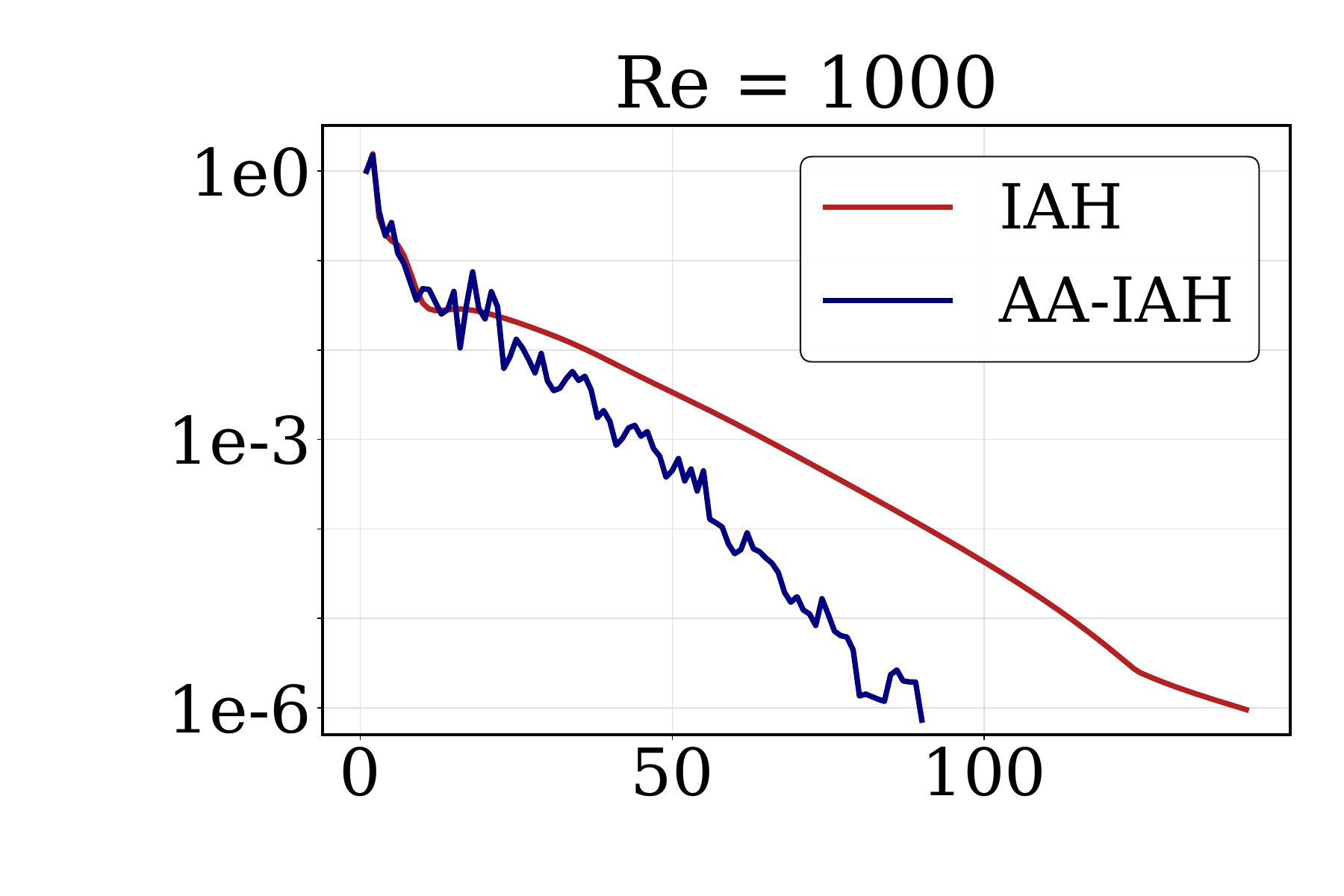}
    \end{subfigure}
    \hfill
    \begin{subfigure}[b]{0.30\textwidth}
        \includegraphics[trim={150 0 0 0}, clip, width=\textwidth]{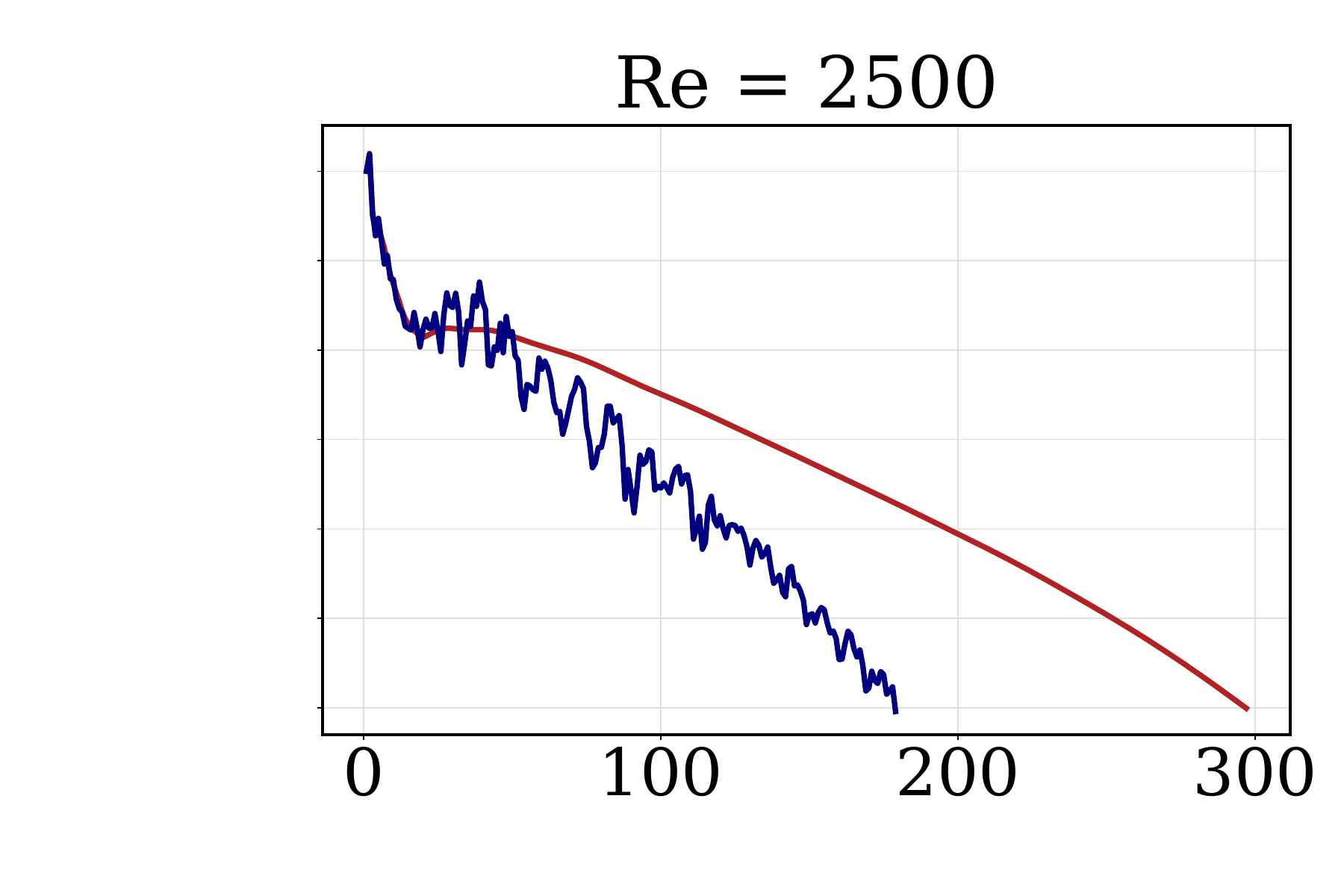}
    \end{subfigure}
    \hfill
    \begin{subfigure}[b]{0.30\textwidth}
        \includegraphics[trim={150 0 0 0}, clip, width=\textwidth]{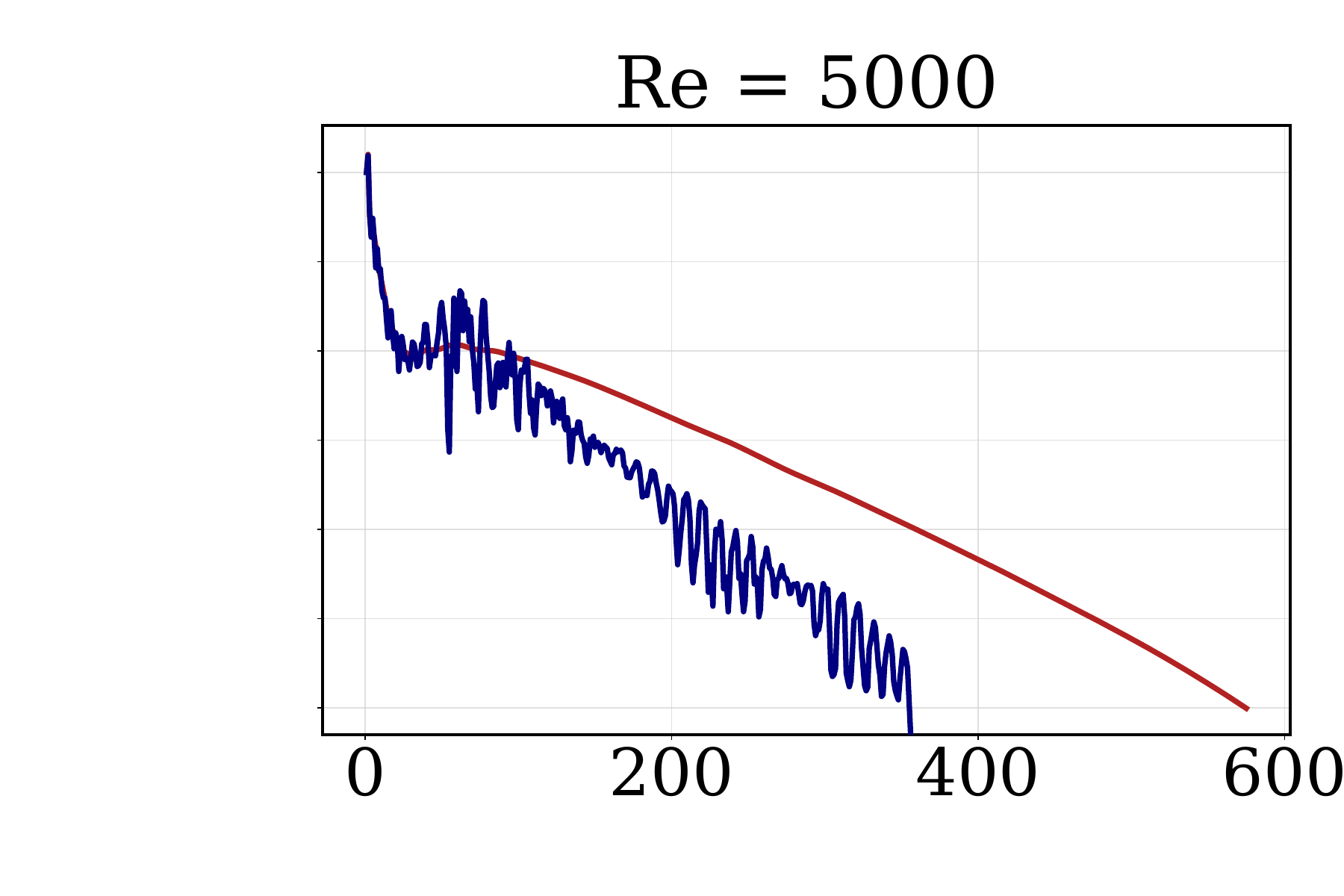}
    \end{subfigure}

    \vspace{0.5em}

    \begin{subfigure}[b]{0.365\textwidth}
        \includegraphics[trim={0 0 0 0}, clip, width=\textwidth]{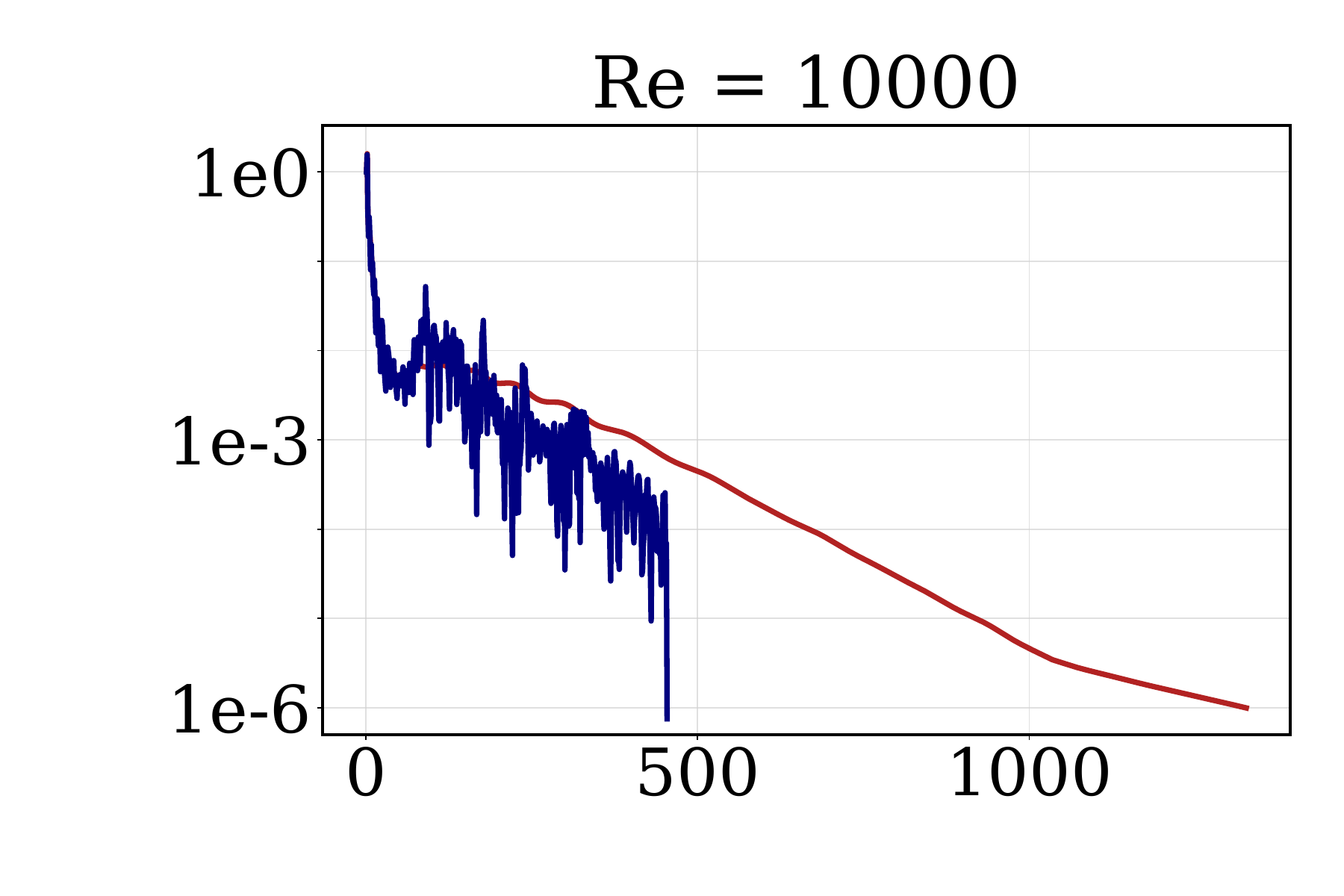}
    \end{subfigure}
    \hfill
    \begin{subfigure}[b]{0.30\textwidth}
        \includegraphics[trim={150 0 0 0}, clip, width=\textwidth]{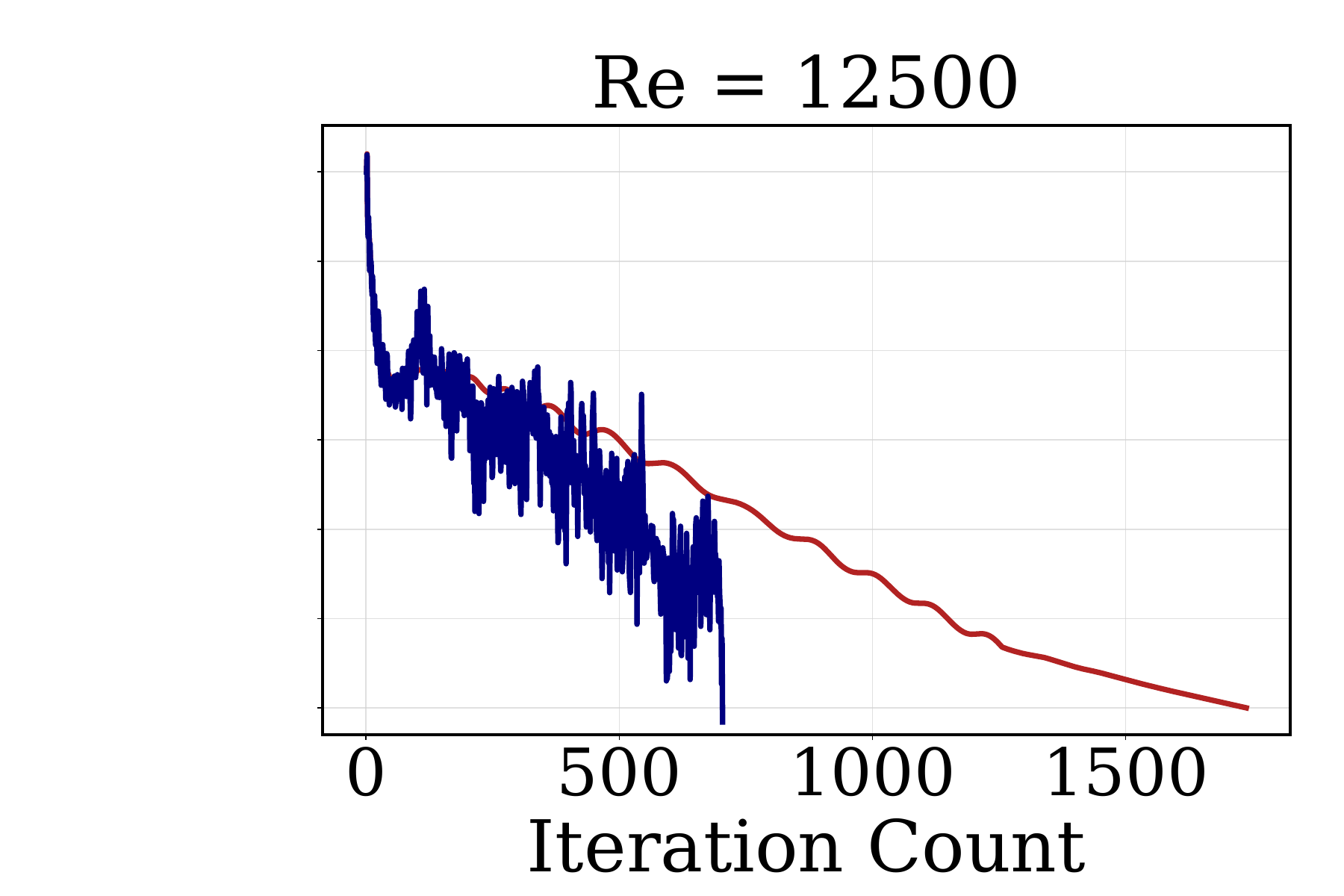}
    \end{subfigure}
    \hfill
    \begin{subfigure}[b]{0.30\textwidth}
        \includegraphics[trim={150 0 0 0}, clip, width=\textwidth]{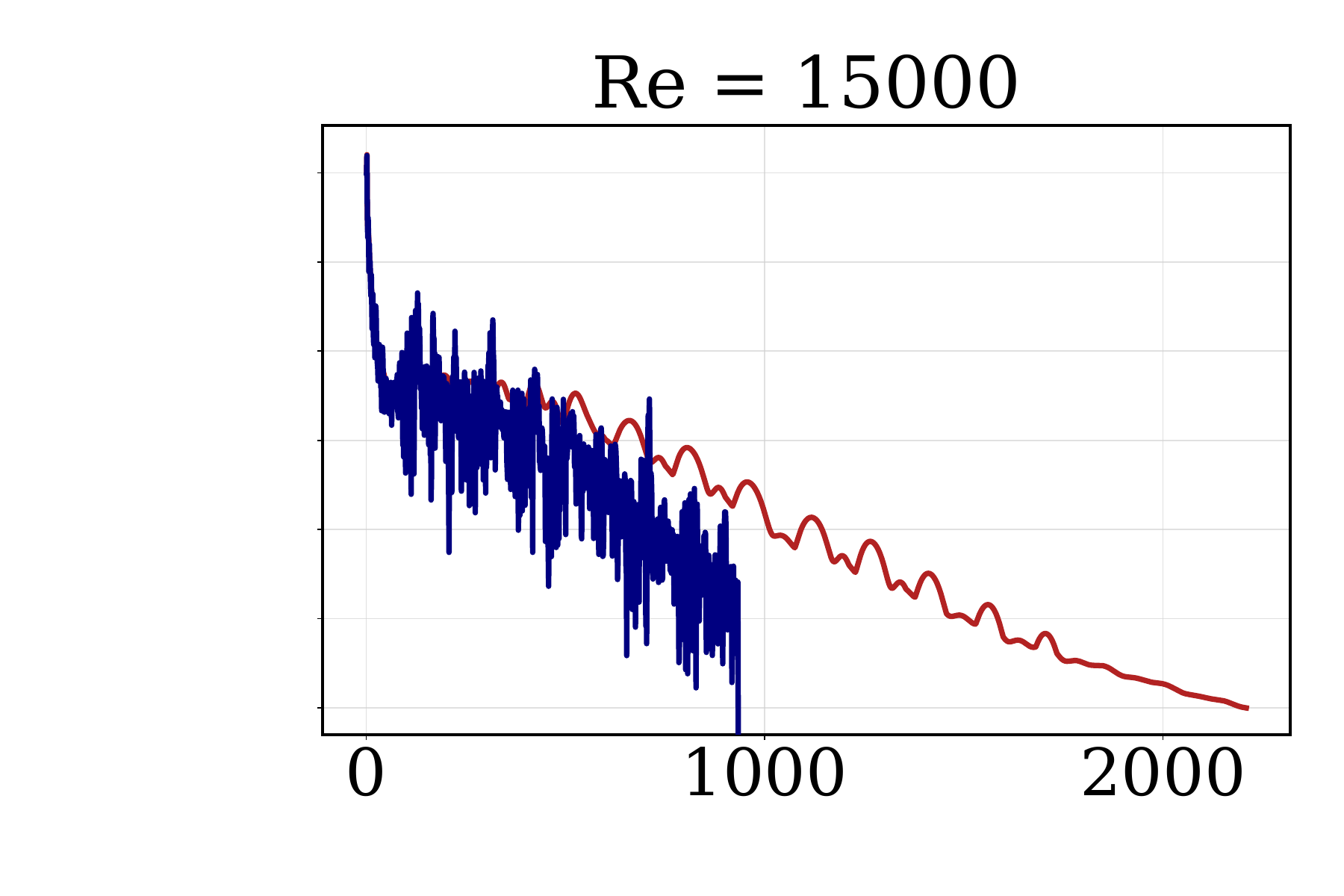}
    \end{subfigure}
    \caption{Relative iterate-change history of the IAH and AA-IAH($4$) algorithms.}
    \label{fig:converge_history}
\end{figure*}

For a quantitative comparison, the numerical results obtained by our algorithm are compared against the
reference values of the velocity components along the centerlines provided by Erturk et al.~\cite{Erturk}.
Tables~\ref{lid_driven_re1000} and~\ref{lid_driven_re15000} present these comparisons. As can be seen from
the tables, the results are in good agreement with the reference values. 

Figure~\ref{fig:streamline} shows the streamline structures for representative Reynolds numbers. The observed vortex topology and general flow patterns agree well with the visual results reported by Erturk et al.~\cite{Erturk}. The effect of the memory depth $m$ on the converged centerline profiles is examined for Reynolds numbers ranging from $Re = 1{,}000$ to $15{,}000$. Figure~\ref{fig:profil} shows the centerline profiles of the $u$ and $v$ velocity components for $Re \ge 10{,}000$ and their agreement with the reference data.

The velocity profiles generated by IAH and AA-IAH for the tested values of $m$ closely follow the reference data at the plotted resolution. The curves largely overlap, while the zoomed windows near the turning points make the small differences among the methods and reference values visible. This comparison is a centerline-velocity diagnostic rather than a norm-based full-field error estimate.

Together, the centerline velocity profiles in Figure~\ref{fig:profil} and streamlines in Figure~\ref{fig:streamline} show that the measured reduction in computation time is accompanied by closely matching velocity diagnostics for these cavity tests.

\begin{figure}[ht]
    \centering
    \begin{subfigure}{0.23\textwidth}
        \includegraphics[width = \linewidth]{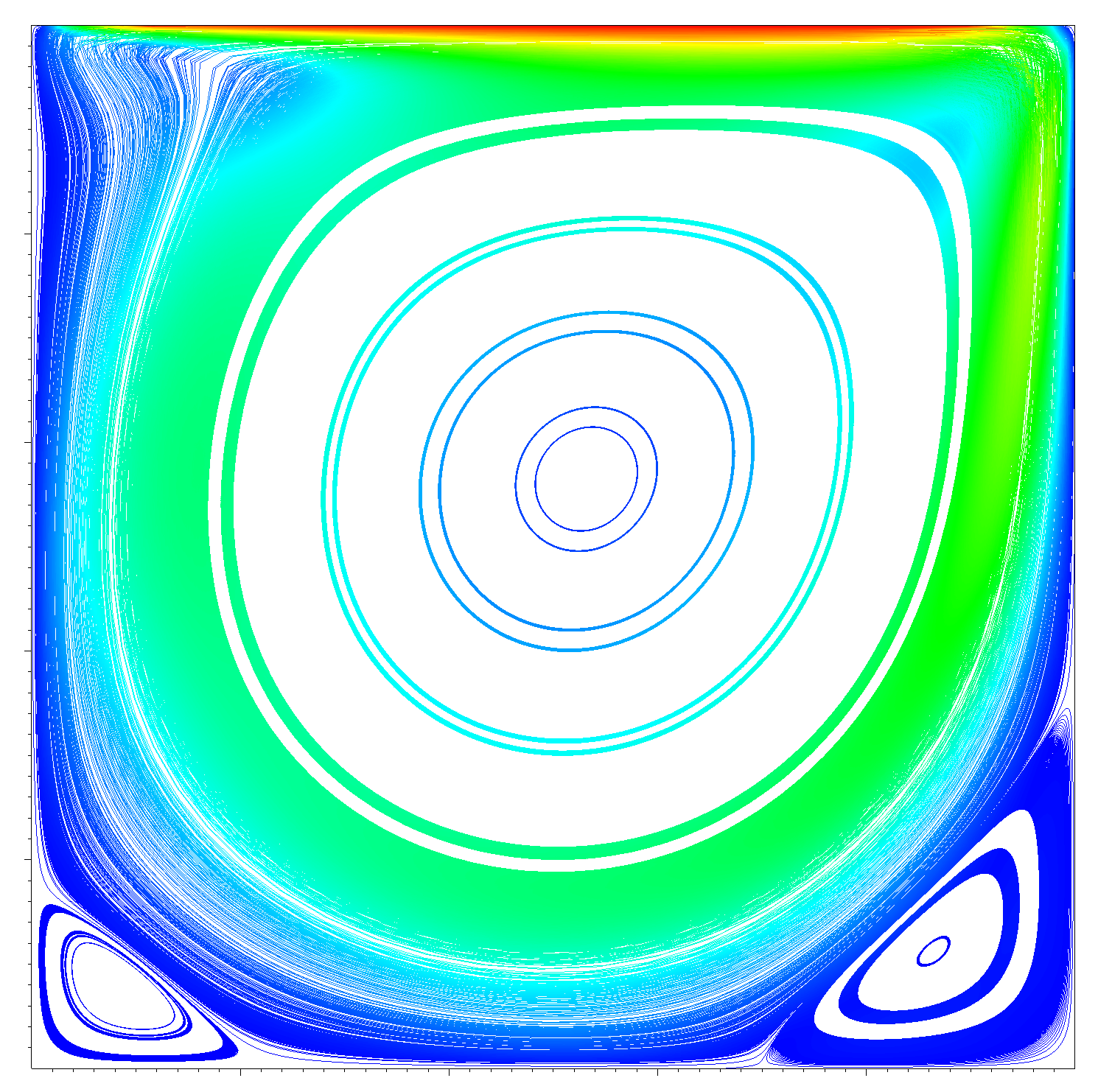}
        \caption{$Re = 1000$}
        \label{fig:streamline_1000}
    \end{subfigure}
    \begin{subfigure}{0.23\textwidth}
        \includegraphics[width = \linewidth]{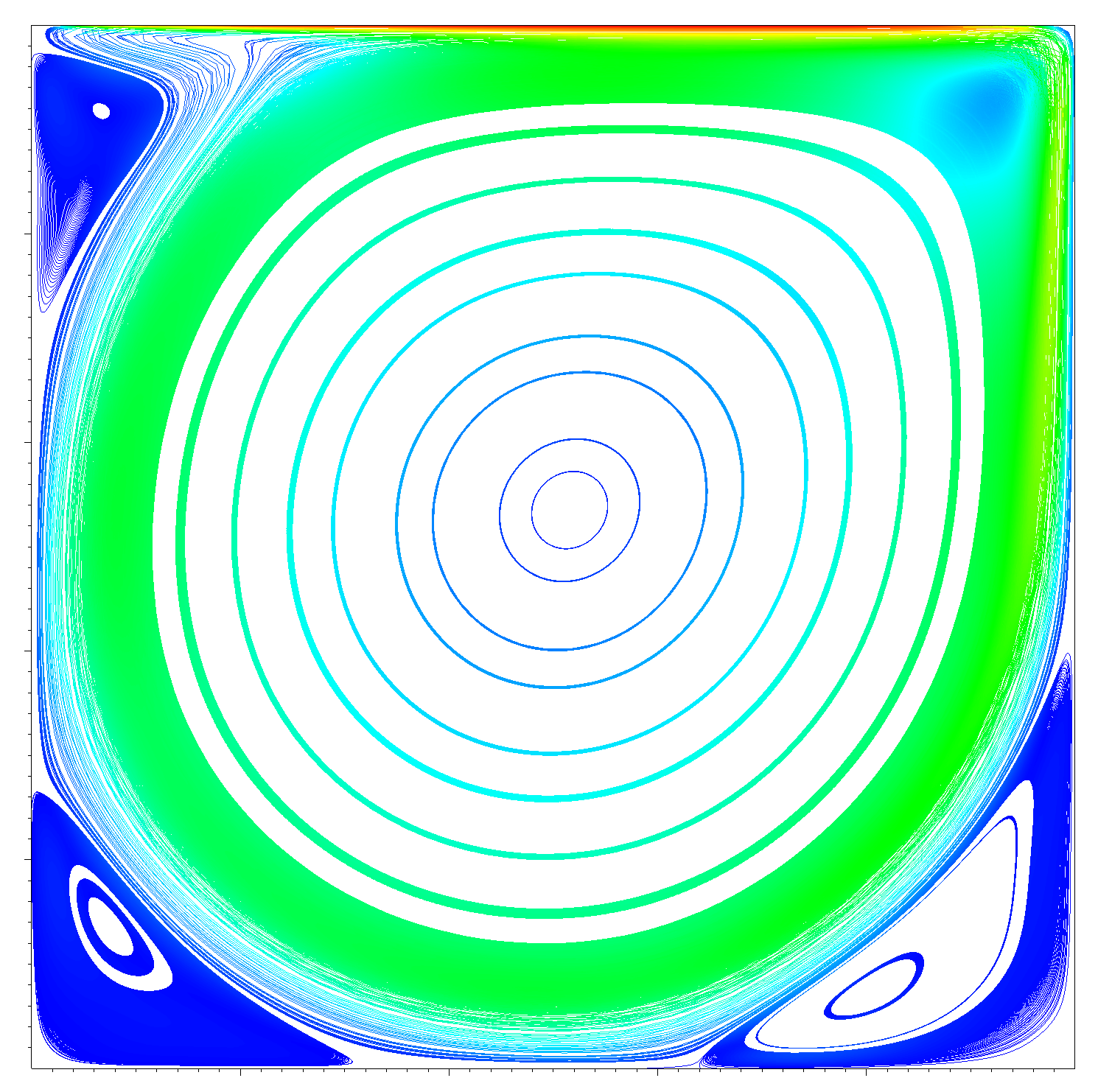}
        \caption{$Re = 5000$}
        \label{fig:streamline_5000}
    \end{subfigure}
    \begin{subfigure}{0.23\textwidth}
        \includegraphics[width = \linewidth]{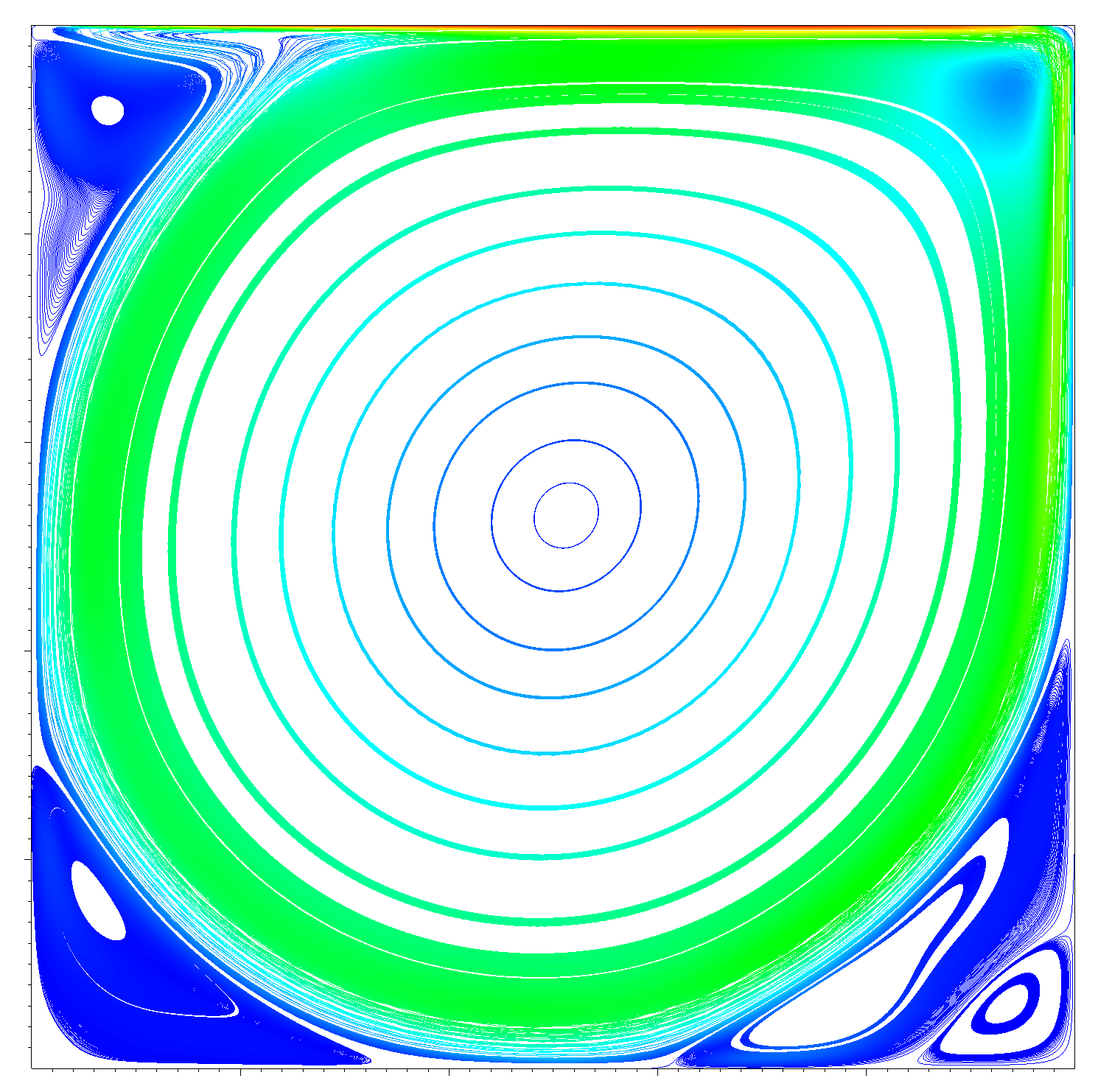}
        \caption{$Re = 10{,}000$}
        \label{fig:streamline_10000}
    \end{subfigure}
    \begin{subfigure}{0.23\textwidth}
        \includegraphics[width = \linewidth]{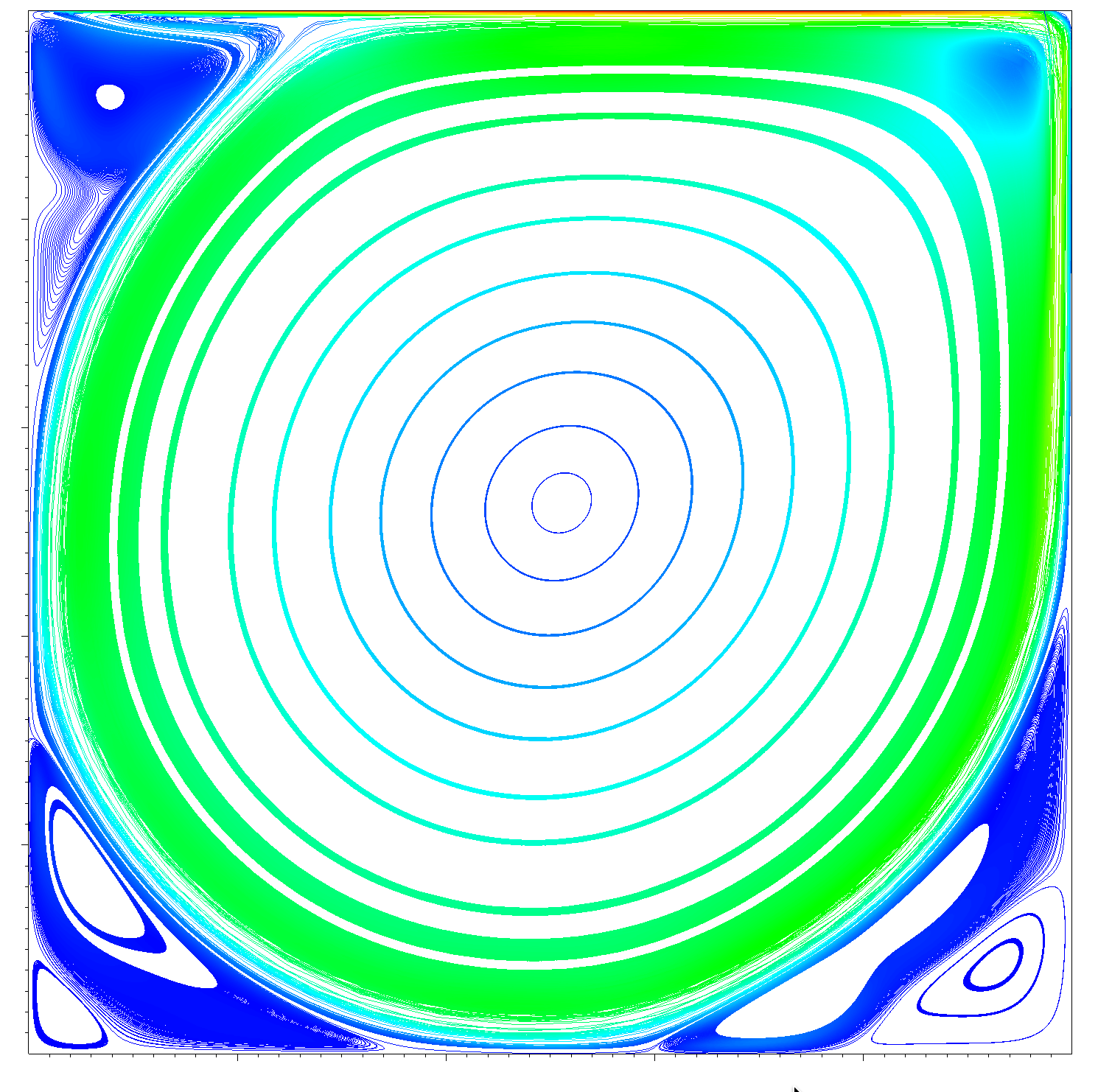}
        \caption{$Re = 15{,}000$}
        \label{fig:streamline_15000}
    \end{subfigure}
    \caption{Streamline patterns for the lid-driven cavity flow.}
    \label{fig:streamline}
\end{figure}

\begin{table}[ht]
    \centering
    \caption{Centerline velocity components at $Re = 1000$ compared with~\cite{Erturk}.}
    \label{lid_driven_re1000}
    \begin{tabular}{rrrrrr}
    \hline
    \multicolumn{1}{c}{$x$} & 
    \multicolumn{1}{c}{~\cite{Erturk}} & 
    \multicolumn{1}{c}{AA-IAH} & 
    \multicolumn{1}{c}{$y$} & 
    \multicolumn{1}{c}{~\cite{Erturk}} & 
    \multicolumn{1}{c}{AA-IAH} \\
    \hline
    1.000 & 0.0000 & 0.00000 & 1.000 & 1.0000 & 1.00000 \\
    0.985 & $-0.0973$ & $-0.09675$ & 0.990 & 0.8486 & 0.84539 \\
    0.970 & $-0.2173$ & $-0.21582$ & 0.980 & 0.7065 & 0.70095 \\
    0.955 & $-0.3400$ & $-0.33739$ & 0.970 & 0.5917 & 0.58553 \\
    0.940 & $-0.4417$ & $-0.43831$ & 0.960 & 0.5102 & 0.50392 \\
    0.925 & $-0.5052$ & $-0.50164$ & 0.950 & 0.4582 & 0.45236 \\
    0.910 & $-0.5263$ & $-0.52305$ & 0.940 & 0.4276 & 0.42239 \\
    0.895 & $-0.5132$ & $-0.51029$ & 0.930 & 0.4101 & 0.40529 \\
    0.880 & $-0.4803$ & $-0.47768$ & 0.920 & 0.3993 & 0.39502 \\
    0.865 & $-0.4407$ & $-0.43831$ & 0.910 & 0.3913 & 0.38728 \\
    0.850 & $-0.4028$ & $-0.40051$ & 0.900 & 0.3838 & 0.38002 \\
    0.500 & 0.0258 & 0.02631 & 0.500 & $-0.0620$ & $-0.06138$ \\
    0.150 & 0.3756 & 0.37196 & 0.200 & $-0.3756$ & $-0.37327$ \\
    0.135 & 0.3705 & 0.36653 & 0.180 & $-0.3869$ & $-0.38419$ \\
    0.120 & 0.3605 & 0.35632 & 0.160 & $-0.3854$ & $-0.38245$ \\
    0.105 & 0.3460 & 0.34168 & 0.140 & $-0.3690$ & $-0.36583$ \\
    0.090 & 0.3273 & 0.32293 & 0.120 & $-0.3381$ & $-0.33498$ \\
    0.075 & 0.3041 & 0.29976 & 0.100 & $-0.2960$ & $-0.29313$ \\
    0.060 & 0.2746 & 0.27042 & 0.080 & $-0.2472$ & $-0.24476$ \\
    0.045 & 0.2349 & 0.23096 & 0.060 & $-0.1951$ & $-0.19309$ \\
    0.030 & 0.1792 & 0.17583 & 0.040 & $-0.1392$ & $-0.13775$ \\
    0.015 & 0.1019 & 0.09982 & 0.020 & $-0.0757$ & $-0.07485$ \\
    0.000 & 0.0000 & 0.00000 & 0.000 & 0.0000 & 0.00000 \\
    \hline
    \end{tabular}%
\end{table}

\begin{table}[ht]
    \centering
    \caption{Centerline velocity components at $Re = 15{,}000$ compared with~\cite{Erturk}.}
    \label{lid_driven_re15000}
    \begin{tabular}{rrrrrr}
    \hline
    \multicolumn{1}{c}{$x$} & 
    \multicolumn{1}{c}{~\cite{Erturk}} & 
    \multicolumn{1}{c}{AA-IAH} & 
    \multicolumn{1}{c}{$y$} & 
    \multicolumn{1}{c}{~\cite{Erturk}} & 
    \multicolumn{1}{c}{AA-IAH} \\
    \hline
    1.000 & 0.0000 & 0.00000 & 1.000 & 1.0000 & 1.00000 \\
    0.985 & $-0.4041$ & $-0.41905$ & 0.990 & 0.5358 & 0.54968 \\
    0.970 & $-0.5593$ & $-0.56610$ & 0.980 & 0.4850 & 0.49420 \\
    0.955 & $-0.4754$ & $-0.48403$ & 0.970 & 0.4969 & 0.50915 \\
    0.940 & $-0.4505$ & $-0.46086$ & 0.960 & 0.4937 & 0.50589 \\
    0.925 & $-0.4361$ & $-0.44636$ & 0.950 & 0.4811 & 0.49182 \\
    0.910 & $-0.4186$ & $-0.42887$ & 0.940 & 0.4653 & 0.47530 \\
    0.895 & $-0.4005$ & $-0.41063$ & 0.930 & 0.4492 & 0.45992 \\
    0.880 & $-0.3828$ & $-0.39209$ & 0.920 & 0.4338 & 0.44480 \\
    0.865 & $-0.3654$ & $-0.37431$ & 0.910 & 0.4190 & 0.43007 \\
    0.850 & $-0.3481$ & $-0.35721$ & 0.900 & 0.4047 & 0.41518 \\
    0.500 & 0.0074 & 0.00669 & 0.500 & $-0.0247$ & $-0.02614$ \\
    0.150 & 0.3483 & 0.35665 & 0.200 & $-0.2942$ & $-0.30302$ \\
    0.135 & 0.3641 & 0.37360 & 0.180 & $-0.3119$ & $-0.32103$ \\
    0.120 & 0.3801 & 0.38959 & 0.160 & $-0.3297$ & $-0.33887$ \\
    0.105 & 0.3964 & 0.40580 & 0.140 & $-0.3474$ & $-0.35704$ \\
    0.090 & 0.4132 & 0.42269 & 0.120 & $-0.3652$ & $-0.37498$ \\
    0.075 & 0.4323 & 0.44262 & 0.100 & $-0.3827$ & $-0.39252$ \\
    0.060 & 0.4529 & 0.46303 & 0.080 & $-0.4001$ & $-0.41006$ \\
    0.045 & 0.4580 & 0.46835 & 0.060 & $-0.4286$ & $-0.43976$ \\
    0.030 & 0.4152 & 0.42146 & 0.040 & $-0.4474$ & $-0.45651$ \\
    0.015 & 0.3083 & 0.31293 & 0.020 & $-0.3278$ & $-0.32766$ \\
    0.000 & 0.0000 & 0.00000 & 0.000 & 0.0000 & 0.00000 \\
    \hline
    \end{tabular}%
\end{table}

\begin{figure}[ht]
    \centering
    \begin{subfigure}[b]{0.48\textwidth}
        \centering
        Centerline Comparison $u$-velocity \\[1ex]
        \includegraphics[trim={50 0 0 0}, clip, width=\textwidth]{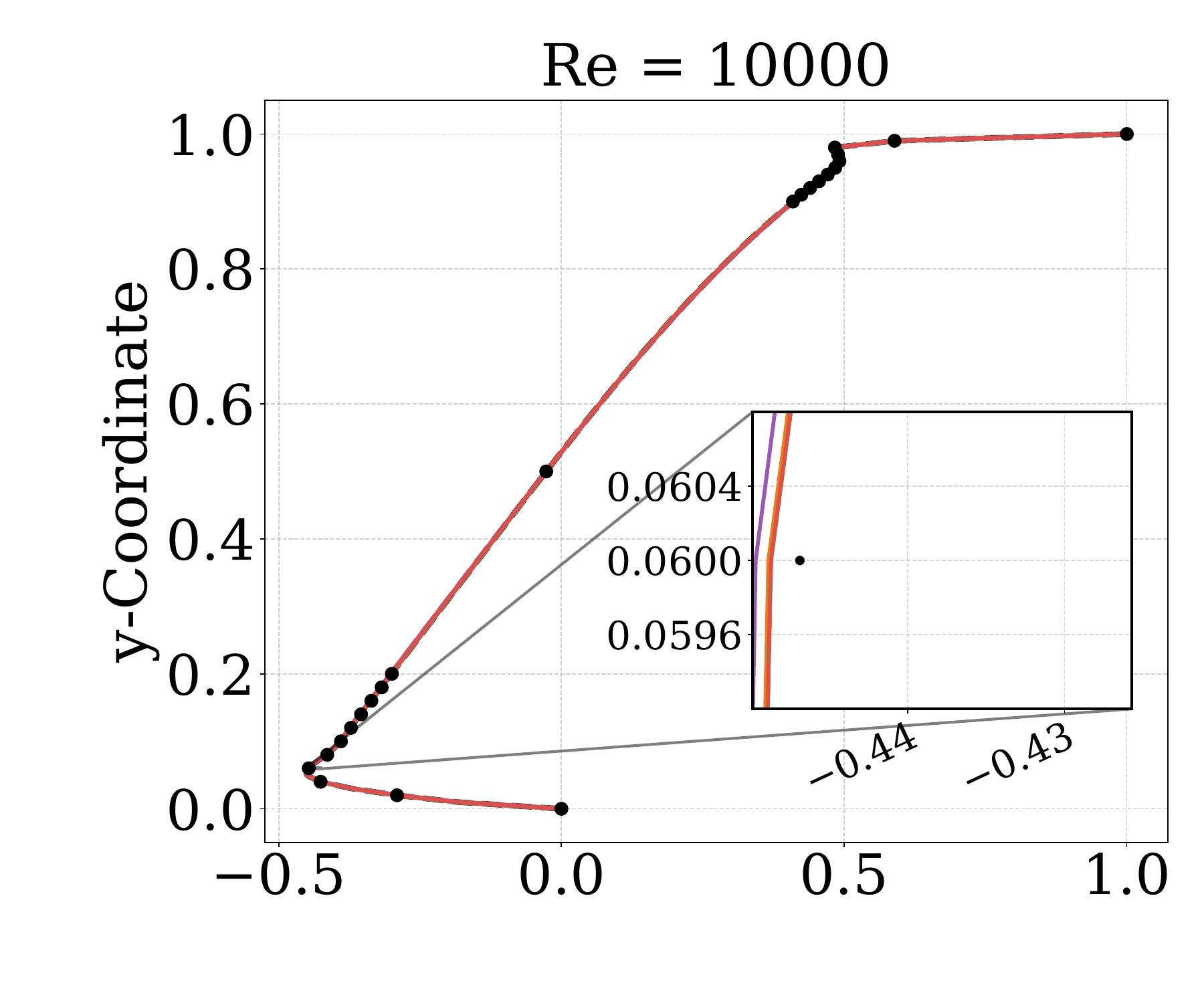}
    \end{subfigure}
    \hfill
    \begin{subfigure}[b]{0.4485\textwidth}
        \centering
        Centerline Comparison $v$-velocity \\[1ex]
        \includegraphics[trim={100 0 0 0}, clip, width=\textwidth]{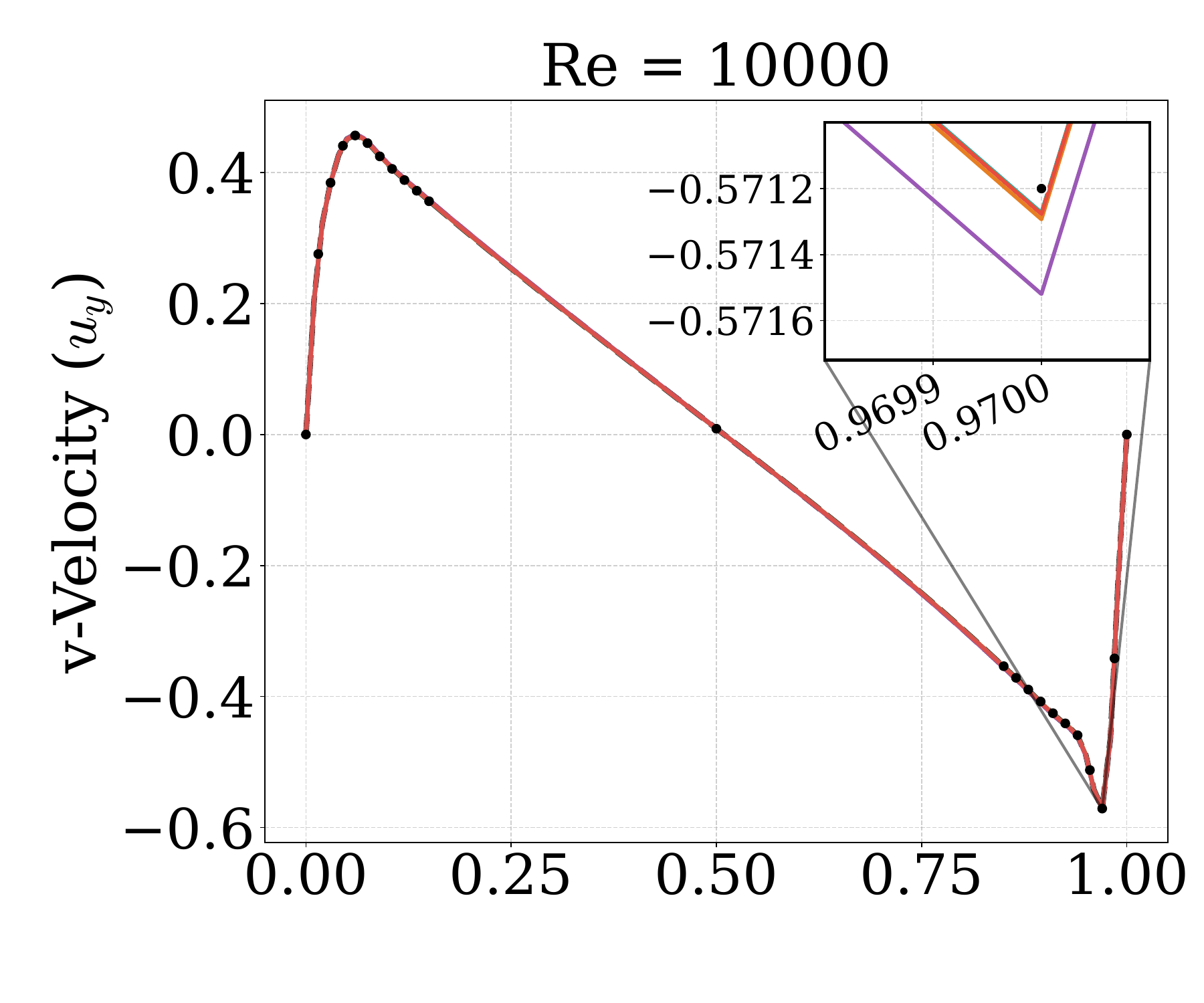}
    \end{subfigure}
    
    \vspace{-1.5em}

    \begin{subfigure}[b]{0.48\textwidth}
        \centering
        \includegraphics[trim={50 0 0 0}, clip, width=\textwidth]{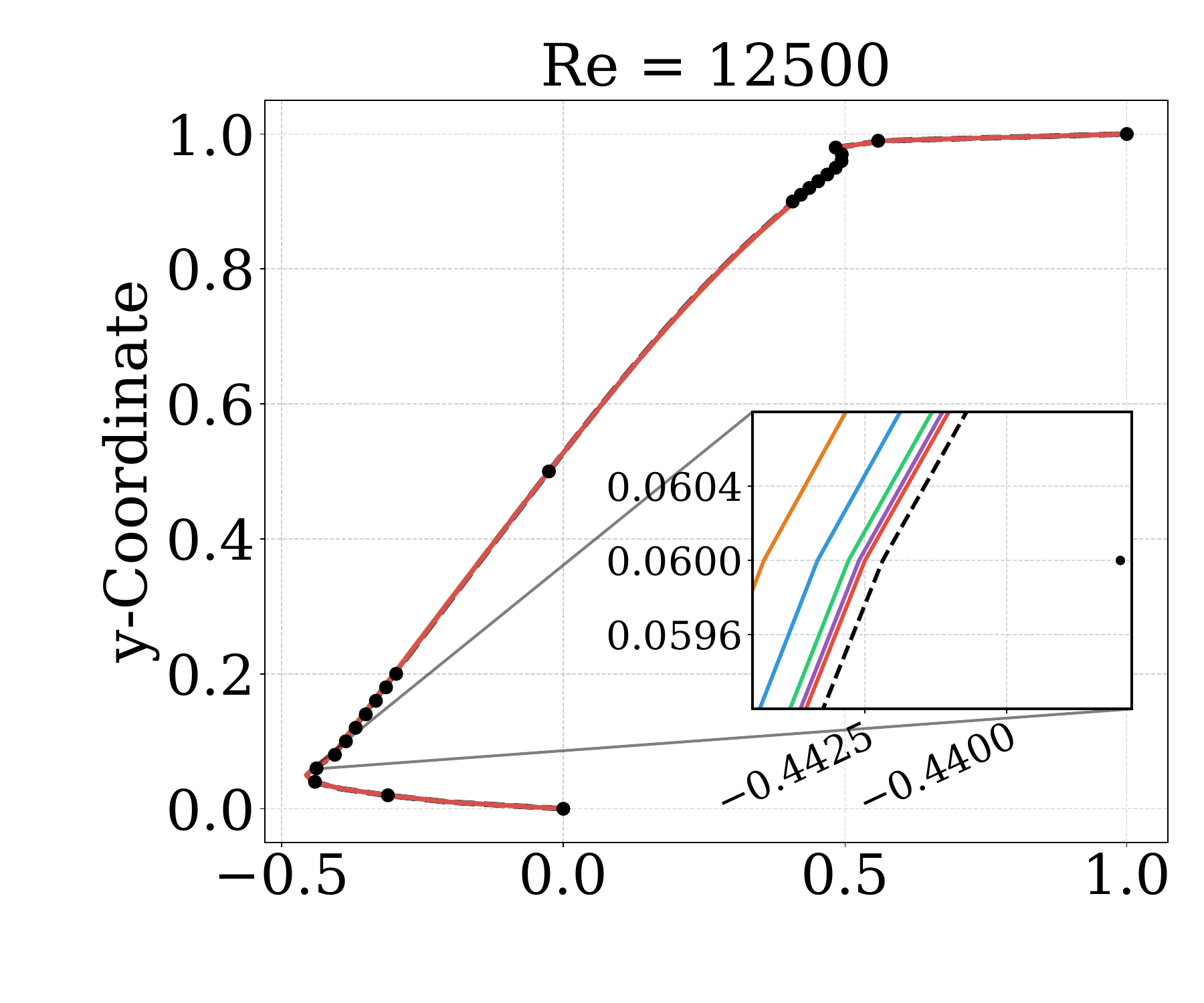}
    \end{subfigure}
    \hfill
    \begin{subfigure}[b]{0.4485\textwidth}
        \centering
        \includegraphics[trim={100 0 0 0}, clip, width=\textwidth]{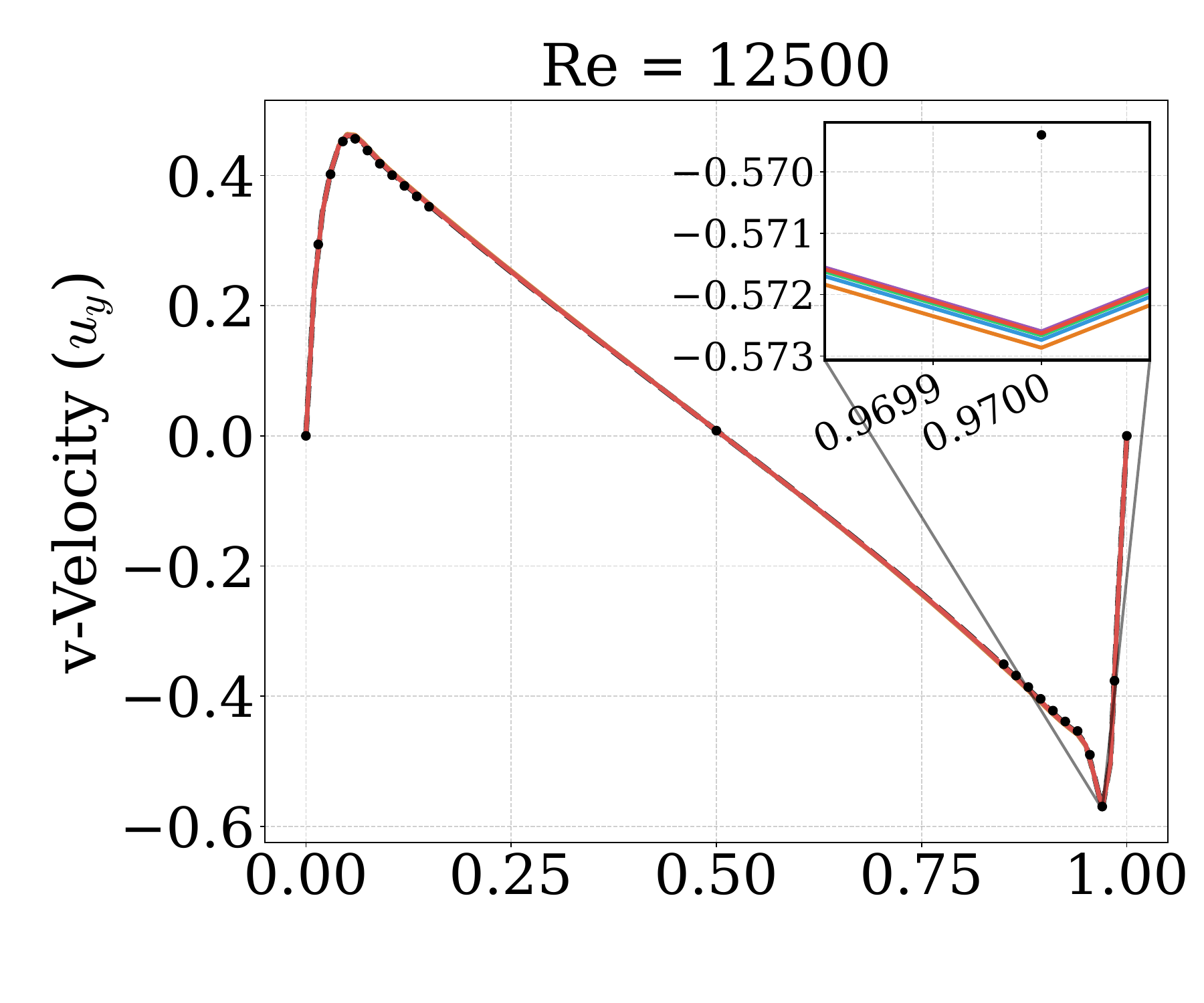}
    \end{subfigure}
    
    \vspace{-1.5em}
    
    \begin{subfigure}[b]{0.48\textwidth}
        \centering
        \includegraphics[trim={50 0 0 0}, clip, width=\textwidth]{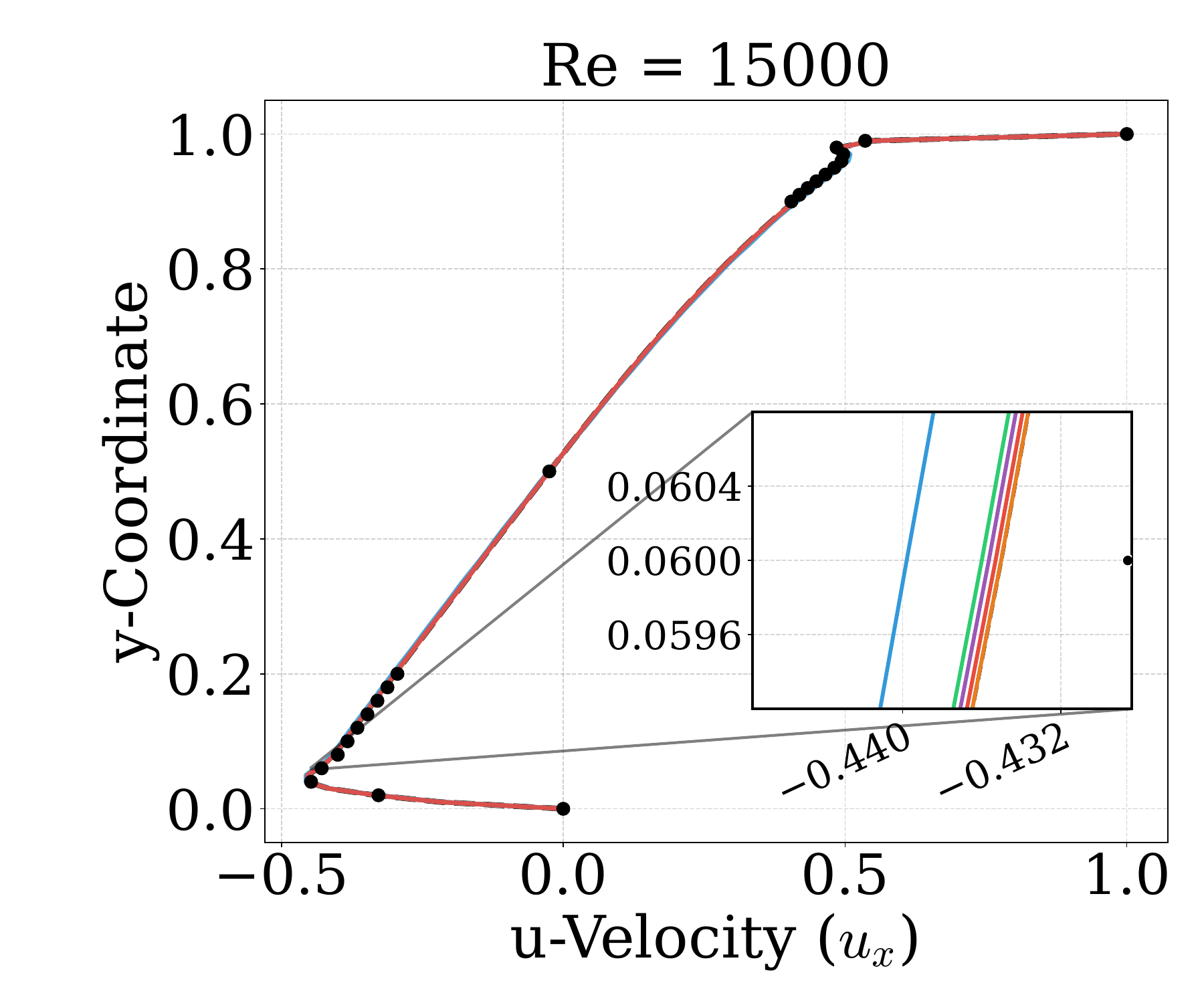}
    \end{subfigure}
    \hfill
    \begin{subfigure}[b]{0.4485\textwidth}
        \centering
        \includegraphics[trim={100 0 0 0}, clip, width=\textwidth]{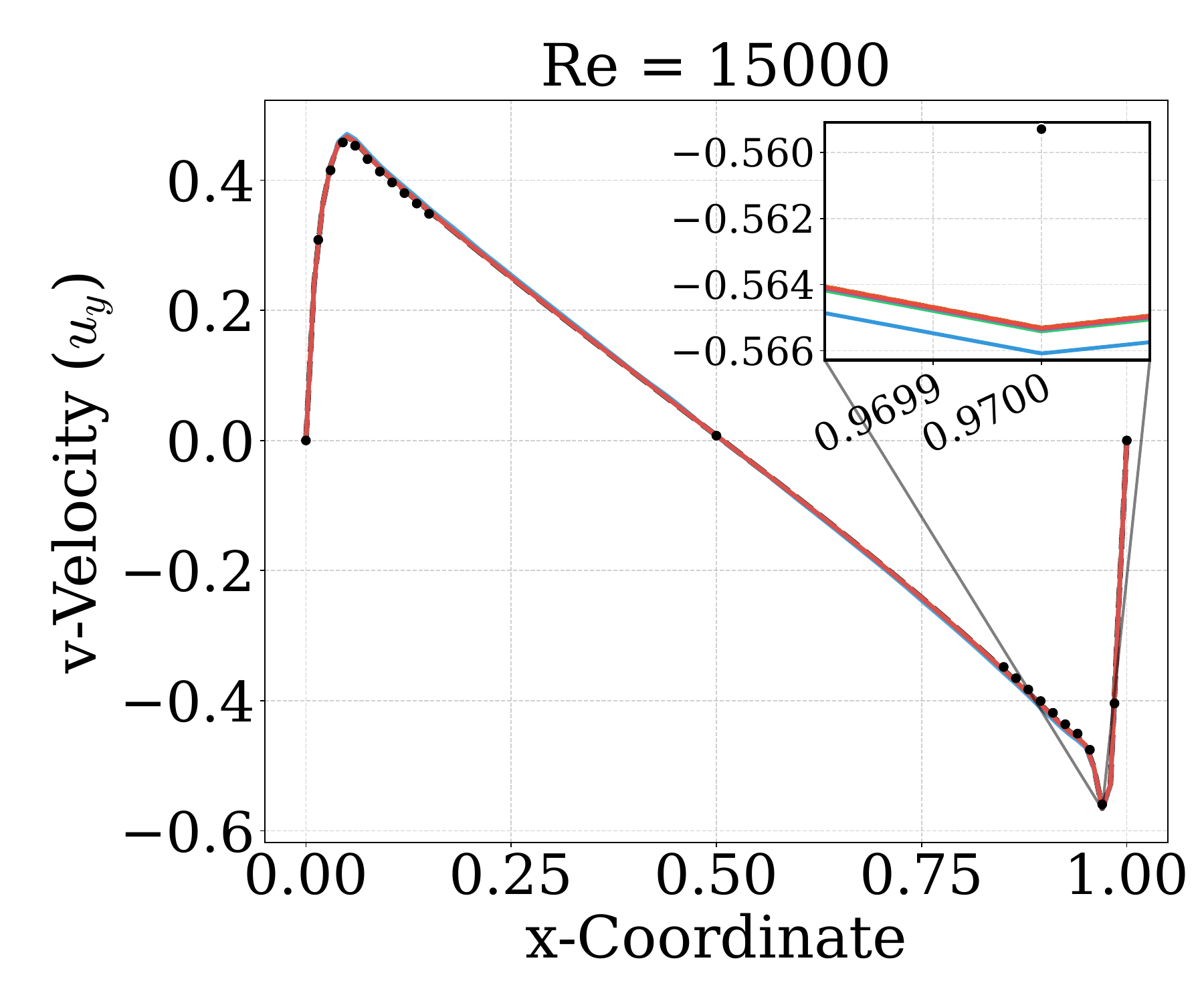}
    \end{subfigure}

    \begin{subfigure}[b]{\textwidth}
        \centering
        \includegraphics[width=\textwidth]{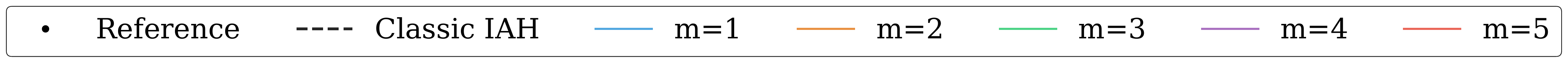}
    \end{subfigure}

    \caption{Comparison of the centerline velocity profiles for different memory depths ($m$).}
    \label{fig:profil}
\end{figure}

\subsection{Channel Flow Over a Full Step}

In this section, the two-dimensional channel flow over a full step problem~\cite{IAH_for_NSE,GEREDELI2023114920} is considered to evaluate the performance of the proposed AA-IAH algorithm. The computational domain consists of a $30 \times 10$ rectangular channel with a $1 \times 1$ obstacle (step) located on the bottom wall, 5 units away from the inlet. The computational mesh is generated with a global refinement level of $N = 4$, which corresponds to a $480 \times 160$ resolution, yielding a total of 76,544 quadrilateral cells and 692,179 degrees of freedom (DoFs), comprising 614,978 DoFs for the velocity field and 77,201 DoFs for the pressure.

To ensure consistency with reference studies in the literature, the kinematic viscosity is set to $\nu = 0.01$ (thus $Re = 100$), and the pressure equation parameter is fixed at $\alpha = 1/\nu = 100$. To investigate its effect on the convergence rate and stability of the algorithm, the parameter $\rho$ is varied as 50, 100, and 200. For each case, the performances of the AA-IAH algorithms with different memory depths ($m$) are compared with the IAH algorithm. The iteration counts and CPU times for the three values of $\rho$ are shown in Figure~\ref{fig:channel_performance}. The corresponding relative iterate-change histories are presented in Figure~\ref{fig:channel_convergence_history}; to avoid visual complexity, only the case $m=4$ is presented for all values of $\rho$. As expected, Anderson acceleration significantly reduces both the number of iterations and the computational time compared to the IAH method.

Figure~\ref{fig:mesh_and_streamline} illustrates the streamlines over the steady-state velocity magnitude. The recirculation zone (eddy) formed immediately behind the step is clearly captured, and the physical behavior of the flow is in full agreement with the literature~\cite{IAH_for_NSE,GEREDELI2023114920}. 

\begin{figure}[ht]
    \begin{subfigure}[b]{\textwidth}
        \centering
        \includegraphics[width = 0.75\textwidth]{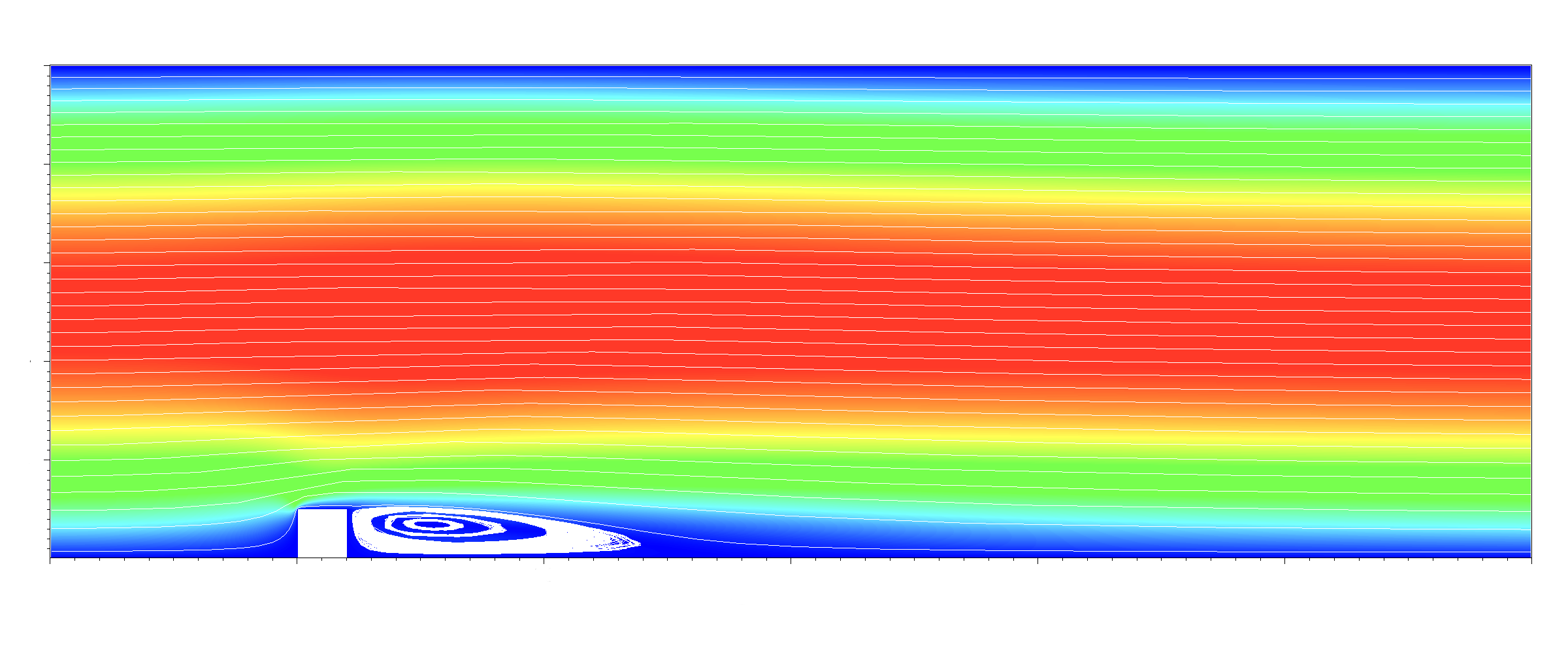}
        \label{fig:channel_streamline}
    \end{subfigure}

    \caption{Steady-state flow streamlines over the velocity magnitude for $\nu = 0.01$.}
    \label{fig:mesh_and_streamline}
\end{figure}

\begin{figure}[ht]
    \centering
    \begin{subfigure}[b]{0.32\textwidth}
        \centering
        \includegraphics[trim={120 25 20 10}, clip, width=\textwidth]{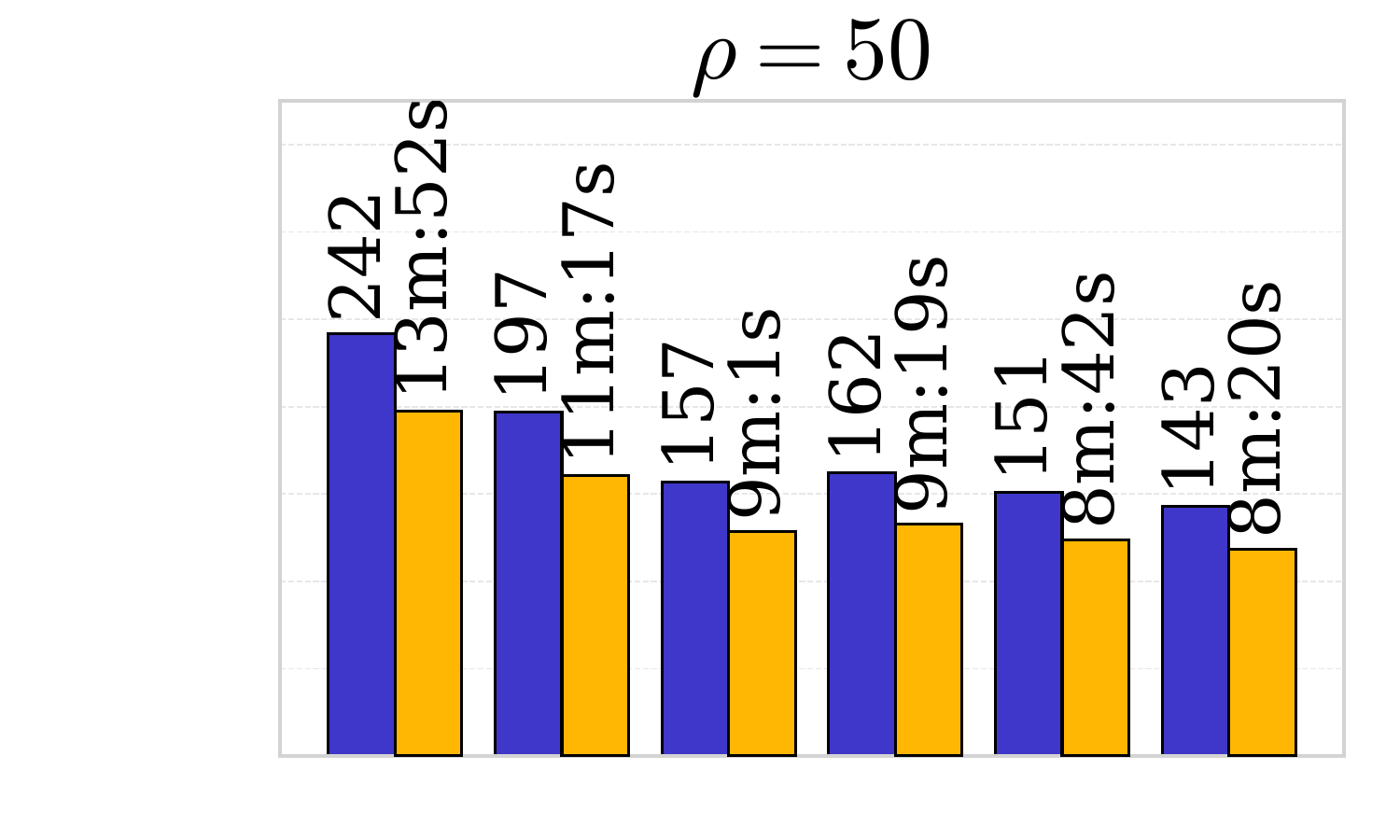}
        \caption{$\rho=50$}
    \end{subfigure}
    \begin{subfigure}[b]{0.32\textwidth}
        \centering
        \includegraphics[trim={120 25 20 10}, clip, width=\textwidth]{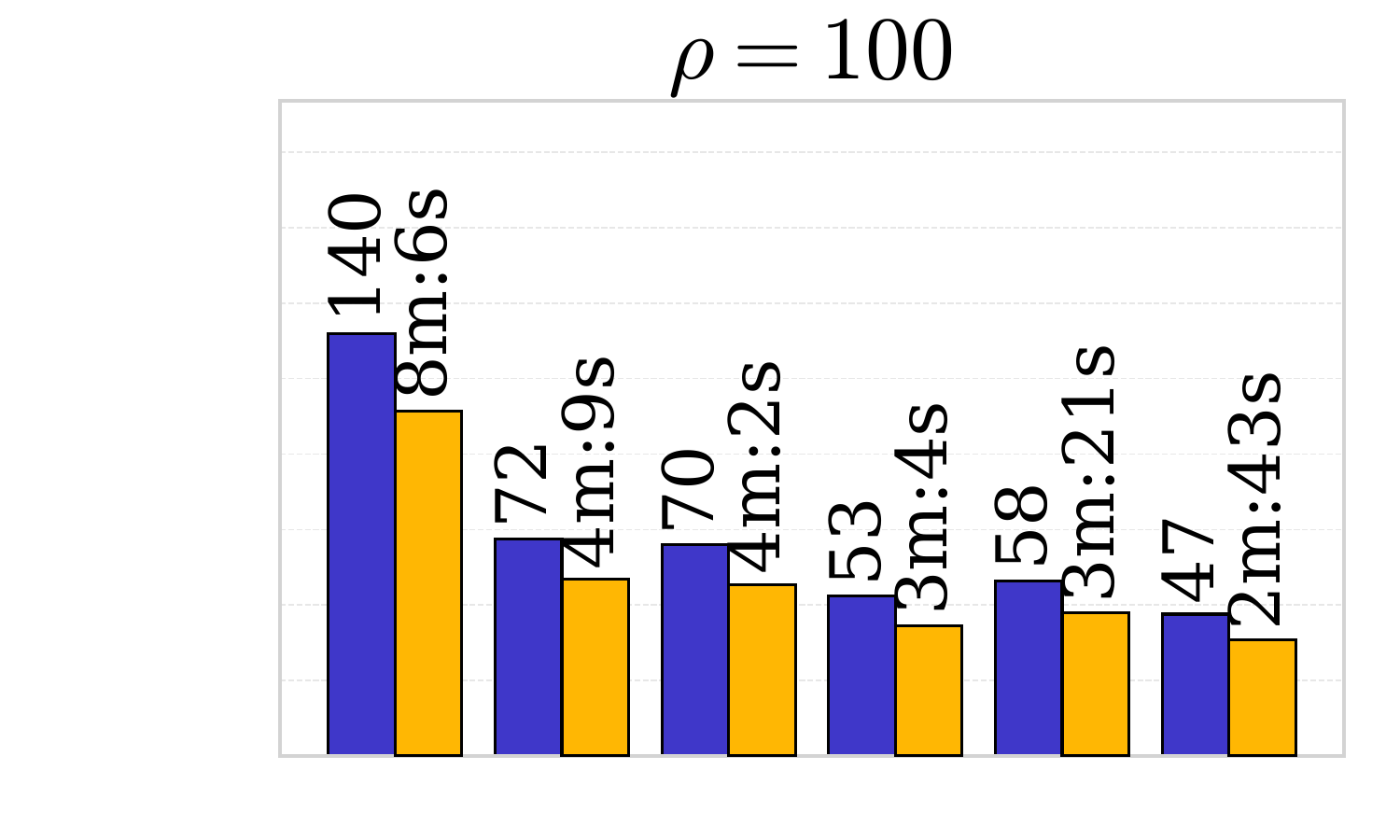}
        \caption{$\rho=100$}
    \end{subfigure}
    \begin{subfigure}[b]{0.32\textwidth}
        \centering
        \includegraphics[trim={120 25 20 10}, clip, width=\textwidth]{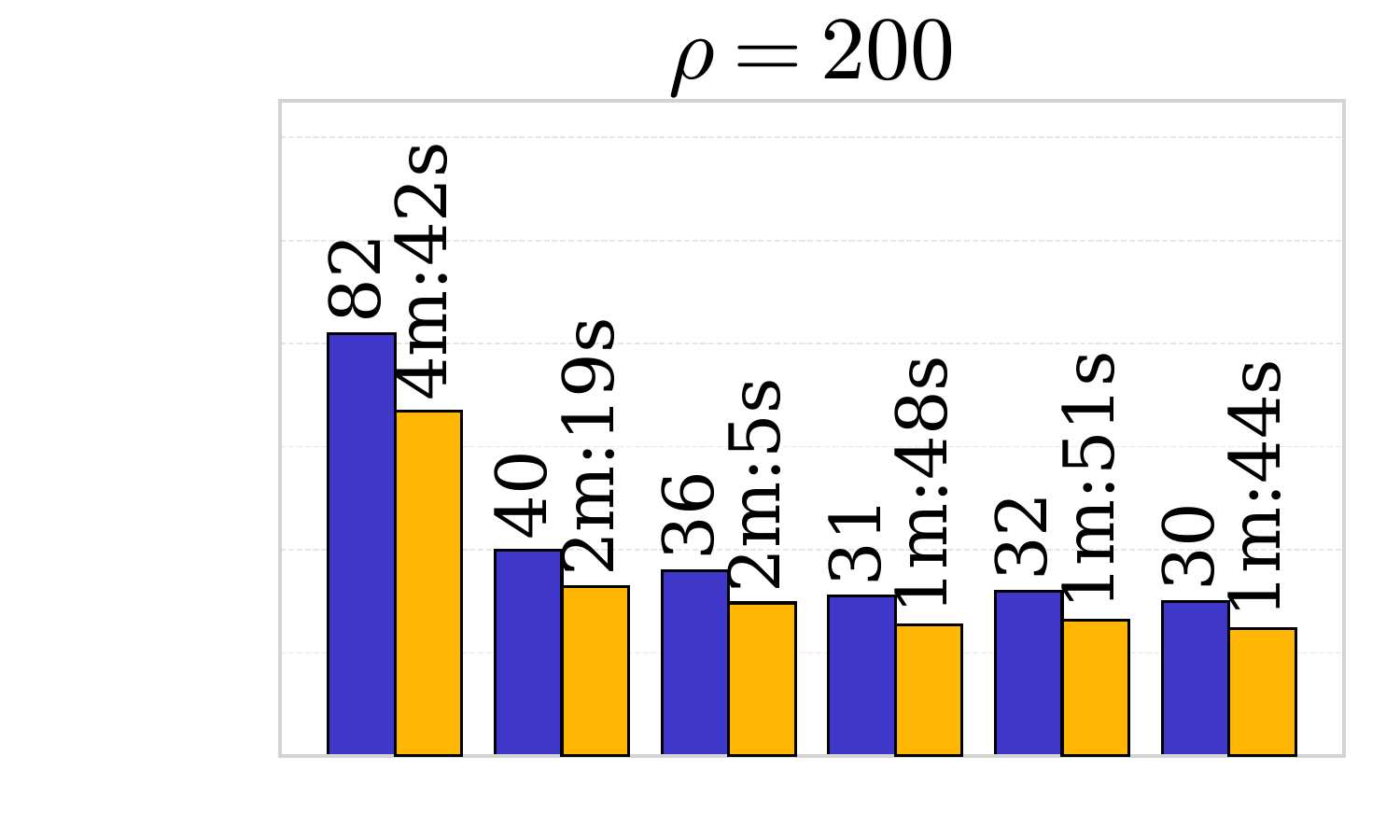}
        \caption{$\rho=200$}
    \end{subfigure}
    \caption{Performance comparison of IAH and AA-IAH($m$) algorithms in terms of total iteration counts and CPU times. Leftmost bars represent IAH, while the remaining bars correspond to AA-IAH with $m = 1, 2, 3, 4$ from left to right.}
    \label{fig:channel_performance}
\end{figure}

\begin{figure}[ht]
    \centering
    \begin{subfigure}[b]{0.3525\textwidth}
        \centering
        \includegraphics[trim={20 0 20 0}, clip, width=\textwidth]{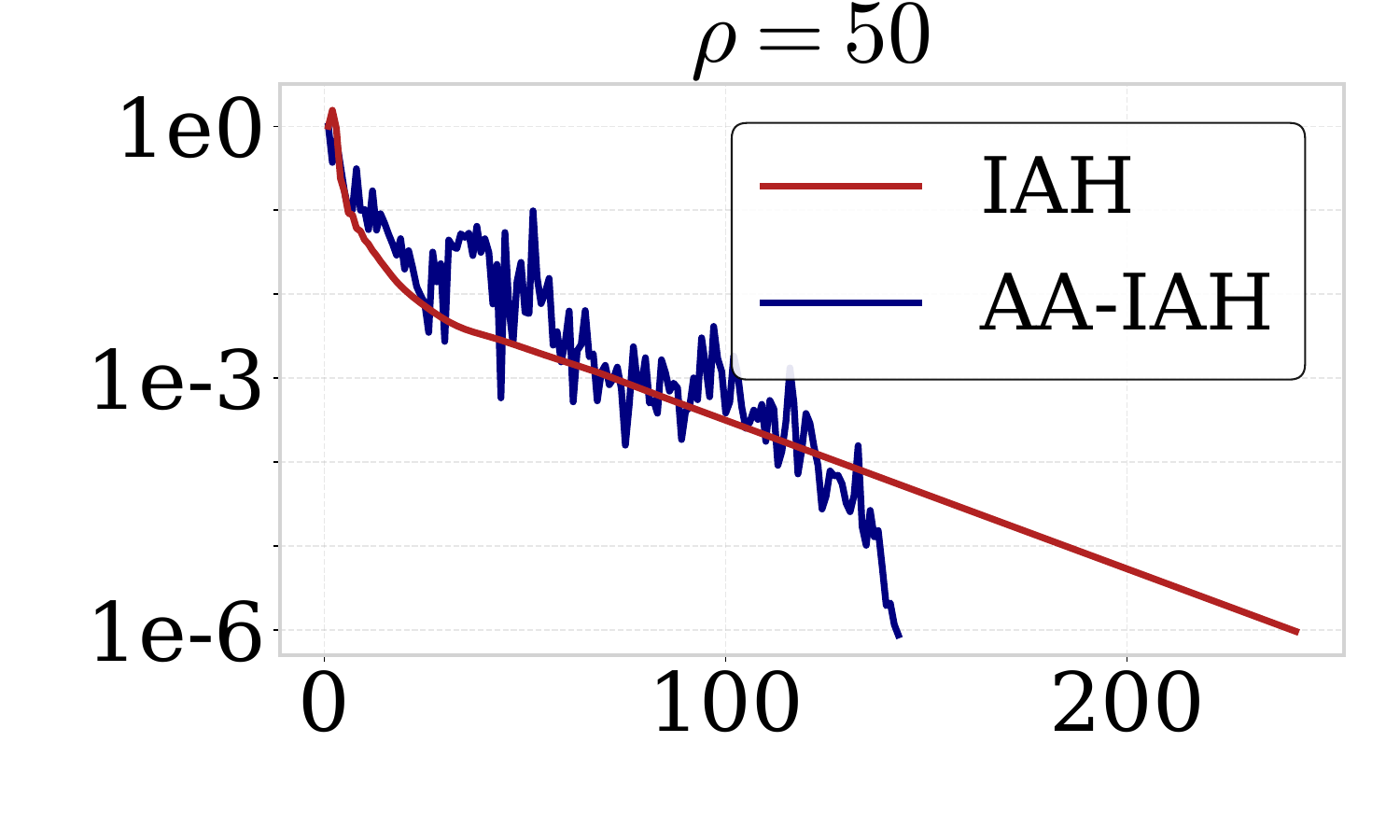}
        \caption{$\rho=50$}
    \end{subfigure}
    \begin{subfigure}[b]{0.30\textwidth}
        \centering
        \includegraphics[trim={120 0 20 0}, clip, width=\textwidth]{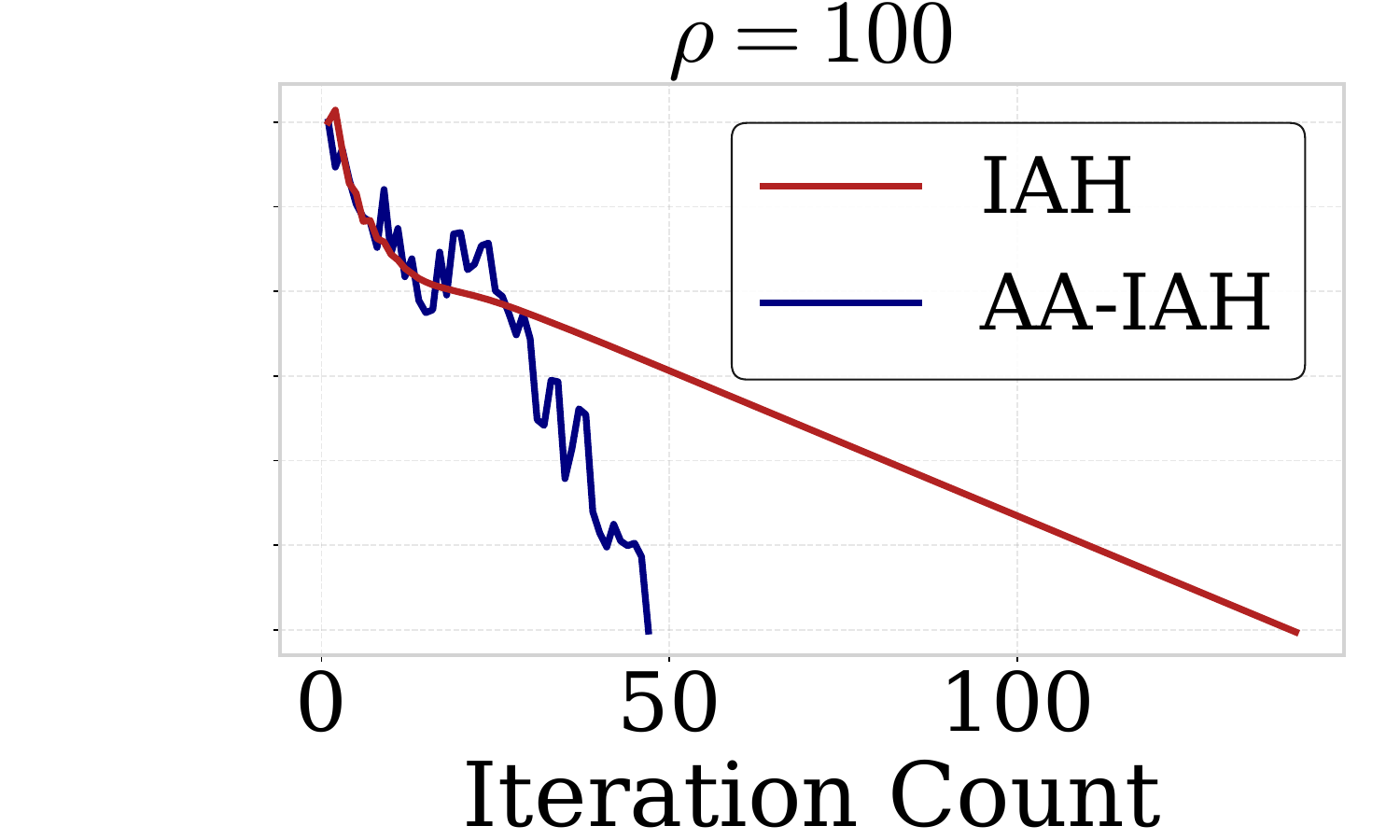}
        \caption{$\rho=100$}
    \end{subfigure}
    \begin{subfigure}[b]{0.30\textwidth}
        \centering
        \includegraphics[trim={120 0 20 0}, clip, width=\textwidth]{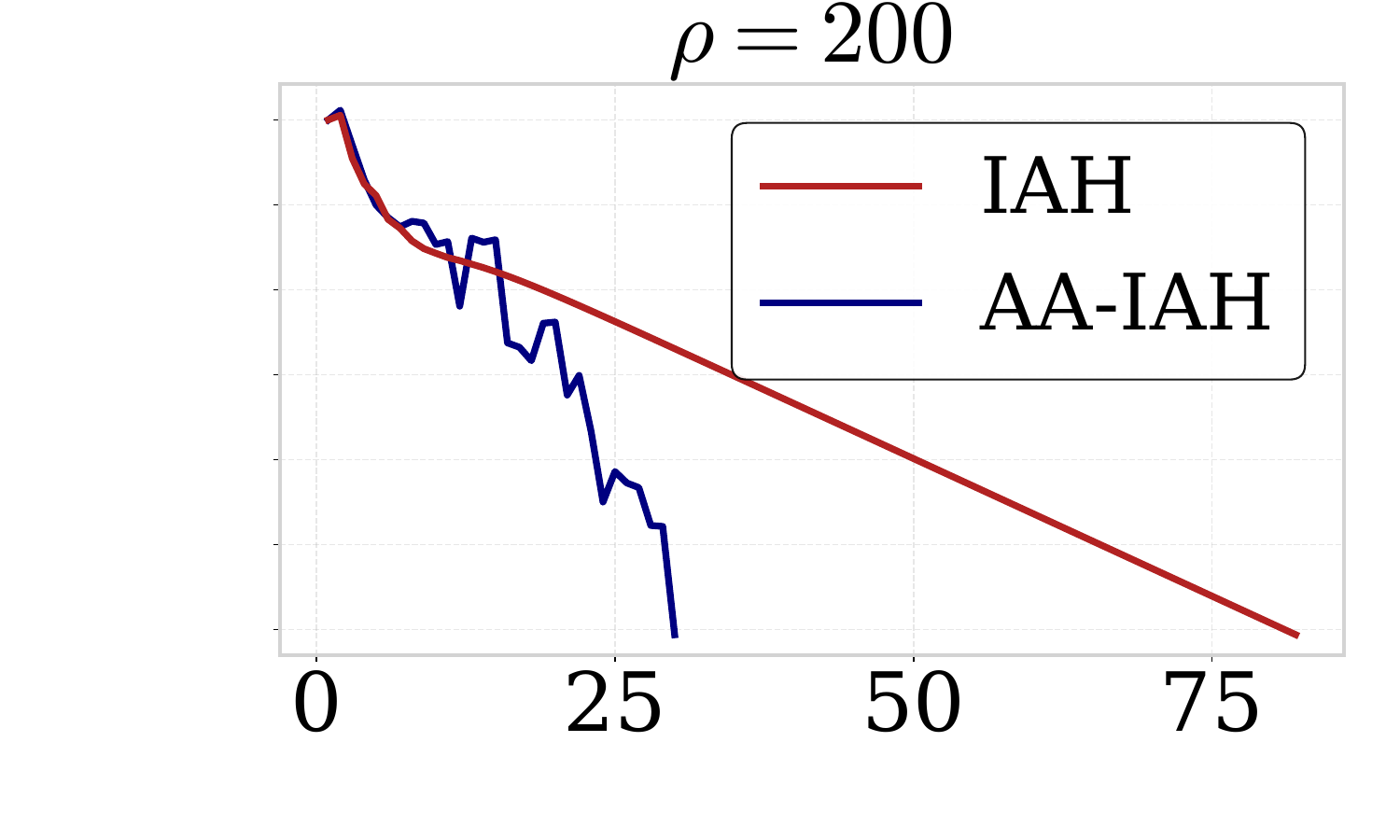}
        \caption{$\rho=200$}
    \end{subfigure}
    \caption{Relative iterate-change history of the IAH and AA-IAH($4$) algorithms.}
    \label{fig:channel_convergence_history}
\end{figure}

\section{Conclusion}
\label{sec:conclusion}

In this study, we presented the Anderson-accelerated improved Arrow--Hurwicz algorithm for solving the steady-state incompressible Navier--Stokes equations and comprehensively compared its performance with the IAH method. Based on the numerical experiments and theoretical analyses conducted, the main outcomes of our work can be summarized as follows:

\begin{itemize}
    \item \textbf{Theoretical Accuracy and Stability:} Using the $Q_2$--$Q_1$ finite element pair for spatial discretization, we verified through manufactured solution tests that the proposed method accurately achieves the expected optimal error convergence rates. Furthermore, by theoretically proving that the underlying operator of our algorithm is well-posed and Lipschitz continuously differentiable, we demonstrated that the method rests on a solid mathematical foundation.

    \item \textbf{High Computational Efficiency:} Compared to the IAH method, integrating Anderson acceleration into the system significantly reduced the number of iterations and thus the CPU time. Despite the additional least squares calculations introduced by the acceleration step, the significant gains in total solution time demonstrated the practical applicability of the method.

    \item \textbf{Robustness in Challenging Flows:} We tested our method at high Reynolds numbers (up to $Re = 15{,}000$) and in a geometry exhibiting flow separation, namely channel flow over a full step. In these tests, AA-IAH remained stable and retained the reported centerline-velocity and manufactured-solution accuracy.
\end{itemize}

In summary, the IAH algorithm enhanced with Anderson acceleration is an efficient and robust alternative for the tested steady incompressible-flow problems.

\section*{Data Availability and Acknowledgements}

The source code and scripts used to generate the numerical results, tables, and figures are available at \url{https://github.com/maggul-research/NSE-IAH-AA}. The numerical data are generated by the included experiment drivers; no external experimental dataset is used.

This article is derived from the master's thesis of the first author. The first author (Sinan Ergen) acknowledges the financial support from the Scientific and Technological Research Council of T\"urkiye (T\"UB\.ITAK) within the scope of the 2210-A National Scholarship Programme for MSc Students.
\bibliography{reference}
\bibliographystyle{plain}
\end{document}